\documentclass[twoside]{amsart}
\usepackage{amsmath,amssymb,amscd,amsthm,amsxtra,amsfonts}
\usepackage{mathtools,mathrsfs,dsfont,xparse}
\usepackage{graphicx,epstopdf}
\usepackage{multirow}
\usepackage{cases}
\usepackage{enumerate}
\usepackage{xcolor}
\usepackage{tikz,pgfplots}
\usetikzlibrary{matrix}
\usepackage{mathtools}
\usepackage[margin=1 in]{geometry}

\mathtoolsset{showonlyrefs}

\allowdisplaybreaks[4]

\date{}

\usepackage{multicol,bbm}
\usepackage[colorlinks=true, pdfstartview=FitV, linkcolor=blue, citecolor=blue, urlcolor=blue]{hyperref}

\numberwithin{equation}{section}

\usepackage{cjhebrew}

\DeclareMathOperator{\re}{Re}
\DeclareMathOperator{\im}{Im}

\usepackage[numbers,sort]{natbib}

{\theoremstyle {definition} \newtheorem {defn} {Definition} [section] }
{\theoremstyle {plain}  \newtheorem {thm} [defn] {Theorem}}
{\theoremstyle {plain}  }
{\theoremstyle {plain} \newtheorem {prop} [defn]{Proposition}}
{\theoremstyle {plain} \newtheorem {lem}[defn] {Lemma}}
{\theoremstyle {definition} \newtheorem {rmk}[defn] {Remark}}
{\theoremstyle {plain} \newtheorem {claim}[defn] {Claim}}

\def\R{{\Bbb{R}}}
\def\C{{\Bbb{C}}}
\def\N{{\Bbb{N}}}
\def\i{{\textbf{i}}}

\DeclareDocumentCommand{\abs}{s m}{
  \operatorname{}
  \IfBooleanTF{#1}{#2}{\left|#2\right|}}

\DeclareDocumentCommand{\norm}{s m}{
  \operatorname{}
  \IfBooleanTF{#1}{#2} {\left\| #2\right\|}}

\DeclareDocumentCommand{\inner}{s m}{
  \operatorname{}
  \IfBooleanTF{#1}{#2}{\left \langle#2\right \rangle}}

\DeclareDocumentCommand{\parenthese}{s m}{
  \operatorname{}
  \IfBooleanTF{#1}{#2}{\left(#2\right)}}

\DeclareDocumentCommand{\square}{s m}{
  \operatorname{}
  \IfBooleanTF{#1} {#2}{\left[#2\right]}}

\DeclareDocumentCommand{\bracket}{s m}{
  \operatorname{}
  \IfBooleanTF{#1}{#2}{\left\{#2\right\}}}

\begin{document}

\author[Sy and Yu]{Mouhamadou Sy$^1$ and Xueying Yu$^2$}

\address{Mouhamadou Sy
\newline 
\indent Department of Mathematics, Imperial College London \indent 
\newline \indent  Huxley Building, London SW7 2AZ, United Kingdom,\indent }
\email{m.sy@imperial.ac.uk}
\thanks{$^1$ The first author's (M.S.) current address is Department of Mathematics, Imperial College London, United Kingdom. This work was written at Department of Mathematics, University of Virginia, Charlottesville, VA}

\address{Xueying  Yu
\newline \indent Department of Mathematics, University of Washington\indent 
\newline \indent  C138 Padelford Hall Box 354350, Seattle, WA 98195,\indent }
\email{xueyingy@uw.edu}
\thanks{$^2$  The second author's (X.Y.) current address is Department of Mathematics, University of Washington, Seattle, WA. This work was written at Department of Mathematics, MIT, Cambridge, MA}

\title[GWP for cubic fractional NLS]{Global well-posedness for the cubic fractional NLS on the unit disk}

\begin{abstract}
In this paper, we show that the cubic nonlinear Schr\"odinger equation with the fractional Laplacian on the unit disk is globally well-posed for certain radial initial data below the energy space and establish a polynomial bound of the global solution. The result is proved by extending the I-method in the fractional nonlinear Schr\"odinger equation setting.

\noindent
\textbf{Keywords}: I-method, global well-posedness, fractional NLS, compact manifold\\

\noindent
\textit{Mathematics Subject Classification (2020):} 35Q55, 35R01, 37K06, 37L50\\
\end{abstract}

\maketitle

\setcounter{tocdepth}{1}
\tableofcontents

\parindent = 10pt     
\parskip = 8pt

\section{Introduction}

We consider the two dimensional defocusing, cubic fractional nonlinear Schr\"odinger equations (FNLS)
\begin{align}\label{fNLS}
\begin{cases}
\i \partial_t u -(-\Delta)^{\alpha} u = |u|^2u, & \alpha \in (0 ,1] ,\\
u(0,x) = \phi(x) , 
\end{cases} 
\end{align}
posed on the unit disk $\Theta = \{x \in \R^2 \, \big| \,  \abs{x} < 1\} $, where $u=u(t,x)$ is a complex-valued function in spacetime $\R \times \Theta$. We assume the  radial symmetry on the initial datum $u_0$ and the Dirichlet boundary condition:
\begin{align*}
u\big|_{\partial \Theta}=0 .
\end{align*}
Note that when $\alpha = 1$, this is the classical nonlinear Schr\"odinger equation (NLS)
\begin{align}
\i \partial_t u + \Delta u = \abs{u}^2 u   \label{NLS} ,
\end{align}
and  for $\alpha \in (0 , 1)$, this is where our main interest located -- fractional nonlinear Schr\"odinger equations \eqref{fNLS}.

Similar as in the NLS setting, the FNLS model conserves its energy and mass in the following forms
\begin{align}
M(u) &=\frac{1}{2}\int_{\Theta} \abs{u}^2 \, dx , \label{Mass} \\
E(u) &= \int_{\Theta} \frac{1}{2} \abs{ \abs{\nabla}^{\alpha} u}^2  +\frac{1}{4} \abs{u}^4 \, dx , \label{Energy}
\end{align}
where $\abs{\nabla} = \sqrt{-\Delta}$. 
Conservation laws above give the control of the $L^2$ and $H^{\alpha}$ norms of the solutions, respectively. Moreover,  the scaling of \eqref{fNLS} is given by 
\begin{align*}
s_c =  1 - \alpha .
\end{align*}

\subsection{Motivation}

In recent decades, there has been of great interest in using fractional Laplacians to model
physical phenomena. The fractional quantum mechanics was introduced by Laskin \cite{laskin} as a generalization of the standard quantum mechanics. This generalization operates on the Feynman path integral formulation by replacing the Brownian motion with a general Levy flight. As a consequence,  one obtains fractional versions of the fundamental Schr\"odinger equation. That means the Laplace operator $(-\Delta)$ arising from the Gaussian kernel used in the standard theory is replaced by its fractional powers $(-\Delta)^{\alpha}$, where $0<\alpha<1$, as such operators naturally generate Levy flights.
The general physical motivation of introducing fractional models is to deal with the so-called anomalous diffusion which arises in many complex systems - in physics, chemistry, biology, social sciences. While the Brownian diffusion, as mathematically studied by Smoluchowski \cite{smolu} and Einstein \cite{einstein}, displays a Gaussian statistics, in an anomalous context the statistics usually follows a power law. This results in heavier distribution queues matching better the long range interaction phenomena.
It turns out that the equation \eqref{fNLS} and its discrete versions are relevant in molecular biology as they have been proposed to describe the charge transport between base pairs in the DNA molecule where typical long range interactions occur (in particular due to intersegmental jumps) \cite{dnadiscrete}. The continuum limit for discrete FNLS was studied rigorously  first  in \cite{KLS} and see also \cite{Gr1, Gr2, HY} for the recent works on the continuum limits.

In this paper, we study the global well-posedness theory{\footnote{With {\it local/gloval well-posedness}, we mean local/global in time existence, uniqueness and continuous dependence of the datum to solution map.}} of FNLS model on the unit disk. The reason why we consider this model is that (1) first, due to the lack of strong dispersion in FNLS (compared to NLS), the global well-posedness theory of FNLS is under development;  (2) second,   the compact manifold  setting is interesting since it allows weaker dispersion than Euclidean spaces (hence less  favorable).   Apart from the challenges behind  weak dispersion mentioned above,  we should be aware of another difficulty  on the unit disk --  the lack of  good Fourier convolution theorem.  This theorem  is  fundamental in analyzing  nonlinearities of the equation, the absence of which  is   due to the different Fourier transform on the unit disk (compared to those in Euclidean spaces) and will cause great difficulties in understanding the nonlinear term in the equation.

Our goal in this paper is to prove the global well-posedness of \eqref{fNLS} with the regularity of the initial datum below the energy space $H^{\alpha}$. Before we present our result, let us first view related works in both NLS and FNLS settings.

\subsection{History and related works}

Let us start from the related works in NLS (that is, $\alpha=1$ in \eqref{fNLS}).  Recall that in Euclidean spaces $\R^d$, the scaling of \eqref{NLS} is given by
\begin{align*}
s_c =  \frac{d}{2}- 1  .
\end{align*}
The problem \eqref{NLS} is called {\it subcritical} when the regularity of the initial datum is smoother than  the scaling $s_c$ of \eqref{NLS}. We will adopt the language in the scaling context in other general manifolds.

In the subcritical  regime ($s> s_c$), it is well-known that the initial value problem \eqref{NLS}   is locally well-posed \cite{Ca}. Thanks to the conservation laws of energy and mass defined in \eqref{Mass} and \eqref{Energy} (with $\alpha =1$), the $H^1$-subcritical initial value problem and the $L^2$-subcritical initial value problem are globally well-posed in the energy space $H^1$ and mass space $L^2$, respectively. In fact, these two initial value problems are also shown to scatter{\footnote{In general terms, with {\it scattering} we intend that the nonlinear solution as time goes to infinity approaches a  linear one.}}. However, one does not expect the scattering phenomenon in compact manifolds.

In the Euclidean space,  the very first global well-posedness result in the subcritical case between the two  (mass and energy) conservation laws  was given by Bourgain in \cite{BourHL}, where he developed the {\it high-low method} to prove  global well-posedness for the cubic NLS in two dimensions for initial data in $H^s, \, s > \frac{3}{5}$.  Bourgain's method is known as {\it high-low method}  since this method consists of estimating separately the evolution of the low frequencies  and of the  high frequencies  of the initial datum. In fact, informally ‘dispersion’ will refer to the fact that the components of the wave packets tend to move at different speeds proportional to the size of the frequencies they are localized on. Hence in the study of dispersive equations (of course FNLS is a member in this family), understanding the behavior of high and low frequencies are very much needed.

To start the high-low method, the initial datum is decomposed into a (smoother) low frequency part and a (rougher) high frequency part. The reason below this cutoff here is that due to  the datum belonging to the space below, its energy is infinity and it is impossible to employ the energy conservation law when  iterating the local solutions up to arbitrarily large time intervals. The choice of the threshold between high and low modes will depend on the regularity of the datum and the arbitrarily large time interval that the solution lives in. After the cutoff, the low frequency part has finite energy, hence its evolution is globally existed.  
Separately, the nonlinear evolution (which is the Duhamel term in the integral equation) to the difference equation for the rougher part has small energy  in an interval of time that is inverse proportional to the size of the low frequency part of the initial datum. Such smallness   of the Duhamel term in the smoother energy space allows one to continue with an iteration by merging this smoother part with the evolution of the low frequency part of the datum. Let us remark that when estimating of the evolution of the rougher part of the  datum, Bourgain used a Fourier transform based space $X^{s,b}$  \cite{BourHL} that captures particularly well the behavior of solutions with low regularity initial datum. At last, let us  also mention that the regularity restriction $s > \frac{3}{5}$ is derived by keeping the accumulation of energy controlled. As a result, in \cite{BourHL} the author obtained a polynomial bound of the sub-energy Sobolev norm of the global solution.

In \cite{CKSTT}, Colliander-Keel-Staffilani-Takaoka-Tao improved the global well-posedness index of the initial datum to $H^s, \, s >\frac{4}{7}$ by introducing a different method, now known as {\it I-method}. Let us  recall  the  {\it I-method} mechanism in \cite{CKSTT} in the next paragraph since it is the method that we will use in this paper. This is also based on an iterative argument.  As we mentioned in the high-low method, in the case of infinite energy solutions, it is hopeless to  make good use of the conserved energy. However, one wishes to make some suitable modification of the energy, so that one  could still have a good control on a substantial portion of the energy (which is known as the almost conservation laws). The modification in the method is to pick up all the low frequencies and certain amount of high frequencies  in the datum. Similarly, the choice of the threshold between high and low frequencies   depend on the regularity of the datum and the arbitrarily large time interval that one iterates  the local solutions up to.

A standard I-method argument consists of three main steps. (1) One first defines a suitable Fourier multiplier  (which is also called `I-operator') that smooths out the initial datum from a rough Sobolev space into the energy space. The  significance of this operator is that it allows one to grab and make use of the energy of the equation in some modified form. (2) Under this modification, the energy of  the smoothed solution is not conserved any more (recall that the energy of the original solution is conserved, but infinity, hence impossible to be employed). Such modification is a trade-off between the  conservation  of infinity energy and its potential to play a role in the argument. With this being said, one can actually   prove that  the energy of the I-operator modified  solution is almost conserved, that is, at each iteration the growth of such modified energy is uniformly small.   (3) Last, iteration of the local well-posedness argument will give a global solution, and the regularity range $s >\frac{4}{7}$ is derived by keeping the accumulation of energy controlled. In fact, such  almost conservation law gives alternative way to measure the growth of the Sobolev norm that the solution lives. As a byproduct, in \cite{CKSTT} the authors obtained a polynomial bound of the sub-energy Sobolev norm of the global solution. 
The cubic NLS in $\R^3$ was also considered in \cite{CKSTT} and the global well-posedness range is given by $s> \frac{5}{6}$.  Later, in \cite{CKSTT2} by combining the Morawetz estimate  with the I-method and a bootstrapping argument, the same authors were able to lower the global well-posedness range to $s > \frac{4}{5}$ and proved, for the first time{\footnote{Actually in \cite{B2} Bourgain proved the global well-posedness  for general data in $H^s , \, s> \frac{11}{13}$ and  scattering for radial data in $H^s, \, s> \frac{5}{7}$.}}, that  the global solution also scatters for data in   $H^s, \, s >\frac{4}{5}$.

The high-low method and I-method have been widely adapted into other dispersive settings and more general manifolds. For instance, \cite{KPVHL} showed the global well-posedness for nonlinear wave equations using the high-low frequency decomposition of Bourgain and \cite{Shen} applied I-method to nonlinear wave equations. \cite{Ha,Zh} studied the global well-posedness of the cubic NLS on closed manifolds without boundary using the I-method. As for the FNLS setting, using the high-low method, \cite{DET} was able to show the global well-posedness for FNLS on the one-dimensional torus.  However, the higher dimensional analogue is still open and challenging. We will, in fact, investigate the  the global behavior of  FNLS in this paper in the higher dimensions. More results on the high-low method and the I-method for NLS can be found in 
\cite{CKSTT3, CR, DPST2, DPST1, D1, D4, D7, D3, FG, Su,Tz, GC} in Euclidean spaces 
and \cite{DPST3, LWX, SY20a} in non-Euclidean settings, 
and for other dispersive models \cite{CKSTT4, CKSTT5, GP, Roy, Wu}. The  high-low method and the I-method are  frequently employed in weak turbulence theory \cite{CKO, Soh, S}, with which the authors proved the growth of higher Sobolev norms. The tools in this paper also can be potentially applied to the weak turbulence theory.

Now let us present the main result in this paper.
\subsection{Main result}
\begin{thm}\label{thm GWP}
The initial value problem \eqref{fNLS} with $\alpha \in (\frac{2}{3} ,1]$ is globally well-posed for radial data $u_0  \in H_{rad}^s (\Theta)$, where
\begin{align*}
s > s_* (\alpha) =  \max \bracket{  \frac{1}{4} \parenthese{\frac{4\alpha^2 - \alpha - 1}{2\alpha-1}  + \sqrt{ \frac{5\alpha^2 -4\alpha +1}{(2\alpha -1)^2}} } , \frac{1}{4} \parenthese{ \frac{\alpha^2 + \alpha -1}{2\alpha -1}  + \sqrt{\frac{\alpha^4 + 10 \alpha^3 -5\alpha^2 - 2\alpha +1}{(2\alpha-1)^2}}} } .
\end{align*}
Moreover, we establish the polynomial bound of the solution 
\begin{align*}
\norm{u(T)}_{H_{rad}^s (\Theta)} \lesssim  T^{\frac{1}{\alpha}(\alpha -s)p} ,
\end{align*}
where the power $p$ above is given by
\begin{align*}
N^p : = \min \{ N^{(\alpha -s)(\frac{2}{\alpha} -4 - \frac{2\alpha+1}{2s-1}) + \alpha - \frac{1}{2}-} , N^{(\alpha -s) (\frac{2}{\alpha} -4) + 3\alpha -2+}  \} .
\end{align*}
\end{thm}

\begin{rmk}
We note that $s_* (\alpha) < \alpha$.
Actually $s_* (\alpha)$ looks very complicated and it might be hard for readers to see the behavior from its expression. Here is a quick plot of $s_*(\alpha)$ and $\alpha$.
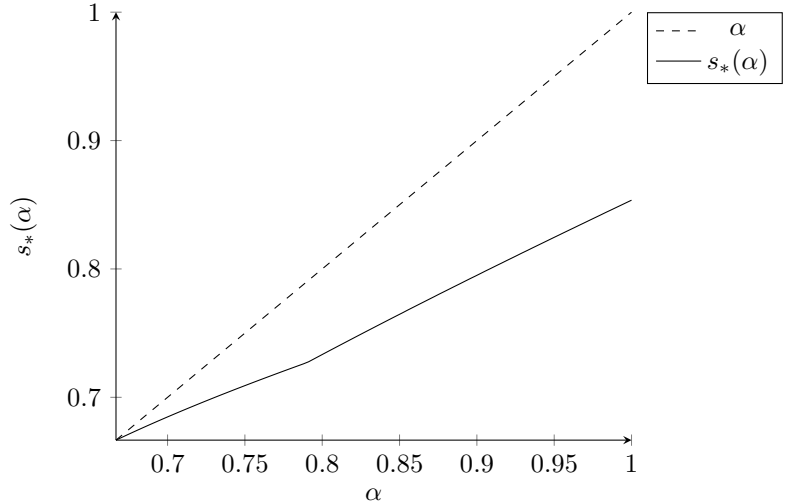
\begin{center}
\begin{tikzpicture}
\begin{axis}[legend pos=outer north east,
    axis lines = left,
    xlabel = $\alpha$,
    ylabel = {$s_*(\alpha)$},
]

\addplot [
	dashed,
    domain=2/3:1, 
    samples=100, 
    ]
    {x};
\addlegendentry{$\alpha$}

\addplot [
    domain=2/3:1, 
    samples=100, 
    ]
    {max( 0.25* ((-4 *x^2 + x + 1)/(1 - 2* x) + sqrt((5 *x^2 - 4 *x + 1)/(1 - 2* x)^2))  ,  (0.25 *(x^2 + x - 1))/(2 *x - 1) + 0.25 *sqrt((x^4 + 10 *x^3 - 5 *x^2 - 2 *x + 1)/(2 *x - 1)^2)  )};
\addlegendentry{$s_*(\alpha)$}

\end{axis}
\end{tikzpicture}
\end{center}
\end{rmk}

\begin{rmk}
Using the bilinear estimates in Section \ref{sec bilinear} and similar strategy in Section \ref{sec LWPI}, it is easy to obtain a local well-posedness argument for FNLS with $\alpha \in (\frac{1}{2} , 1)$. Then thanks to the conservation of energy, it is standard to show the global well-posedness for $\alpha \in (\frac{1}{2} , 1)$ with data in energy space by iterating  the local theory. This is the reason why in the following plot, we can extend our global index in $\alpha \in (\frac{1}{2},\frac{2}{3}]$ trivially and  the global well-posedness curve for $\alpha \in (\frac{1}{2},\frac{2}{3}]$ agrees with the energy line.

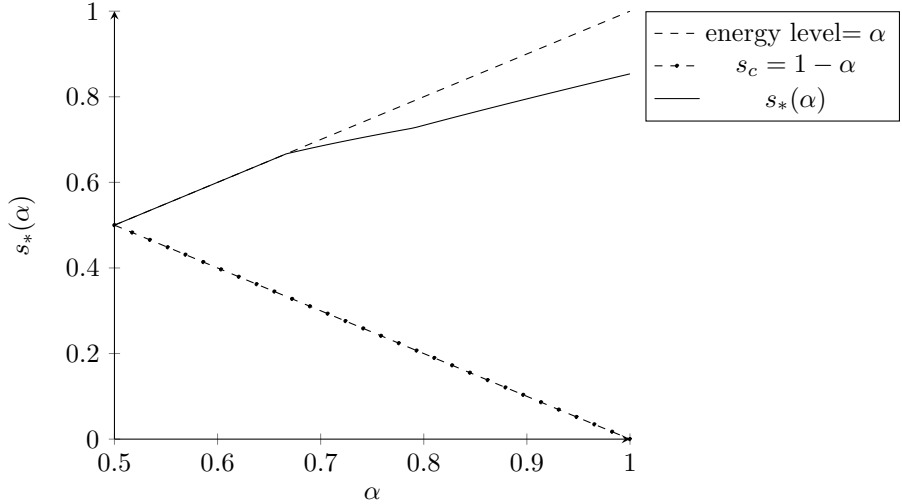
\begin{center}
\begin{tikzpicture}
\begin{axis}[ legend pos=outer north east,
    axis lines = left,
    xlabel = $\alpha$,
    ylabel = {$s_*(\alpha)$},
]

\addplot [
	dashed,
    domain=1/2:1, 
    samples=100, 
    ]
    {x};
\addlegendentry{energy level$=\alpha$}

\addplot [
	dashed, mark=*, mark options={scale=0.3},
    domain=1/2:1, 
    samples=30, 
    ]
    {1-x};
\addlegendentry{$s_c= 1-\alpha$}

\addplot [
    domain=1/2:1, 
    samples=100, 
    ]
    {min(x,  max( 0.25* ((-4 *x^2 + x + 1)/(1 - 2* x) + sqrt((5 *x^2 - 4 *x + 1)/(1 - 2* x)^2))  ,  (0.25 *(x^2 + x - 1))/(2 *x - 1) + 0.25 *sqrt((x^4 + 10 *x^3 - 5 *x^2 - 2 *x + 1)/(2 *x - 1)^2)  ))};
\addlegendentry{$s_*(\alpha)$}

\end{axis}
\end{tikzpicture}
\end{center}

\end{rmk}

\subsubsection{Discussion on the setting and the difficulties}

\paragraph{$\bullet$ \it Compact manifold}
Compact domains usually allow weaker dispersion than Euclidean spaces do. Mathematically we can observe this  phenomenon (`loss of regularity')  in the Strichartz estimates on the bounded manifolds.  For example in \cite{bss}, the loss of $\frac{1}{p}$ derivatives  was established for the classical NLS posed on the compact Riemannian manifold $\Omega$ with boundary
\begin{align}\label{eq loss of reg}
\norm{e^{\i t \Delta} f }_{L^p ([0,T] ; L^q (\Omega))} \leq C \norm{f}_{H^{\frac{1}{p}} (\Omega)}
\end{align}
for fixed finite $T$, $p > 2$, $q < \infty$ and $\frac{2}{p} + \frac{d}{q} = \frac{d}{2} $. We expect that a similar loss of regularity phenomenon happens in the  FNLS setting. To beat the weaker dispersion caused by the compact domain, we assume the radial symmetry on the initial datum. Under this assumption, we  can  benefit a lot from the decay of the radial Laplace operator. More precisely, the radial eigenfunctions of the Laplace operator $-\Delta$ with Dirichlet boundary conditions behave like 
\begin{align}\label{eq approx ex}
e_n (r) \sim \frac{\cos((n-\frac{1}{4}) \pi r- \frac{\pi}{4})}{\sqrt{r}} .
\end{align}
(where $r = \abs{x}$)  and their associated eigenvalues are $z_n^2 \sim n^2$ (see Subsection \ref{ssec eigen} for more detailed discussion on $e_n$'s and $z_n$'s). Relying on the decay of  $e_n$'s, we are able to derive a bilinear Strichartz estimate for a product of two functions that are localized in high and low frequencies respectively. The benefit  of the  bilinear Strichartz estimate is that the `loss of regularity' falls on the term with low frequency instead of on both terms (if naively splitting two functions in the bilinear form into two separate estimates then applying the Strichartz estimates), which  is crucial to  make up for the lack of dispersion.

\paragraph{$\bullet$ \it Absence of Fourier convolution theorem}
Note that the Fourier convolution theorem plays an essential rule in I-method, since the convolution theorem translates the Fourier transform of a product into the convolution of Fourier transforms. Combining this fundamental fact with Littlewood-Paley decomposition, we can interpret  the nonlinear term $\abs{u}^2u$ as the sum of the interaction between functions $u_1, u_2 ,u_3$ with frequencies localized at $\xi_i$  ($i = 1,2,3$) on the Fourier side. For example, let the output frequency of in the nonlinearity $\abs{u}^2 u$ to be $\xi$ and each function $u$ is frequency localized  at $\xi_1, \xi_2 , \xi_3$. This convolution theorem implies that $\xi_1 -\xi_2 + \xi_3 = \xi$, which means that this connection in $\xi_1, \xi_2 , \xi_3$ does not allow  the existence of any extremely huge frequency (compared to the output frequency $\xi$). However, on the unit disk, we lose such control in the highest frequency  due to the absence of Fourier convolution theorem. This  causes great difficulty in summing over the frequencies produced from Littlewood-Paley decomposition back to the original nonlinearity.

Let us mention that in a recent work \cite{SY20a}, where the authors extended the high-low method of Bourgain in the hyperbolic setting. They had similar issue with the convolution theorem, and they managed to recover the smoothing estimate on the Duhamel term via the local smoothing estimate combining the radial Sobolev embedding. However, one does not expect to hold such local smoothing estimates on the compact domains.

Back to our unit disk setting, in order to make up for this absence, let us first take a closer look at the eigenfunctions. In the approximate expression of $e_n$ \eqref{eq approx ex}, we see nothing but trigonometric functions.  This suggests  in some sense the existence of certain type of weak interaction between functions whose frequencies are far from each other. Another hope for us to expect such `convolution' type control is behind the following result. It is shown in  \cite{bgtBil} that in the compact domain without boundary (for example $\mathbb{S}^2$),   the weak interaction functions with separated frequencies. More precisely,  for any $j = 1,2,3$, $ z_{n_j} \ll z_{n_0}$ (recall that $z_n$'s are eigenvalues corresponding to eigenfunctions $e_n$'s), then for every $p >0$ there exists $C_p > 0$ such that for every $w_j \in L^2 (\mathbb{S}^2)$, $j =0,1,2,3$,
\begin{align}\label{eq BGT1}
\abs{\int_{\mathbb{S}^2} P_{n_0} w_0 P_{n_1} w_1 P_{n_2} w_2 P_{n_3} w_3 \, dx} \leq C_p z_{n_0}^{-p} \prod_{j=0}^3 \norm{w_j}_{L^2} .
\end{align}
Note that the factor $z_{n_0}^{-p}$ above  can be understood as the weak interaction in their setting.  Hence to obtain a similar weak interaction in our domain with boundary, we develop Proposition \ref{prop weak}, which essentially captures the features in \eqref{eq BGT1}. That is,
\begin{align}\label{eq BGT2}
\abs{\int_{\R \times \Theta} P_{n_0} w_0 P_{n_1} w_1 P_{n_2} w_2 P_{n_3} w_3 \, dx dt } \lesssim \frac{z_{n_2}^{\frac{3}{2}} z_{n_3}^{\frac{1}{2}}}{z_{n_0}^2} \frac{ 1}{ \inner{   z_{n_0}^{\alpha +} }}    \prod_{j=0}^3 \norm{P_{n_i} w_i}_{X^{0, b}} 
\end{align}
for $z_{n_0} \geq 2 z_{n_1} \geq z_{n_2} \geq z_{n_3}$ (see Subsection \ref{ssec X} for the definition of $X^{s,b}$ norms{\footnote{Roughly speaking, this $X^{s,b}$ norm is defined based on a spacetime Fourier transformation, and is very adapted to the dispersive context as it is constructed upon the underlying dispersive operator. In a perturbative regime (subcritical nonlinearities), the Fourier transform of the solutions is supported around the characteristic surface given by the linear operator, hence the $X^{s,b}$ spaces capture efficiency this clustering. }}). Here the factor $\frac{z_{n_2}^{\frac{3}{2}} z_{n_3}^{\frac{1}{2}}}{z_{n_0}^2} \frac{ 1}{ \inner{   z_{n_0}^{\alpha +} }}  $ serves as a similar role of $z_{n_0}^{-p}$ in \eqref{eq BGT1}. It also should be pointed out that this is the key that allows us to sum up decomposed functions with frequencies greatly separated.

Now let us give the main ideas of the proofs.

\subsubsection{Outline of the proofs}

In this subsection we summarize the main three parts in the proof of the main Theorem \ref{thm GWP}.

In the first part of the proof we present the local theory of the I-operator modified FNLS. In this local theory, as one did in the NLS case, we need a Strichartz-type estimate to run the contraction mapping argument. To this end, we adapt the proof of bilinear estimates for NLS on the unit ball in \cite{an}  in Section \ref{sec bilinear} (see also \cite{SY20b} for the multilinear estimates for NLS on the unit ball). However, it is worth pointing out that due to  the fractionality of the dispersion operator, it is impossible for us to periodize the time in the bilinear estimates and count the integer points on its Fourier characteristic surface. Instead, we have to count the integer points  near the characteristic surface, which results in the local well-posedness index not as good as one  obtained in the NLS setting. As for the proof of the local theory, with the help of the bilinear estimtes in Section \ref{sec bilinear},  we are able to obtain an estimate on the nonlinear term, hence obtain the local  well-posedness  via a standard contraction argument. Let us also mention that since this counting argument does not see the difference in the fractional power $\alpha$ of Laplacian, the local well-posedness  index is in fact uniform for all power $\alpha \in [\frac{1}{2} ,1)$.

Following the I-method mechanism in \cite{CKSTT}, the second part of the proof deals with  the analysis of the energy increment of the modified equation.  A typical strategy to follow is that  one dyadically decomposes all the functions in the change of energy, then proceeds the analysis  in different localized frequency  scenarios,  and in the end sums all the decomposed frequencies back to the original form. In order to sum up all the decomposed functions in frequencies, we require a good control on the highest frequency, whose range is usually governed by the Fourier convolution theorem. However, such nice control in the highest frequency  does not hold on the disk due to different format of eigenfunctions of the radial Dirichlet Laplacian. Hence a different analysis  is needed. Instead of the dyadic decomposition,  we make a finer and  delicate decomposition on frequencies, which allows us to observe a very weak interaction between  functions localized in uncomparable frequencies.  Fortunately, this treatment  fulfills the role of convolution theorem and allows to sum the frequency localized functions in a proper way, which is presented in \eqref{eq BGT2}.

At last, we iterate the local theory obtained in the first part, hence obtain  the global solution. It should be noted that   in this argument, to make the  iteration work, we need to guarantee that the accumulated energy increment does not surpass the size of the initial energy of the modified initial datum, which ensures that the initial setup remains the same in the next iteration. As a byproduct of the method,  one obtains that the global solutions satisfy polynomial-in-time bounds.

\subsection{Organization of the paper}
In Section \ref{sec Preliminaries}, we introduce the notations, eigenfunctions and eigenvalues of the radial Dirichlet Laplacian and the functional spaces with their properties that we will use in this paper.  In Section \ref{sec bilinear}, we prove  bilinear Strichartz estimates, which is an important tool in the proof of the energy increment in Section \ref{sec energy increment}. In Section \ref{sec LWPI}, we first define the I-operator in our setting and  present a local theory based on the I-operator modified equation. In Section \ref{sec weak}, we discuss the weak interaction between functions whose frequencies  are localized far away. Then in Section \ref{sec energy increment}, we compute the energy increment of the modified energy on small time intervals. Finally, in Section \ref{sec gwp}, we show the global well-posedness and establish the polynomial bound for the global solutions in Theorem \ref{thm GWP}.

\subsection*{Acknowledgement} 
X.Y. is funded in part by the Jarve Seed Fund and an AMS-Simons travel grant. Both authors would like to thank Gigliola Staffilani for very insightful comments on a preliminary draft of this paper. The authors are very grateful to the anonymous referees for valuable comments and suggestions.

\section{Preliminaries}\label{sec Preliminaries}
In this section, we first discuss  notations used in the rest of the paper, provide the properties of Bessel functions that will be used in later sections, and recall the behaviors of eigenfunctions and eigenvalues of the radial Dirichlet Laplacian. Then we introduce the function spaces ($H^s$ and $X^{s,b}$ spaces) that we will be working on and list some useful inequalities from harmonic analysis.

\subsection{Notations}
We define
\begin{align*}
\norm{f}_{L_t^q L_x^r (I \times \Theta)} : = \square{\int_I \parenthese{\int_{\Theta} \abs{f(t,x)}^r \, dx}^{\frac{q}{r}} dt}^{\frac{1}{q}},
\end{align*}
where $I$ is a time interval.

For $x\in \R$, we set $\inner{x} = (1 + \abs{x}^2)^{\frac{1}{2}}$. We adopt the usual notation that $A \lesssim  B$ or $B \gtrsim A$ to denote an estimate of the form $A \leq C B$ , for some constant $0 < C < \infty$ depending only on the {\it a priori} fixed constants of the problem. We write $A \sim B$ when both $A \lesssim  B $ and $B \lesssim A$.

\subsection{Bessel functions and their properties}

The Bessel function of order $n$, $J_n(x)$, is defined by
\begin{align*}
J_{n}(x) = \sum_{j=0}^{\infty} \frac{(-1)^j}{j! \, \Gamma (j+n +1)} \parenthese{\frac{x}{2}}^{2j+n}.
\end{align*}
In fact, we will only need  Bessel functions  of order zero and order one, that is,
\begin{align*}
J_0 (x) & = \sum_{j=0}^{\infty} \frac{(-1)^j}{(j!)^2} \parenthese{\frac{x}{2}}^{2j} , \\
J_{1}(x) & = \sum_{j=0}^{\infty} \frac{(-1)^j}{j! (j+1)!} \parenthese{\frac{x}{2}}^{2j+1}.
\end{align*}
Moreover, the derivatives of $J_0(x)$ and $J_1 (x)$ satisfy
\begin{align}
& \frac{d}{dx} J_0 (x) = -J_1 (x),\label{eq dJ0}\\
& \frac{d}{dx} (x J_1 (x)) = x J_0 (x) \label{eq dJ1}.
\end{align}
We also have the following approximation formulas  for $n =0,1$
\begin{align}
&\text{ when }\abs{x} < 1, & J_n(x) & = \frac{1}{n! 2^n}  x^n + \mathcal{O} (x^{n+2}) ,   \label{eq J0}\\
& \text{ when }\abs{x} \geq 1,  & J_n(x) & = \sqrt{\frac{2}{\pi}} \frac{\cos(x- \frac{n \pi}{2} - \frac{ \pi}{4})}{\sqrt{x}} + \mathcal{O} (x^{-\frac{3}{2}}) . \label{eq Jinfty}
\end{align}

\subsection{Eigenfunctions and eigenvalues of the radial Dirichlet Laplacian}\label{ssec eigen}
We denote $e_n (r)$ (where $r = \abs{x}$) to be the eigenfunctions of the radial Laplace operator $-\Delta$ with Dirichlet boundary condition  $\Theta$, and the eigenvalues associated to $e_n$ are $z_n^2$. Both $e_n$'s and $z_n$'s are defined via Bessel functions. 

Recall that $J_0$ is the Bessel function of order zero 
\begin{align}\label{eq J_0}
J_0 (x) = \sqrt{\frac{2}{\pi}} \frac{\cos(x- \frac{\pi}{4})}{\sqrt{x}} + \mathcal{O} (x^{-\frac{3}{2}}) .
\end{align} 
Let $z_n$'s  be the (simple) zeros of $J_0 (x)$ such that $0 < z_1 < z_2 <  \cdots < z_n < \cdots$.  It is known that $z_n$ satisfies
\begin{align}\label{eq z_n}
z_n = \pi (n-\frac{1}{4}) + \mathcal{O} (\frac{1}{n}) .
\end{align}
Also $J_0 (z_n r)$ are eigenfunctions of the Dirichlet self adjoint realization of $-\Delta$, corresponding to eigenvalues $z_n^2$. Moreover any $L^2(\Theta)$ radial function can be expanded with respect to $J_0 (z_n r)$. Let us set 
\begin{align}\label{eq e_n}
e_n : = e_n (r) = \norm{J_0 (z_n \cdot)}_{L^2(\Theta)}^{-1} J_0 (z_n r) .
\end{align}
A direct computation gives 
\begin{align}\label{eq J_0 norm}
\norm{J_0 (z_n \cdot)}_{L^2(\Theta)}^{-1}  \sim z_{n}^{\frac{1}{2}} \sim n^{-\frac{1}{2}},
\end{align}
then combining with \eqref{eq J_0}, \eqref{eq z_n} and \eqref{eq e_n} we have
\begin{align}\label{eq e_n approx}
e_n (r) \sim \frac{\cos((n-\frac{1}{4}) \pi r- \frac{\pi}{4})}{\sqrt{r}} .
\end{align}
In Lemma 2.5 in \cite{AT}, one also has 
\begin{align}\label{eq e_n bdd}
\norm{e_n}_{L_x^p(\Theta)} & \lesssim 
\begin{cases}
1 , & \text{ if } 2 \leq p < 4,\\
\ln (1+n)^{\frac{1}{4}} & \text{ if } p= 4,\\
n^{\frac{1}{2}-\frac{2}{p}} , & \text{ if } p> 4 
\end{cases} 
\end{align}

\subsection{$H_{rad}^s$ spaces}
Recall that $(e_n)_{n=1}^{\infty}$ form an orthonormal bases of the Hilbert space of $L^2$ radial functions on $\Theta$. That is, 
\begin{align*}
\int e_n^2 \, dL = 1 ,
\end{align*}
where $dL = \frac{1}{4\pi} r \, d\theta dr$ is the normalized Lebesgue measure on $\Theta$. 
Therefore, we have the expansion formula for a function $u \in L^2 (\Theta)$, 
\begin{align*}
u=\sum_{n=1}^{\infty} \inner{u , e_n} e_n .
\end{align*}
For $s \in \R$, we define the Sobolev space $H^{s} (\Theta)$ on the closed unit ball $\Theta$ as 
\begin{align*}
H_{rad}^{s} (\Theta) : = \bracket{ u = \sum_{n=1}^{\infty} c_n e_n, \, c_n \in \C : \norm{u}_{H^{s} (\Theta)}^2 = \sum_{n=1}^{\infty} z_n^{2s} \abs{c_n}^2 < \infty } .
\end{align*}
We can equip $H_{rad}^{s} (\Theta)$ with the natural complex Hilbert space structure. In particular, if $s =0$, we denote $H_{rad}^{0} (\Theta)$ by $L_{rad}^2 (\Theta)$. For $\gamma \in \R$, we define the map $\sqrt{-\Delta}^{\gamma}$ acting as isometry from $H_{rad}^{s} (\Theta)$ and $H_{rad}^{s - \gamma} (\Theta)$ by
\begin{align*}
\sqrt{-\Delta}^{\gamma} \parenthese{\sum_{n=1}^{\infty} c_n e_n} = \sum_{n=1}^{\infty} z_n^{\gamma} c_n e_n .
\end{align*}
We denote 
\begin{align*}
S_{\alpha}(t) = e^{- \i t (-\Delta)^{\alpha}}
\end{align*}
the flow of the linear Schr\"odinger equation with Dirichlet boundary conditions on the unit ball $\Theta$, and it can be written into
\begin{align*}
S_{\alpha}(t) \parenthese{\sum_{n=1}^{\infty} c_n e_n} = \sum_{n=1}^{\infty} e^{-\i t z_n^{2 \alpha} } c_n e_n.
\end{align*}

\subsection{$X_{rad}^{s,b}$ spaces}\label{ssec X}
Using again the $L^2$ orthonormal basis of eigenfunctions $\{ e_n\}_{n=1}^{\infty}$ with their eigenvalues $z_n^2$ on $\Theta$, we define the $X^{s,b}$ spaces of functions on $\R \times \Theta$ which are radial with respect to the second argument.
\begin{defn}[$X_{rad}^{s,b}$ spaces]\label{defn Xsb}
For $s \geq 0$ and $b \in \R$,
\begin{align*}
X_{rad}^{s,b} (\R \times \Theta) = \{ u \in \mathcal{S}' (\R , L^2(\Theta)) : \norm{u}_{X_{rad}^{s,b} (\R \times \Theta)} < \infty \} ,
\end{align*}
where 
\begin{align}\label{eq Xsb}
\norm{u}_{X_{rad}^{s,b} (\R \times \Theta)}^2 = \sum_{n=1}^{\infty} \norm{\inner{\tau + z_n^{2\alpha}}^b \inner{z_n}^{s} \widehat{c_n} (\tau) }_{L^2(\R_{\tau} ) }^2 ,
\end{align}
and
\begin{align*}
u(t) = \sum_{n=1}^{\infty} c_n (t) e_n .
\end{align*}
Moreover, for $u \in X_{rad}^{0, \infty} (\Theta) =  \cap_{b \in \R} X_{rad}^{0,b} (\Theta)$ we define, for $s \leq 0$ and $b \in \R$, the norm $\norm{u}_{X_{rad}^{s,b} (\R \times \Theta)}$ by \eqref{eq Xsb}.
\end{defn}
Equivalently, we can write the norm \eqref{eq Xsb} in the definition above into
\begin{align*}
\norm{u}_{X_{rad}^{s,b} (\R \times \Theta)} = \norm{S_{\alpha}(-t) u}_{H_t^b H_x^s (\R \times \Theta)} .
\end{align*}
For $T > 0$, we define the restriction spaces $X_T^{s,b} (\Theta)$ equipped with the natural norm
\begin{align*}
\norm{u}_{X_T^{s,b} ( \Theta)} = \inf \{ \norm{\tilde{u}}_{X_{rad}^{s,b} (\R \times \Theta)} : \tilde{u}\big|_{(-T,T) \times \Theta} =u\} .
\end{align*}

\begin{lem}[Basic properties of $X_{rad}^{s,b}$ spaces]\label{lem X property1}
\begin{enumerate}
\item
We have the trivial nesting 
\begin{align*}
X_{rad}^{s,b} \subset  X_{rad}^{s' , b' } 
\end{align*}
whenever $s' \leq s$ and $b' \leq b$, and
\begin{align*}
X_{T}^{s,b} \subset  X_{T'}^{s,b} 
\end{align*}
whenever $T' \leq T$ .
\item
The $X_{rad}^{s,b}$ spaces interpolate nicely in the $s, b$ indices.
\item
For $b > \frac{1}{2}$, we have the following embedding
\begin{align*}
\norm{u}_{L_t^{\infty} H_x^{s} (\R \times \Theta) } \leq C \norm{u}_{X_{rad}^{s,b} (\R \times \Theta)}.
\end{align*}
\item An embedding that will be used frequently in this paper
\begin{align*}
X^{0, \frac{1}{4}} \hookrightarrow L_t^4 L_x^2 .
\end{align*}
\end{enumerate}
\end{lem}

Note that 
\begin{align*}
\norm{f}_{L_t^{4} L_x^2 } = \norm{S_{\alpha}(t) f}_{L_t^{4} L_x^2} \leq \norm{S_{\alpha}(t) f}_{H_t^{\frac{1}{4}} L_x^2}  = \norm{f}_{X^{0, \frac{1}{4}}} .
\end{align*}

\begin{lem}\label{lem X property2}
Let $b,s >0$ and $u_0 \in H_{rad}^s (\Theta)$. Then there exists $c >0$ such that for $0 < T \leq 1$,
\begin{align*}
\norm{S_{\alpha}(t) u_0}_{X_{rad}^{s,b} ((-T, T) \times \Theta)} \leq c \norm{u_0}_{H_{rad}^s  (\Theta)}.
\end{align*} 
\end{lem}

The proofs of Lemma \ref{lem X property1} and Lemma \ref{lem X property2}  can be found in \cite{an}.

We also recall the following lemma in \cite{BourExp, gi}
\begin{lem}\label{lem Duhamel}
Let $0 < b' < \frac{1}{2}$ and $0 < b < 1-b'$. Then for all $f \in X_\delta^{s, -b'} (\Theta)$, we have the Duhamel term $w(t) = \int_0^t  S_{\alpha}(t-s) f(\tau) \, ds \in X_\delta^{s,b} (\Theta)$ and moreover
\begin{align*}
\norm{w}_{X_T^{s,b} (\Theta)} \leq C T^{1-b-b'} \norm{f}_{X_T^{s,-b'} (\Theta)} .
\end{align*} 
\end{lem}

\subsection{Useful inequalities}
\begin{lem}[Gagliardo-Nirenberg interpolation inequality]\label{lem GN}
Let $1 < p < q \leq \infty$ and $s > 0$ be such that $\frac{1}{q} = \frac{1}{p} - \frac{s\theta}{d}$ for some $0 < \theta = \theta (d , p ,q ,s) < 1$. Then for any $u \in \dot{W}^{s,p} (\R^d)$, we have
\begin{align*}
\norm{u}_{L^q (\R^d)} \lesssim_{d, p, q, s} \norm{u}_{L^p (\R^d)}^{1-\theta} \norm{u}_{\dot{W}^{s,p} (\R^d)}^{\theta} .
\end{align*}
\end{lem}

\begin{lem}[Sobolev embedding]\label{lem Sobolev}
For any $u \in C_0^{\infty} (\R^d)$, $\frac{1}{p} - \frac{1}{q} = \frac{s}{d}$ and $s > 0$, we have 
\begin{align*}
\dot{W}^{s,p} (\R^d) \hookrightarrow L^q (\R^d) .
\end{align*}
\end{lem}

From now on, for simplicity of notation, we write $H^s$ and $X^{s,b}$ for the spaces $H_{rad}^s$  and $X_{rad}^{s,b}$ defined in Section \ref{sec Preliminaries}.

\section{Bilinear Strichartz estimates}\label{sec bilinear}
In this section, we prove the bilinear estimates that will be heavily used in the rest of this paper. The proof is adapted from \cite{an} with two dimensional modification and a different counting lemma.

\subsection{Bilinear Strichartz estimates for FNLS}
\begin{lem}[Bilinear estimates]\label{lem bilinear}
Consider $\alpha \in [\frac{1}{2} ,1)$. For $j  =1,2$, $N_j >0$ and $u_j \in L^2 (\Theta)$ satisfying 
\begin{align*}
\mathbf{1}_{\sqrt{-\Delta} \in [N_j , 2 N_j]} u_j = u_j , 
\end{align*}
we have the following bilinear estimates. 
\begin{enumerate}
\item
The bilinear estimate without derivatives.\\
Without loss of generality, we assume $N_1 \geq N_2 $, then for any $\varepsilon >0 $
\begin{align}\label{eq bilinear1}
\norm{S_{\alpha}(t) u_1  \, S_{\alpha}(t) u_2}_{L_{t,x}^2 ((0,1) \times \Theta)} \lesssim  N_2^{\frac{1}{2} + \varepsilon}   \norm{u_1}_{L_x^2 (\Theta)} \norm{u_2}_{L_x^2 (\Theta)} .
\end{align}

\item
The bilinear estimate with derivatives.\\
Moreover, if $u_j \in H_0^1 (\Theta)$ and assume $N_1 \geq N_2 $, then for any $\varepsilon >0 $
\begin{align}\label{eq bilinear2}
\norm{ \nabla S_{\alpha}(t) u_1  \,  S_{\alpha}(t) u_2}_{L_{t,x}^2 ((0,1) \times \Theta)} \lesssim N_1N_2^{\frac{1}{2} +\varepsilon}  \norm{u_1}_{L_x^2 (\Theta)} \norm{u_2}_{L_x^2 (\Theta)}.
\end{align}
\end{enumerate}
\end{lem}

\begin{rmk}
Lemma \ref{prop bilinear} also holds for $S_{\alpha}(t) u_0  \,  \overline{S_{\alpha}(t) v_0}$. In fact, 
\begin{align*}
\norm{S_{\alpha}(t) u_0 \, S_{\alpha}(t)v_0}_{L_t^2 L_x^2}^2 = \norm{S_{\alpha}(t)u_0  \, S_{\alpha}(t) v_0  \overline{S_{\alpha}(t) u_0}  \, \overline{S_{\alpha}(t)v_0}}_{L_t^1 L_x^1}  = \norm{S_{\alpha}(t) u_0  \,  \overline{S_{\alpha}(t) v_0}}_{L_t^2 L_x^2}^2 .
\end{align*}
\end{rmk}

\begin{prop}[Lemma 2.3 in \cite{bgtBil}: Transfer principle]\label{prop bilinear}
For any $b > \frac{1}{2}$ and for $j =1,2$, $N_j >0$ and $f_j \in X^{0,b} (\R \times \Theta)$ satisfying 
\begin{align*}
\mathbf{1}_{\sqrt{-\Delta} \in [N_j , 2 N_j]} f_j = f_j ,
\end{align*}
one has the following bilinear estimates. 
\begin{enumerate}
\item
The bilinear estimate without derivatives.

Without loss of generality, we assume $N_1 \geq N_2$, then for any $\varepsilon >0 $
\begin{align}\label{eq bilinear1'}
\norm{ f_1 f_2}_{L_{t,x}^2 ((0,1) \times \Theta)} \lesssim  N_2^{\frac{1}{2} + \varepsilon}  \norm{f_1}_{X^{0,b} ((0,1) \times \Theta)} \norm{f_2}_{X^{0,b} ((0,1) \times \Theta)} .
\end{align}

\item
The bilinear estimate with derivatives.

Moreover, if $f_j \in H_0^1 (\Theta)$ and assume $N_1 \geq N_2 $,, then for any $\varepsilon >0 $
\begin{align}\label{eq bilinear2'}
\norm{ \nabla f_1  f_2}_{L_{t,x}^2 ((0,1) \times \Theta)} \lesssim N_1  N_2^{\frac{1}{2} + \varepsilon}  \norm{f_1}_{X^{0,b} ((0,1) \times \Theta)} \norm{f_2}_{X^{0,b} ((0,1) \times \Theta)}.
\end{align}
\end{enumerate}
\end{prop}

\begin{rmk}[Interpolation of bilinear estimates]\label{rmk inter bilinear}
In fact, using H\"older inequality, Bernstein inequality and Lemma \ref{lem X property1}, we write
\begin{align*}
\norm{ f_1 f_2}_{L_{t,x}^2 ((0,1) \times \Theta)} & \lesssim   \norm{f_1}_{L_t^4 L_x^2 ((0,1 ) \times \Theta)} \norm{f_2}_{L_t^4 L_x^{\infty} ((0,1 ) \times \Theta)}  \lesssim  \norm{f_1}_{X^{0, \frac{1}{4}} (\Theta)} N_2 \norm{f_2}_{L_t^4 L_x^{2} ((0,1 ) \times \Theta)} \\
& \lesssim N_2 \norm{f_1}_{X^{0, \frac{1}{4}} ((0,1) \times \Theta)} \norm{f_2}_{X^{0, \frac{1}{4}} ((0,1) \times \Theta)} .
\end{align*}
Then interpolating it with \eqref{eq bilinear1'}, for $b = \frac{1}{2}+$
\begin{align*}
\norm{ f_1 f_2}_{L_{t,x}^2 ((0,1) \times \Theta)} \lesssim  N_2^{\frac{1}{2} + \varepsilon}  \norm{f_1}_{X^{0,b} ((0,1) \times \Theta)} \norm{f_2}_{X^{0,b} ((0,1) \times \Theta)}  ,
\end{align*}
we obtain
\begin{align*}
\norm{ f_1 f_2}_{L_{t,x}^2 ((0,1) \times \Theta)} \lesssim  N_2^{\beta} \norm{f_1}_{X^{0, b(\beta) } ((0,1) \times \Theta)} \norm{f_2}_{X^{0, b(\beta) } ((0,1) \times \Theta)} .
\end{align*}
where $b(\beta) =\frac{1}{4}+ (1-\beta)\frac{1}{2}+$, $\beta \in (\frac{1}{2} , 1]$.
Moreover, if restricting  in the time interval $[0, \delta]$, we have by H\"older inequality
\begin{align}\label{eq inter bilinear}
\norm{ f_1 f_2}_{L_{t,x}^2 ((0,\delta) \times \Theta)} \lesssim  N_2^{\beta} \delta^{2(b- b(\beta) )} \norm{f_1}_{X_{\delta}^{0, b} } \norm{f_2}_{X_{\delta}^{0, b}} .
\end{align}
\end{rmk}

\begin{proof}[Proof of Lemma \ref{lem bilinear}]
First  we write
\begin{align*}
u_1 = \sum_{n_1 \sim N_1} c_{n_1}   e_{n_1}(r) ,  \quad u_2 = \sum_{n_2 \sim N_2} d_{n_2}   e_{n_2}(r) ,
\end{align*}
where $c_{n_1} = \inner{u_1 , e_{n_1}}_{L^2}$ and $d_{n_2} = \inner{u_2 , e_{n_2}}_{L^2}$. Then 
\begin{align*}
S_{\alpha}(t) u_1 = \sum_{n_1 \sim N_1} e^{-\i t z_{n_1}^{2 \alpha} } c_{n_1}  e_{n_1}(r) , \quad S_{\alpha}(t) u_2= \sum_{n_2 \sim N_2} e^{-\i t z_{n_2}^{2 \alpha} } d_{n_2}  e_{n_2}(r)  .
\end{align*}
Therefore, the bilinear objects that one needs to estimate are the $L_{t,x}^2$ norms of 
\begin{align*}
E_0(N_1, N_2) & = \sum_{n_1 \sim N_1} \sum_{n_2 \sim N_2}  e^{-\i t  (z_{n_1}^{2\alpha} +z_{n_2}^{2\alpha}) }  (c_{n_1} d_{n_2}) ( e_{n_1} e_{n_2}) ,\\
E_1(N_1, N_2) & = \sum_{n_1 \sim N_1} \sum_{n_2 \sim  N_2}  e^{-\i t  (z_{n_1}^{2\alpha} +z_{n_2}^{2\alpha}) }  (c_{n_1} d_{n_2} )(\nabla e_{n_1}  e_{n_2} ).
\end{align*}
Let us focus on \eqref{eq bilinear1} first. 
\begin{align}\label{eq bi2}
\begin{aligned}
(\text{LHS  of } \eqref{eq bilinear1})^2 & = \norm{E_0(N_1, N_2) }_{L^2 ((0, 1) \times \Theta)}^2  \\
& = \int_{\R \times \Theta} \abs{\sum_{n_1 \sim N_1} \sum_{n_2 \sim N_2}  e^{-\i t  (z_{n_1}^{2\alpha} +z_{n_2}^{2\alpha}) }  (c_{n_1} d_{n_2}) ( e_{n_1} e_{n_2}) }^2 \, dx dt
\end{aligned}
\end{align}
Here we employ a similar argument used in the proof of Lemma 2.6  in \cite{ST}. We fix $\eta \in C_0^{\infty} ((0,1))$, such that $\eta \big|_{I} \equiv 1$  where $I$ is a slight enlargement of $(0,1)$. Thus we continue from \eqref{eq bi2}
\begin{align}\label{eq bi3}
\eqref{eq bi2} & \leq \int_{\R \times \Theta} \eta(t) \abs{ \sum_{n_1 \sim N_1} \sum_{n_2 \sim N_2}  e^{-\i t  (z_{n_1}^{2\alpha} +z_{n_2}^{2\alpha}) }  (c_{n_1} d_{n_2}) ( e_{n_1} e_{n_2})}^2 \, dxdt \notag\\
& = \int_{\R \times \Theta} \eta(t) \abs{\sum_{\tau}   \sum_{ (n_1 ,n_2) \in \Lambda_{N_1 ,N_2, \tau} }  e^{-\i t  (z_{n_1}^{2\alpha} +z_{n_1}^{2\alpha}) }  (c_{n_1} d_{n_2}) ( e_{n_1} e_{n_2})}^2 \, dxdt ,
\end{align}
where
\begin{align*}
\#  \Lambda_{N_1 ,N_2, \tau} = \# \{ (n_1 ,n_2) \in \N^2 : n_1 \sim N_1 , n_2 \sim N_2 , \abs{z_{n_1}^{2\alpha} +z_{n_2}^{2\alpha} - \tau}\leq \frac{1}{2} \} .
\end{align*}
By expanding the square above and using Plancherel in time, we have
\begin{align}\label{eq bi4}
\eqref{eq bi3} & = \int_{\R \times \Theta} \eta(t) \sum_{\tau , \tau'} \sum_{\substack{ (n_1 ,n_2) \in \Lambda_{N_1 ,N_2, \tau}\\ (n_1' ,n_2') \in \Lambda_{N_1' ,N_2', \tau'} }}    e^{\i t  (z_{n_1'}^{2\alpha} +z_{n_2'}^{2\alpha} - z_{n_1}^{2\alpha} - z_{n_2}^{2\alpha}) }  (c_{n_1} d_{n_2})  (\overline{c_{n_1'} d_{n_2'}})  ( e_{n_1} e_{n_2})( e_{n_1'} e_{n_2'}) \, dxdt \notag\\
& =   \sum_{\tau , \tau'} \sum_{\substack{ (n_1 ,n_2) \in \Lambda_{N_1 ,N_2, \tau}\\ (n_1' ,n_2') \in \Lambda_{N_1' ,N_2', \tau'} }}   \widehat{\eta}((z_{n_1’}^{2\alpha} +z_{n_1'}^{2\alpha}) - (z_{n_1}^{2\alpha} +z_{n_2}^{2\alpha} ))    (c_{n_1} d_{n_2}) (\overline{c_{n_1'} d_{n_2'} }) \int_{ \Theta} ( e_{n_1} e_{n_2}) ( e_{n_1'} e_{n_2'}) \, dx \notag\\
& \lesssim  \sum_{\tau ,\tau'} \frac{1}{1+ \abs{\tau -\tau'}^2}  \sum_{\substack{n_1 \sim N_1 , n_2 \sim N_2 \\n_1' \sim N_1' , n_2' \sim N_2'}}  \mathbf{1}_{\Lambda_{N_1 ,N_2, \tau}}  \mathbf{1}_{\Lambda_{N_1' ,N_2', \tau'}} (c_{n_1} d_{n_2}) (\overline{c_{n_1'} d_{n_2'}})  \norm{ e_{n_1} e_{n_2}}_{L^2 (\Theta)}  \norm{ e_{n_1'} e_{n_2'}}_{L^2 (\Theta)} .
\end{align}
Then by Schur's test, we arrive at
\begin{align}\label{eq bi5}
\eqref{eq bi4} & \lesssim  \sum_{\tau \in \N} \parenthese{ \sum_{(n_1 ,n_2) \in \Lambda_{N_1, N_2, \tau}}  \abs{c_{n_1} d_{n_2}}  \norm{ e_{n_1} e_{n_2}}_{L^2 (\Theta)} }^2 \notag\\
& \lesssim    \sum_{\tau \in \N}  \# \Lambda_{N_1 ,N_2, \tau}  \sum_{(n_1 ,n_2) \in \Lambda_{N_1, N_2, \tau}}  \abs{c_{n_1} d_{n_2}}^2  \norm{ e_{n_1} e_{n_2}}_{L^2 (\Theta)}^2 .
\end{align}

We claim that
\begin{claim}\label{claim bilinear1}
\begin{enumerate}
\item
$\# \Lambda_{N_1, N_2, \tau}  = \mathcal{O}(N_2)$ ;
\item
$\norm{ e_{n_1} e_{n_2} }_{L^2(\Theta)}^2 \lesssim   N_2^{\varepsilon}$ .
\end{enumerate}
\end{claim}

Assuming Claim \ref{claim bilinear1}, we see that
\begin{align*}
\eqref{eq bi5} & \lesssim   \sum_{\tau \in \N}   N_2^{1+\varepsilon}  \sum_{(n_1 ,n_2) \in \Lambda_{N_1, N_2, \tau}}  \abs{c_{n_1} d_{n_2}}^2  \lesssim N_2^{1 + \varepsilon}  \norm{u_1}_{L^2 (\Theta)}^2 \norm{u_2}_{L^2 (\Theta)}^2.
\end{align*}
Therefore, \eqref{eq bilinear1} follows.

Now we are left to prove Claim \ref{claim bilinear1}. 
\begin{proof}[Proof of Claim \ref{claim bilinear1}]
In fact, {\it (2)} is due to H\"older inequality and the logarithmic bound on the $L^p$ norm of $e_n$ in \eqref{eq e_n bdd}.

For {\it (1)}, we have that for fixed $\tau \in \N$ and fixed $n_2 \sim N_2$
\begin{align*}
\abs{z_{n_1}^{2\alpha} +z_{n_2}^{2\alpha} - \tau}\leq \frac{1}{2}  \implies z_{n_1} \in [(\tau -\frac{1}{2} - z_{n_2}^{2\alpha})^{\frac{1}{2\alpha}}  , (\tau + \frac{1}{2} - z_{n_2}^{2\alpha})^{\frac{1}{2\alpha}} ] .
\end{align*}
There are at most 1 integer $z_{n_1}$ in this interval by concavity
\begin{align*}
(\tau + \frac{1}{2} - z_{n_2}^{2\alpha})^{\frac{1}{2\alpha}} - (\tau -\frac{1}{2} - z_{n_2}^{2\alpha})^{\frac{1}{2\alpha}}  \leq 1^{\frac{1}{2\alpha}}  =1 .
\end{align*}
Let us remark that the  restriction $\alpha \geq \frac{1}{2}$ on the fractional Laplacian in this section  comes from the concavity that we used here.

Then
\begin{align*}
\#  \Lambda_{N_1 ,N_2, \tau} = \# \{ (n_1 ,n_2) \in \N^2 : n_1 \sim N_1 , n_2 \sim N_2 , \abs{z_{n_1}^{2\alpha} +z_{n_2}^{2\alpha} - \tau}\leq \frac{1}{2} \} \sim \mathcal{O} (N_2) .
\end{align*}
We finish the proof of Claim \ref{claim bilinear1}.
\end{proof}

The estimation of  \eqref{eq bilinear2} is similar, hence omitted. 

The proof of Lemma \ref{lem bilinear} is complete now.
\end{proof}

\begin{rmk}
One may guess that the bilinear Strichartz could be done via the radial Sobolev embedding,
\begin{align*}
\abs{\abs{x}^{\frac{1}{2}-} f} \lesssim \norm{f}_{\dot{H}^{\frac{1}{2}+}} ,
\end{align*}
however it is not clear how to deal with the weight on the left hand side. If there were no such weight, it should be sufficient to prove the bilinear estimate using the radial Sobolev embedding. 
\end{rmk}

\subsection{Bilinear Strichartz estimates for NLS}
The computation above also holds for  the classical NLS ($\alpha=1$). However, we have a slightly better local well-posedness index  because of the following better  bilinear estimate.
\begin{lem}[Bilinear estimates for classical NLS]\label{lem bilinear'}
Under the same setup as in Lemma \ref{lem bilinear}, the bilinear estimate for classical NLS is given by
\begin{align*}
\norm{S_1(t) u_1  \, S_1(t) u_2}_{L_{t,x}^2 ((0,1) \times \Theta)} & \lesssim  N_2^{ \varepsilon}   \norm{u_1}_{L_x^2 (\Theta)} \norm{u_2}_{L_x^2 (\Theta)}, \\
\norm{ \nabla S_1(t) u_1  \,  S_1(t) u_2}_{L_{t,x}^2 ((0,1) \times \Theta)} & \lesssim N_1N_2^{\varepsilon}  \norm{u_1}_{L_x^2 (\Theta)} \norm{u_2}_{L_x^2 (\Theta)}.
\end{align*}
\end{lem}
The proof of Lemma \ref{lem bilinear'}  can be found in \cite{tzvNLS06}. See also  \cite{an} for its extension in the unit ball setting.

It is worth pointing out that  in the proof of Lemma \ref{lem bilinear'}, for example in a similar step like \eqref{eq bi2}, one can periodize the time and use the following stronger counting lemma to estimate the number of integer points on the characteristic surface instead of counting the integer points in a thin neighborhood of the characteristic surface
\begin{lem}[Lemma 3.2 in \cite{bgtBil}]\label{lem counting}
Let $M,N \in \N$, then for any $\varepsilon > 0$, there exists $C>0$ such that
\begin{align*}
\# \{ (k_1 , k_2) \in \N^2 : N \leq k_1 \leq 2N, k_1^2 + k_2^2 = M \} \leq C N^{\varepsilon} .
\end{align*}
\end{lem}

However, we will not distinguish $\alpha=1$ case from other fractional ones in the following sections. This is because the dominated term in the energy increment (the term that will give the largest energy increment and then determine  the global well-posedness index in Section \ref{sec gwp}) will be the almost the same even if we take  this better bilinear estimate into consideration.

\section{I-operator and a modified local theory}\label{sec LWPI}
In this section, we first define the I-operator  in our setting and then present a local well-posedness argument for the I-operator modified equation.
\subsection{Definition of I-operator}
\begin{defn}[I-operator]\label{defn I}
For $N \gg 1$, and a function $u = \sum_{n=1}^{\infty} c_n e_n$, define a smooth operator $I_N$, such that
\begin{align*}
I_N u =    \sum_{n=1}^{\infty} m_N(z_n)c_n  e_n.
\end{align*}
where $m_N$ is a smooth function satisfying 
\begin{align*}
m_N(\xi) = 
\begin{cases}
1, & \abs{\xi} \leq N\\
\parenthese{\frac{\abs{\xi}}{N}}^{s-\alpha}, & \abs{\xi} \geq 2N .
\end{cases}
\end{align*}
\end{defn}

\begin{rmk}\label{rmk ID}
The I-operator defined above is the analogue of the one in \cite{CKSTT} in the physical space. In the rest of this section, we will adopt the name `multiplier' of $m$  from the context in \cite{CKSTT}. 
It is easy to check that for $N \gg 1$
\begin{align*}
\norm{u}_{H^s} & \lesssim \norm{I_N u}_{H^{\alpha}} \lesssim N^{\alpha -s} \norm{u}_{H^s} ,\\
\norm{u}_{X^{s, \frac{1}{2}+}} & \lesssim \norm{I_N u}_{X^{\alpha , \frac{1}{2}+}} \lesssim N^{\alpha -s} \norm{u}_{X^{s, \frac{1}{2}+}} .
\end{align*}
\end{rmk}

A standard I-method argument usually comes in three major parts.
\begin{enumerate}[\it \text{Part} 1]
\item
(Subsection \ref{ssec LWP}) a well-adapted local theory for the I-operator modified fractional NLS,
\item
(Section \ref{sec energy increment}) the almost conservation law argument addressing the energy increment on each iteration,
\item
(Section \ref{sec gwp}) an iterative globalization argument giving the global index and a polynomial bound of the $H^s$ norm of the solution. 
\end{enumerate}

\subsection{A local theory based on $I_N$-operator}\label{ssec LWP}
Consider the following $I_N$-operator modified FNLS equation with initial data also being smoothed into the energy space.
\begin{align}\label{fNLSI}
\begin{cases}
(i \partial_t  - (-\Delta)^{\alpha}) I_N u = I_N (\abs{u}^2 u),\\
I_N u(0) = I_N u_0 .
\end{cases}
\end{align}
Note that $Iu_0 \in H^{\alpha}$ and $\norm{Iu_0}_{H^{\alpha}} \lesssim N^{\alpha -s}$.

To keep our notation compact, we will write $I$ and $m$ instead of $I_N$ and $m_N$ as in Definition \ref{defn I}.

The main result in this subsection is the following local well-posedness theory. 
\begin{prop}[Local well-posedness]\label{prop LWPI}
For $s \in ( \frac{1}{2} , \alpha]$ and $Iu_0 \in H^{\alpha}$, \eqref{fNLSI} is locally well posed. That is, there exists $\delta  \sim \norm{Iu_0}_{H^{\alpha}}^{-\frac{2}{1 + 2b -4b(s)}} \gtrsim N^{-\frac{2(\alpha -s)}{1 + 2b -4b(s)}} $, where $b(s) = \frac{1}{4} +\frac{1}{2} (1-s)+$, such that $Iu \in C([0, \delta] , H^{\alpha}(\Theta))$ solves \eqref{fNLSI} on $[0, \delta]$ and satisfies
\begin{align*}
\norm{I u}_{X_{\delta}^{\alpha ,\frac{1}{2}+}}  \lesssim \norm{Iu_0}_{H^{\alpha}} \lesssim N^{\alpha -s}.
\end{align*}
\end{prop}

We will prove Proposition \ref{prop LWPI} by a standard contraction mapping argument.  Note that the key step to close such argument is the following nonlinear estimate lemma. 
\begin{lem}[Nonlinear estimates]\label{lem nonlinear est}
For $s > \frac{1}{2}$, there exist $b, b' \in \R$ satisfying
\begin{align*}
0 < b' < \frac{1}{2} < b, \quad b + b' < 1 ,
\end{align*}
such that for every triple $(u_1, u_2 , u_3)$ in $X^{s,b} (\R \times \Theta)$,
\begin{align*}
\norm{I (\abs{u}^2 u)}_{X^{s, -b'} (\R \times \Theta)} \lesssim   \prod_{j=1}^3 \norm{Iu}_{X^{s,b} (\R \times \Theta)}^3  .
\end{align*} 
\end{lem}

Assuming Lemma \ref{lem nonlinear est}, we can easily finish the proof of Proposition \ref{prop LWPI}.
\begin{proof}[Proof of Proposition \ref{prop LWPI}]
Using Lemma \ref{lem Duhamel} and Proposition \ref{prop LWPI}, we have the following standard contraction mapping calculation 
\begin{align*}
\norm{Iu}_{X_{\delta}^{\alpha,b}} & \lesssim  \norm{Iu_0}_{H^{\alpha}} +  \delta^{1- b - b(s)} \norm{I (\abs{u}^{2} u)}_{X_{\delta}^{\alpha, -b(s)}} \\
& \lesssim \norm{Iu_0}_{H^{\alpha}} +  \delta^{1- b - b(s)} \norm{Iu}_{X_{\delta}^{\alpha, b(s)}}^{3} \\
& \lesssim \norm{Iu_0}_{H^{\alpha}} + \delta^{1- b - b(s)} \delta^{3(b-b(s))-} \norm{Iu}_{X_{\delta}^{\alpha, b}}^{3}  ,
\end{align*}
where $b = \frac{1}{1}+$. 
By choosing $\delta^{1+ 2b - 4b(s)}  \sim \norm{Iu_0}_{H^{\alpha}}^{-2} \gtrsim N^{-2(\alpha -s)}$ as what one did in a standard contraction mapping proof  and a  continuity argument we have that
\begin{align*}
\norm{Iu}_{X_{\delta}^{\alpha,b}}  \lesssim  \norm{Iu_0}_{H^{\alpha}} \lesssim N^{\alpha -s}
\end{align*}
and
\begin{align*}
\delta \gtrsim N^{-\frac{2(\alpha -s)}{1 + 2b -4b(s)}} .
\end{align*}
\end{proof}

Now we are left to prove the key Lemma \ref{lem nonlinear est} in this section.
\begin{proof}[Proof of Lemma \ref{lem nonlinear est}]
It is sufficient to show the following nonlinear estimate: for $\frac{1}{2} < s < 1$, $b(s) = \frac{1}{4} +\frac{1}{2} (1-s)+$ and $u \in X_{\delta}^{s,b(s)}$ 
\begin{align}\label{eq loc nonlinear}
\norm{I (\abs{u}^2 u)}_{X_{\delta}^{\alpha , -b(s)}} \lesssim \norm{I  u}_{X_{\delta}^{\alpha , b(s)}}^3 .
\end{align}
By duality argument, it is sufficient to prove for $v \in X^{-\alpha , b(s)}$
\begin{align*}
\abs{\int_{\R \times \Theta} \overline{v} \,  I (\abs{u}^2 u) \, dx dt } \lesssim \norm{v}_{X^{-\alpha , b(s)}} \norm{I  u}_{X_{\delta}^{\alpha , b(s)}}^3 .
\end{align*}
We will frequently make use of a dyadic decomposition in frequency using the orthonomal basis $e_n$'s of the radial Dirichlet Laplacian $-\Delta$, writing 
\begin{align*}
v_0 = \sum_{N_0 \leq \inner{z_n} < 2N_0} P_n v,
\end{align*}
and
\begin{align}\label{eq u decomp}
u_i = \sum_{N_i \leq \inner{z_n} < 2 N_i} P_n u, \quad \text{for } i =1,2,3.
\end{align}
After the dyadic frequency decomposition, we take a typical term and compute for the quadruple $\underline{N} = (N_0, N_1, N_2, N_3)$
\begin{align*}
L(\underline{N}) & :=\abs{\int_{\R \times \Theta} \overline{v_0} \, I(u_1 \overline{u_2} u_3) \, dxdt } .
\end{align*}
In order to distribute the I-operator inside the nonlinear term, we first move the  I-operator  on $v_0$, then introduce $m(N_i)$ (instead of  I-operator) into each $u_i$.
\begin{align}
L(\underline{N}) & = \abs{\int_{\R \times \Theta} \overline{I v_0} \, (u_1 \overline{u_2} u_3) \, dxdt } \notag\\
& = \frac{1}{m(  N_1) m(  N_2) m(  N_3)}  \abs{\int_{\R \times \Theta} \overline{I v_0} \, \cdot m(  N_1) u_1 \cdot \overline{m(  N_2)u_2} \cdot m(  N_3) u_3 \, dxdt } . \label{eq loc2}
\end{align}
We will explain the reason why we brought in $m$ instead of  I-operator in this calculation in Remark \ref{rmk abuse}.

By symmetry argument and because the presence of complex conjugates will play no role here, we can assume $N_1 \geq N_2 \geq N_3$. Then we can reduce the sum into the following two cases:
\begin{enumerate}
\item
$ N_0 \lesssim N_1$
\item
$ N_0 \gtrsim N_1$.
\end{enumerate}

{\bf Case 1:} $N_0 \lesssim  N_1$. 

Recall Remark \ref{rmk inter bilinear} where by taking $\beta = s$, we have
\begin{align}\label{eq bi1}
\norm{ f_i f_j}_{L_{t,x}^2 ((0,\delta) \times \Theta)} \lesssim  \min \{N_i, N_j \}^{s}  \norm{f_i}_{X_{\delta}^{0,  b(s)} } \norm{f_j}_{X_{\delta}^{0,  b(s)} } ,
\end{align}
where
\begin{align*}
b(s) = \frac{1}{4} +\frac{1}{2} (1-s)+ , \quad s \in ( \frac{1}{2} , 1].
\end{align*}
Using  \eqref{eq bi1} and Definition \ref{defn Xsb}, we write  \eqref{eq loc2} as
\begin{align}
L (\underline{N}) & \lesssim  \frac{m(  N_0)}{m(  N_1) m(  N_2) m(  N_3)} \int_{\R \times \Theta} \abs{ \overline{v_0} \, \cdot m(  N_1) u_1 \cdot \overline{m(  N_2)u_2} \cdot m(  N_3) u_3} \, dxdt  \notag\\
&  \lesssim \frac{m (N_0)}{m (N_1) m (N_2) m (N_3)} \norm{v_0  m(N_2)u_2}_{L_{t,x}^2} \norm{m(N_1)u_1 \cdot m(N_3)u_3}_{L_{t,x}^2} \notag\\
& \lesssim \frac{m (N_0 )}{m (N_1 ) m (N_2 ) m (N_3 )} (N_2N_3)^{s}  \norm{v_0}_{X_{\delta}^{0,  b(s) }} \prod_{i=1}^3 \norm{Iu_i}_{X_{\delta}^{0,  b(s) }} \notag\\
& \lesssim \frac{m (N_0 )}{m (N_1 ) m (N_2 ) m (N_3 )} (N_0 N_1^{-1})^{\alpha} (N_2 N_3)^{s-\alpha} \norm{v_0}_{X_{\delta}^{-\alpha, b(s) }} \prod_{i=1}^3 \norm{Iu_i}_{X_{\delta}^{\alpha, b(s) }}\label{eq loc1} .
\end{align}
Note that here we used $\norm{Iu_i}_{X^{s,b}} \sim \norm{m(N_i) u_i}_{X^{s,b}}$.

To continue the computation, we then consider the following two scenarios for $N_2$ and $N_3$
\begin{align*}
\frac{N_i^{s -\alpha}}{m(N_i )} = 
\begin{cases}
N_i^{s -\alpha} & \text{ if } N_i \leq N\\
(N^{-1}N_i)^{\alpha -s} N_i^{s -\alpha}  = N^{s -\alpha}  & \text{ if } N_i > 2N .
\end{cases}
\end{align*}
This observation implies that $L(\underline{N})$ is summable in $N_2$ and $N_3$. That is, by taking out the terms in $L(\underline{N})$ in \eqref{eq loc1}  that only depend on frequencies $N_2$ and $N_3$, we see 
\begin{align*}
\sum_{N_3 \leq N_2} \frac{(N_2 N_3)^{s-\alpha} }{m (N_2 ) m (N_3 )}  \norm{Iu_2}_{X^{\alpha, b(s) }} \norm{Iu_3}_{X^{\alpha, b(s) }} \lesssim  \norm{Iu}_{X^{\alpha, b(s) }}^2 .
\end{align*}

Now we focus on the sum over $N_0$ and $N_1$ in \eqref{eq loc1}.

First write
\begin{align*}
\frac{m (N_0 )}{m (N_1 ) } \parenthese{\frac{N_0}{N_1}}^{\alpha} &  \lesssim 
\begin{cases}
\parenthese{\frac{N_0}{N_1}}^{\alpha}  & \text{ if } N_0 \lesssim  N_1 \leq N\\
\parenthese{\frac{N_1}{N}}^{\alpha -s} \parenthese{\frac{N_0}{N_1}}^{\alpha}  & \text{ if } N_0 \leq N  \leq  N_1 \\
\parenthese{\frac{N_1}{N_0}}^{\alpha -s}  \parenthese{\frac{N_0}{N_1}}^{\alpha}  & \text{ if } N \leq N_0  \leq  N_1 
\end{cases}\\
& \lesssim \parenthese{\frac{N_0}{N_1}}^{\alpha -s} ,
\end{align*}
then the sum over $N_0$ and $N_1$ in \eqref{eq loc1} becomes
\begin{align*}
\sum_{N_0 \lesssim N_1} \frac{m (N_0 )}{m (N_1 ) } \parenthese{\frac{N_0}{N_1}}^{\alpha}  \norm{v_0}_{X^{-\alpha, b(s) }}  \norm{Iu_1}_{X^{\alpha, b(s) }} & \lesssim \sum_{N_0 \lesssim N_1} \parenthese{\frac{N_0}{N_1}}^{\alpha -s}  \norm{v_0}_{X^{-\alpha, b(s) }}  \norm{Iu_1}_{X^{\alpha, b(s) }} \\
&  \lesssim \norm{v}_{X^{-\alpha , b(s)}} \norm{I  u}_{X_{\delta}^{\alpha ,  b(s)}} .
\end{align*}

Here is a quick remark on the calculation in \eqref{eq loc2} and what follows. We originally planned to introduce the I-operator into \eqref{eq loc2} instead of $m(N_i)$. But this needs to bring the absolute value sign inside of the integral in \eqref{eq loc2}.

{\bf Case 2: } $N_0 \gtrsim N_1$.

First recall Green's theorem,
\begin{align*}
\int_{\Theta} \Delta f g - f \Delta g \, dx = \int_{\mathbb{S}} \frac{\partial f}{\partial v} g - f \frac{\partial g}{\partial v} \, d \sigma .
\end{align*}
Note that
\begin{align*}
-\Delta e_k = z_k^2 e_k ,
\end{align*}
where $ z_k^2$'s are the eigenvalues defined in \eqref{eq z_n}. Then we write
\begin{align*}
Iv_{0} = -\frac{\Delta}{N_0^2} \sum_{z_{n_0} \sim N_0} c_{n_0} \parenthese{\frac{N_0}{z_{n_0}}}^2 e_{n_0} ,
\end{align*}
where 
$c_{n_0} = m_N (z_{n_0}) c_{n_0}'$ and $c_{n_0}' = \inner{v_0 , e_{n_0}}_{L^2}$. Here  $c_{n_0}$ is the coefficient in front of eigenfunction $e_n$ for $Iv_0$ while $c_{n_0}'$ is the coefficient in front of eigenfunction $e_n$ for $v_0$.

Define
\begin{align*}
T (Iv_{0 }) & = \sum_{z_{n_0} \sim N_0} c_{n_0} \parenthese{\frac{N_0}{z_{n_0}}}^2 e_{n_0} , \qquad V(Iv_{0 })  = \sum_{z_{n_0} \sim N_0} c_{n_0} \parenthese{\frac{z_{n_0}}{N_0}}^2 e_{n_0} .
\end{align*}
It is easy to see that for all $s$
\begin{align*}
TV  (Iv_{0 }) & = VT (Iv_{0 }) = Iv_{0 } ,\\
\norm{T (Iv_{0 })}_{H_x^s} & \sim \norm{ Iv_{0 } }_{H_x^s} \sim \norm{V (Iv_{0 }) }_{H_x^s} .
\end{align*} 
Using this notation, we write
\begin{align*}
Iv_{0} = -\frac{\Delta}{N_0^2} T (Iv_{0} )
\end{align*}
and by Green's theorem
\begin{align*}
L(\underline{N}) \lesssim   \frac{1}{m (N_1 ) m (N_2 ) m (N_3 )} \frac{1}{N_0^2} \int_{\R \times \Theta} T (Iv_{0 }) \Delta (\prod_{j=1}^{3} m(N_j)u_{j}) .
\end{align*}

By the product rule and the assumption that $N_1 \geq N_2 \geq N_3$, we only need to consider the two largest cases of $\Delta ( u_1  u_{2} u_{3}) $. They are
\begin{enumerate}
\item
$(\Delta u_{1}) u_{2} u_{3}$,
\item
$(\nabla u_{1}) \cdot (\nabla u_{2})  u_{3} $.
\end{enumerate}
We denote
\begin{align*}
J_{11} (\underline{N}) & = \int_{\R \times \Theta} T (Iv_{0}) (\Delta m(N_1)u_{1}) (m(N_2)u_{2}) (m(N_3)  u_{3}), \\
J_{12} (\underline{N}) & =\int_{\R \times \Theta} T (Iv_{0}) (\nabla m(N_1)u_{1}) \cdot (\nabla m(N_2)u_{2}) (m(N_3)u_{3})  .
\end{align*}
Using $\Delta u_i = -N_i^2 V u_i$ and  \eqref{eq bi1} under the similar calculation as in \eqref{eq loc1}, we obtain
\begin{align*}
\frac{1}{N_0^2} \abs{J_{11} (\underline{N})} \lesssim m(N_0) \parenthese{\frac{N_1}{N_0}}^2 (N_0 N_1^{-1})^{\alpha} (N_2 N_3)^{s-\alpha} \norm{v_0}_{X^{-\alpha, b(s) }} \prod_{i=1}^3 \norm{Iu_i}_{X^{\alpha, b(s) }}.
\end{align*}
Now for $\abs{J_{12} (\underline{N})} $, we estimate it in a similar fashion and obtain that
\begin{align*}
\frac{1}{N_0^2} \abs{J_{12} (\underline{N})} \lesssim m(N_0) \frac{N_1 N_2}{N_0^2} (N_0 N_1^{-1})^{\alpha} (N_2 N_3)^{s-\alpha} \norm{v_0}_{X^{-\alpha, b(s) }} \prod_{i=1}^3 \norm{Iu_i}_{X^{\alpha, b(s) }}.
\end{align*}
Therefore, we have
\begin{align*}
L(\underline{N}) \lesssim  \frac{m (N_0 )}{m (N_1 ) m (N_2 ) m (N_3 )}  \parenthese{\frac{N_1}{N_0}}^2 (N_0 N_1^{-1})^{\alpha} (N_2 N_3)^{s-\alpha} \norm{v_0}_{X^{-\alpha, b(s) }} \prod_{i=1}^3 \norm{Iu_i}_{X^{\alpha, b(s) }}.
\end{align*}

To sum $N_2$ and $N_3$, we can do exactly the same thing as in {\bf Case 1}. Then the sum over $N_0$ and $N_1$ becomes
\begin{align*}
\sum_{N_0 \gtrsim N_1} \frac{m (N_0 )}{m (N_1 ) } \parenthese{\frac{N_0}{N_1}}^{\alpha}  \parenthese{\frac{N_1}{N_0}}^2 \norm{v_0}_{X^{-\alpha, b(s) }}  \norm{Iu_1}_{X^{\alpha, b(s) }} & \lesssim \sum_{N_0 \gtrsim N_1} \parenthese{\frac{N_1}{N_0}}^{2-\alpha }  \norm{v_0}_{X^{-\alpha, b(s) }}  \norm{Iu_1}_{X^{\alpha, b(s) }} \\
&  \lesssim \norm{v}_{X^{-\alpha , b(s)}} \norm{I  u}_{X_{\delta}^{\alpha ,  b(s)}} .
\end{align*}

Therefore
\begin{align*}
\abs{\int_{\R \times \Theta} \overline{v} I (\abs{u}^2 u) \, dx dt } \lesssim \sum_{\underline{N}} L(\underline{N}) \lesssim  \norm{v}_{X^{-\alpha , b(s)}} \norm{I  u}_{X_{\delta}^{\alpha , b(s)}}^3 .
\end{align*}
which implies \eqref{eq loc nonlinear}.

This finishes the proof of Lemma  \ref{lem nonlinear est}.
\end{proof}

\begin{rmk}\label{rmk abuse}
Here is a quick remark on the calculation in \eqref{eq loc2} and what follows. We originally planned to introduce the I-operator into \eqref{eq loc2} instead of $m(N_i)$. But this needs to bring the absolute value sign inside of the integral in \eqref{eq loc2}, which may ruin the Green's theorem. So with the calculation in this proof, we justify the `legality' of bringing I-operators inside. In the rest of this paper, we will use a slight abuse of calculation -- moving the I-operator without justification. For example, in \eqref{eq loc2} 
\begin{align*}
 \abs{\int_{\R \times \Theta} \overline{I v_0} \, (u_1 \overline{u_2} u_3) \, dxdt }  \lesssim  \frac{m(N_0)}{m(  N_1) m(  N_2) m(  N_3)}  \abs{\int_{\R \times \Theta} \overline{ v_0} \, \cdot I u_1 \cdot \overline{ Iu_2} \cdot I u_3 \, dxdt } . 
\end{align*}
\end{rmk}

\section{Weak interaction between functions localized in uncomparable frequencies}\label{sec weak}

Before we start the I-method argument, let us first understand the interaction between the functions at separate frequencies, which will be heavily used our proof and simplifies a lot of the case classification in Section \ref{sec energy increment}.

Recall that we have no similar convolution properties as what we have on $\R^d$, which results in losing control of the frequency connection. In fact, after taking Fourier transformation on the nonlinear term, the convolution property implies that the frequencies in each function are linearly connected  $\xi = \xi_1 -\xi_2 + \xi_3$. However this is no longer true under our setting, which means that for instance the maximum frequency could be extremely large instead of being controlled by a linear form of other lower frequencies. To deal with this bad scenario, let us take a closer look at the interaction among the nonlinearity.

This weak interaction is inspired by Lemma 2.6 in \cite{bgtBil}, that is
\begin{lem}
There exists $C >0$ such that, if for any $j = 1,2,3$, $C z_{n_j} \leq z_{n_0}$, then for every $p >0$ there exists $C_p > 0$ such that for every $w_j \in L^2 (\mathcal{M})$, $j =0,1,2,3$,
\begin{align}\label{eq BGT}
\abs{\int_{\mathcal{M}} P_{n_0} w_0 P_{n_1} w_1 P_{n_2} w_2 P_{n_3} w_3 \, dx} \leq C_p z_{n_0}^{-p} \prod_{j=0}^3 \norm{w_j}_{L^2} .
\end{align}
\end{lem}
Notice that on the right-hand side of \eqref{eq BGT}, the factor $z_{n_0}^{-p}$ gives a huge decay, which means that the interaction in fact is weak. Now let us present our version of such weak interaction.

\begin{prop}[Weak interaction]\label{prop weak}
For the frequency quadruple $\underline{n} = (n_0 , n_1 , n_2 , n_3) \in \N^4$ and $n_0 \gg  n_1 \geq n_2 \geq n_3$ (`$n_0 \gg  n_1$' means $n_0 \geq 2 n_1$), we consider the functions $w, f, g, h \in X^{0,b}$, $b=\frac{1}{2}+$ with their frequencies localized at $n_0, n_1, n_2 , n_3$. More precisely,
\begin{align*}
w (t,x) =  w_{n_0} (t) e_{n_0}(x), \quad f (t,x) =f_{n_1} (t) e_{n_1}(x), \quad g (t,x) =  g_{n_2} (t) e_{n_2}(x), \quad h(t,x) =  h_{n_3} (t) e_{n_3}(x).
\end{align*}
Define the interaction between these functions as follows
\begin{align}\label{eq J}
J(\underline{n}) : = \abs{ \int_{\R \times \Theta} \overline{w}  \cdot f  \cdot \overline{g} \cdot h \, dx dt }.
\end{align}
Then this interaction $J$ satisfies
\begin{align}\label{eq J bdd}
J(\underline{n})  \lesssim \frac{n_2^{\frac{3}{2}} n_3^{\frac{1}{2}}}{n_0^2} \frac{\norm{w}_{X^{0, b}} \norm{f}_{X^{0, b}} \norm{g}_{X^{0, b}}  \norm{h}_{X^{0, b}} }{ \inner{   n_0^{2\alpha} - n_1^{2\alpha} + n_2^{2\alpha} - n_3^{2\alpha}}^{b}}   . 
\end{align}
\end{prop}

\begin{proof}[Proof of Proposition \ref{prop weak}]
Notice that \eqref{eq J} can be written as 
\begin{align*}
J (\underline{n}) & = \int_{\R \times \Theta} \overline{w_{n_0} (t)} e_{n_0} (x) \cdot  f_{n_1}(t) e_{n_1}(x)  \cdot  \overline{g_{n_2}(t) } e_{n_2} (x)   \cdot  h_{n_3}(t) e_{n_3} (x) \, dxdt    \\
& = \parenthese{ \int_{\R} \overline{w_{n_0} (t)} \cdot  f_{n_1}(t) \cdot \overline{g_{n_2}(t)} \cdot  h_{n_3}(t) \,  dt}  \parenthese{\int_{\Theta}   e_{n_0} (x) e_{n_1}(x) e_{n_2} (x) e_{n_3}(x) \,  dx} = : A  \times B . 
\end{align*}

We will estimate the contributions of the two terms  in \eqref{eq J} in Lemma \ref{lem weak e_n} and Lemma \ref{lem weak fn} separately.
\begin{lem}[Weak interaction among separated eigenfunctions]\label{lem weak e_n}
If $n_0 \gg  n_1 \geq n_2 \geq n_3 $, then
\begin{align*}
B : = \int_{\Theta} e_{n_0} e_{n_1} e_{n_2} e_{n_3} \, dx < \mathcal{O} \parenthese{\frac{n_2^{\frac{3}{2}} n_3^{\frac{1}{2}}}{n_0^2}}  .
\end{align*}
\end{lem}

\begin{lem}[Interaction between frequency localized functions]\label{lem weak fn}
Under the same assumption as in Proposition \ref{prop weak}, the interaction between their coefficients at frequencies $n_0, n_1, n_2 , n_3$ satisfies
\begin{align*}
A  : =  \abs{\int_{\R} \overline{w_{n_0} (t)}   f_{n_1}(t) \overline{g_{n_2}(t)} h_{n_3}(t) \,  dt} \lesssim  \frac{\norm{w}_{X^{0, b}} \norm{f}_{X^{0, b}} \norm{ g}_{X^{0, b}}  \norm{h}_{X^{0, b}} }{ \inner{   n_0^{2\alpha} - n_1^{2\alpha} + n_2^{2\alpha} - n_3^{2\alpha}}^{b}}  .
\end{align*}
\end{lem}

Assuming the two lemmas above, it is easy to see the weak interaction as in \eqref{eq J bdd}.
\end{proof}

Now we are left to present the proofs of Lemma \ref{lem weak e_n} and Lemma \ref{lem weak fn}.
\begin{proof}[Proof of Lemma \ref{lem weak e_n}]
Before proving it, let us first make an observation. In fact, a naive estimate from \eqref{eq e_n bdd} gives that
\begin{align*}
\abs{\int_{\Theta} e_{n_0} e_{n_1} e_{n_2} e_{n_3} \, dx}  \lesssim \prod_{i=0}^3 \norm{e_{n_i}}_{L^4} \lesssim n_0^+ n_1^+ n_2^+ n_3^+ .
\end{align*}
However, we should expect  a much weaker interaction when the frequencies are separate ($n_0 \gg  n_1 \geq n_2 \geq n_3 $).

To prove this lemma, we first write 
\begin{align}\label{eq phiF}
\begin{aligned}
\phi (r) & = e_{n_2} (r) \, e_{n_3} (r) ,\\
F(r) & = \int_0^r \gamma \, e_{n_0}(\gamma) \, e_{n_1}(\gamma) \, d \gamma ,
\end{aligned}
\end{align}
then performing an integration by parts, we see that
\begin{align*}
\int_{\Theta} e_{n_0} e_{n_1} e_{n_2} e_{n_3} \, dx & \sim  \int_0^1 (r e_{n_0}(r) e_{n_1}(r)) \phi(r) \, dr = F(r) \phi(r) \Big|_0^1 - \int_0^1 F(r) \phi'(r) \, dr .
\end{align*}

To see the weak interaction, we claim that
\begin{claim}\label{claim e_n}
\begin{enumerate}
\item The boundary terms are zeros
\begin{align*}
F(1) \phi(1) = F(0) \phi(0) =0 ;
\end{align*}

\item Approximate formula for $F$
\begin{align*}
F(r) &  \sim  \frac{\sqrt{z_{n_0} z_{n_1}} }{z_{n_0}^2 -z_{n_1}^2} \,  r \parenthese{z_{n_0} J_1 (z_{n_0} r) J_0 (z_{n_1} r) - z_{n_1} J_0 (z_{n_0} r) J_1 (z_{n_1} r)};
\end{align*}

\item Approximate formula for $\phi'$
\begin{align*}
\phi'(r) & \sim \sqrt{z_{n_2} z_{n_3}}\parenthese{ z_{n_2}  J_1 (z_{n_2} r)  J_0 (z_{n_3} r) + z_{n_3} J_0 (z_{n_2} r)  J_1 (z_{n_3} r) } ;
\end{align*}

\item With the assumption $n_0 \gg  n_1 \geq n_2 \geq n_3 $ (recall that `$n_0 \gg  n_1$' means $n_0 \geq 2 n_1$), we obtain the weak interaction using the approximate formula for $F$ and $\phi'$, that is,
\begin{align*}
\abs{\int_{\Theta} e_{n_0} e_{n_1} e_{n_2} e_{n_3} \, dx } & = \abs{ \int_0^1 F(r) \phi'(r) \, dr } < \mathcal{O} \parenthese{\frac{z_{n_2}^{\frac{3}{2}} z_{n_3}^{\frac{1}{2}}}{z_{n_0}^2}} \sim \mathcal{O} \parenthese{\frac{n_2^{\frac{3}{2}} n_3^{\frac{1}{2}}}{n_0^2}}  .
\end{align*}
\end{enumerate}
\end{claim}

\begin{proof}[Proof of Claim \ref{claim e_n}]
For {\it (1)}, the zero boundary conditions can be justified since  $F(r)=0$ at $r=0$ and $\phi(r) =0$ at $r=1$ (recall $e_n (1) =0$).

For {\it (2)}, recall $F(r)  = \int_0^r \gamma \, e_{n_0}(\gamma) \, e_{n_1}(\gamma) \, d \gamma$,  and $e_n (r)  = \norm{J_0 (z_n \cdot)}_{L^2(\Theta)}^{-1} J_0 (z_n r)  $, then we write 
\begin{align*}
F(r) & = \int_0^r \gamma \, e_{n_0}(\gamma) \, e_{n_1}(\gamma) \, d \gamma  =  \norm{J_0 (z_{n_0} \cdot)}_{L^2(\Theta)}^{-1} \norm{J_0 (z_{n_1} \cdot)}_{L^2(\Theta)}^{-1} \int_0^r \gamma \, J_0 (z_{n_0} \gamma)\, J_0 (z_{n_1} \gamma) \, d \gamma .
\end{align*}
Combining  the following anti-derivative (for detailed calculation, see Lemma \ref{lem intF} in Appendix \ref{sec Appendix})
\begin{align*}
\int_0^r \gamma \, J_0 (z_{n_0} \gamma)\, J_0 (z_{n_1} \gamma) \, d \gamma =  \frac{1}{z_{n_0}^2 -z_{n_1}^2} \, r [z_{n_0} J_1 (z_{n_0} r) J_0 (z_{n_1} r) - z_{n_1} J_0 (z_{n_0} r) J_1 (z_{n_1} r)]
\end{align*}
we have
\begin{align*}
F(r) &  =  \frac{\norm{J_0 (z_{n_0} \cdot)}_{L^2(\Theta)}^{-1} \norm{J_0 (z_{n_1} \cdot)}_{L^2(\Theta)}^{-1}}{z_{n_0}^2 -z_{n_1}^2} \, r [z_{n_0} J_1 (z_{n_0} r) J_0 (z_{n_1} r) - z_{n_1} J_0 (z_{n_0} r) J_1 (z_{n_1} r)]\\
&  \sim  \frac{\sqrt{z_{n_0} z_{n_1}} }{z_{n_0}^2 -z_{n_1}^2} \,  r \parenthese{z_{n_0} J_1 (z_{n_0} r) J_0 (z_{n_1} r) - z_{n_1} J_0 (z_{n_0} r) J_1 (z_{n_1} r)},
\end{align*}
where in the last step, we used $\norm{J_0 (z_n \cdot)}_{L^2(\Theta)}^{-1}  \sim z_{n}^{\frac{1}{2}}$ in \eqref{eq J_0 norm}.

For {\it (3)}, by \eqref{eq e_n} and \eqref{eq dJ0}, we have
\begin{align*}
\phi'(r) &  = (e_{n_2} (r) \, e_{n_3} (r) )' =  \norm{J_0 (z_{n_2} \cdot)}_{L^2(\Theta)}^{-1} \norm{J_0 (z_{n_3} \cdot)}_{L^2(\Theta)}^{-1} (J_0 (z_{n_2} r)  J_0 (z_{n_3} r))' \\
& = \norm{J_0 (z_{n_2} \cdot)}_{L^2(\Theta)}^{-1} \norm{J_0 (z_{n_3} \cdot)}_{L^2(\Theta)}^{-1} (z_{n_2}  J_1 (z_{n_2} r)  J_0 (z_{n_3} r) + z_{n_3} J_0 (z_{n_2} r)  J_1 (z_{n_3} r) )\\
& \sim \sqrt{z_{n_2} z_{n_3}}( z_{n_2}  J_1 (z_{n_2} r)  J_0 (z_{n_3} r) + z_{n_3} J_0 (z_{n_2} r)  J_1 (z_{n_3} r) )
\end{align*}
where in the approximation above, we used \eqref{eq J_0 norm} again. Since $J_0(x) $, $J_1 (x)$ are bounded, we see that
\begin{align*}
\abs{\phi'(r) } \lesssim  \sqrt{z_{n_2} z_{n_3}} (z_{n_2} + z_{n_3}) .
\end{align*}

Now let us move on to {\it (4)}.

Intuitively, we can think of  $\gamma e_{n_0}(\gamma) e_{n_1} (\gamma)$ as a produce of two trigonometric functions, which is essentially $\cos ((z_{n_0} \pm z_{n_1})r)$ using the product to sum identities. The integral of $F$ is basically the twice integration of $\cos ((z_{n_0} \pm z_{n_1})r)$, which will bring out a factor of $\frac{1}{(z_{n_0} - z_{n_1})^2}$ by changing of variables. Putting together  the bound of $\phi'$ that we observed in {\it (2)}, which is  $z_{n_2}^{\frac{3}{2}} z_{n_3}^{\frac{1}{2}}$, we should be able to see the estimate in {\it (4)}. However, we need to justify this bound in the rest of this proof.

Recall \eqref{eq J0} and \eqref{eq Jinfty}. We will use the first formula for $\abs{x} < 1$ and the second one for $\abs{x} \geq 1$, where $n = 0,1$. 
\begin{align*}
J_n(x) & = \frac{1}{n! 2^n}  x^n + \mathcal{O} (x^{n+2}) , \\
J_n(x) & = \sqrt{\frac{2}{\pi}} \frac{\cos(x- \frac{n \pi}{2} - \frac{ \pi}{4})}{\sqrt{x}} + \mathcal{O} (x^{-\frac{3}{2}}) 
\end{align*}

Combining {\it (2)} and {\it (3)}, we write
\begin{align*}
\int_0^1 F(r) \phi'(r) \, dr & \sim \int_0^1 \frac{\sqrt{z_{n_0} z_{n_1}} }{z_{n_0}^2 -z_{n_1}^2} \,  r \square{z_{n_0} J_1 (z_{n_0} r) J_0 (z_{n_1} r) - z_{n_1} J_0 (z_{n_0} r) J_1 (z_{n_1} r) } \\
& \quad \times \sqrt{z_{n_2} z_{n_3}}  \square{ z_{n_2}  J_1 (z_{n_2} r)  J_0 (z_{n_3} r) + z_{n_3} J_0 (z_{n_2} r}  J_1 (z_{n_3} r) )  \, dr
\end{align*}
To estimate $ \int_0^1 F(r) \phi'(r) \, dr $, let us consider the following term as an example:
\begin{align}\label{eq ex}
\begin{aligned}
& \quad  \int_0^1 \frac{\sqrt{z_{n_0} z_{n_1}} }{z_{n_0}^2 -z_{n_1}^2} \, r  z_{n_0} J_1 (z_{n_0} r) J_0 (z_{n_1} r) \sqrt{z_{n_2} z_{n_3}}  z_{n_2}  J_1 (z_{n_2} r)  J_0 (z_{n_3} r)  \, dr \\
& =  \frac{ \sqrt{z_{n_0} z_{n_1} z_{n_2} z_{n_3}}  z_{n_0} z_{n_2}}{z_{n_0}^2 -z_{n_1}^2}  \int_0^1 r   J_1 (z_{n_0} r) J_0 (z_{n_1} r)   J_1 (z_{n_2} r)  J_0 (z_{n_3} r)  \, dr.
\end{aligned}
\end{align}
In fact, since $n_0 \geq n_1$ and $n_2 \geq n_3$, this term will contribute the most in the estimate of $ \int_0^1 F(r) \phi'(r) \, dr $.

We will use the different approximations of $J_0(x)$ and $J_1 (x)$ for $\abs{x} <1$ and $\abs{x} \geq 1$, hence we consider the the following five cases. 
\begin{enumerate}[-]
\item
{\bf Case I}: $0 \leq r < \frac{1}{z_{n_0}}$;
\item
{\bf Case II}: $\frac{1}{z_{n_0}} \leq r < \frac{1}{z_{n_1}}$;
\item
{\bf Case III}: $\frac{1}{z_{n_1}} \leq r < \frac{1}{z_{n_2}}$;
\item
{\bf Case IV}: $\frac{1}{z_{n_2}} \leq r < \frac{1}{z_{n_3}}$;
\item
{\bf Case V}: $\frac{1}{z_{n_3}} \leq r \leq 1$.
\end{enumerate}
We will focus on only the main terms in the approximation formulas, and the control of error terms can be found in Subsection \ref{ssec Error}.

\noindent {\bf Case I:} $0 \leq r < \frac{1}{z_{n_0}}$. 

In this case, all the arguments in $J_0$ and $J_1$ in \eqref{eq ex} will be smaller than $1$, hence we will need four approximations around the origin, that is,
\begin{align*}
& \quad  \frac{ \sqrt{z_{n_0} z_{n_1} z_{n_2} z_{n_3}}  z_{n_0} z_{n_2}}{z_{n_0}^2 -z_{n_1}^2}    \int_0^{\frac{1}{z_{n_0}}} r   J_1 (z_{n_0} r) J_0 (z_{n_1} r)   J_1 (z_{n_2} r)  J_0 (z_{n_3} r)  \, dr\\
& \sim  \frac{ \sqrt{z_{n_0} z_{n_1} z_{n_2} z_{n_3}}  z_{n_0} z_{n_2}}{z_{n_0}^2 -z_{n_1}^2}   \int_0^{\frac{1}{z_{n_0}}} r   ( z_{n_0} r )  ( z_{n_2} r)  \, dr  \\
& \sim  \frac{ \sqrt{z_{n_0} z_{n_1} z_{n_2} z_{n_3}}  z_{n_0} z_{n_2}}{z_{n_0}^2 -z_{n_1}^2}   z_{n_0}z_{n_2} \, r^4 \Big|_0^{ \frac{1}{z_{n_0}}} \\
& =  \frac{ \sqrt{z_{n_0} z_{n_1} z_{n_2} z_{n_3}}  z_{n_0} z_{n_2}}{z_{n_0}^2 -z_{n_1}^2}  \frac{z_{n_0}z_{n_2} }{z_{n_0}^4}  \sim  \frac{n_1^{\frac{1}{2}} n_2^{\frac{5}{2}} n_3^{\frac{1}{2}}}{n_0^{\frac{3}{2}} (n_0^2 - n_1^2)} ,
\end{align*}
where in the last step, we used \eqref{eq z_n}.

\noindent {\bf Case II:} $\frac{1}{z_{n_0}} \leq r < \frac{1}{z_{n_1}}$.

In this case, the argument  $z_{n_0} r$ turns to be larger than $1$ in \eqref{eq ex}, hence we need its asymptotic approximation, while the other three are treated the same as in the previous case. 
\begin{align*}
& \quad \frac{ \sqrt{z_{n_0} z_{n_1} z_{n_2} z_{n_3}}  z_{n_0} z_{n_2}}{z_{n_0}^2 -z_{n_1}^2}  \int_{\frac{1}{z_{n_0}}}^{\frac{1}{z_{n_1}}} r   J_1 (z_{n_0} r) J_0 (z_{n_1} r)   J_1 (z_{n_2} r)  J_0 (z_{n_3} r)  \, dr\\
& \sim \frac{ \sqrt{z_{n_0} z_{n_1} z_{n_2} z_{n_3}}  z_{n_0} z_{n_2}}{z_{n_0}^2 -z_{n_1}^2}   \int_{\frac{1}{z_{n_0}}}^{\frac{1}{z_{n_1}}} r \frac{\sin (z_{n_0} r -\frac{\pi}{4})}{\sqrt{z_{n_0} r}} (z_{n_2} r) \, dr \\
& \sim \frac{ \sqrt{z_{n_0} z_{n_1} z_{n_2} z_{n_3}}  z_{n_0} z_{n_2}}{z_{n_0}^2 -z_{n_1}^2}   \frac{z_{n_2}}{\sqrt{z_{n_0}}} \int_{\frac{1}{z_{n_0}}}^{\frac{1}{z_{n_1}}} r^{\frac{3}{2}} \sin (z_{n_0} r -\frac{\pi}{4}) \, dr \\
& \lesssim  \frac{ \sqrt{z_{n_0} z_{n_1} z_{n_2} z_{n_3}}  z_{n_0} z_{n_2}}{z_{n_0}^2 -z_{n_1}^2}  \frac{z_{n_2}}{\sqrt{z_{n_0}}}  (\frac{z_{n_0}}{z_{n_1}})^{\frac{3}{2}}  z_{n_0}^{-\frac{5}{2}} \sim \frac{n_2^{\frac{5}{2}} n_3^{\frac{1}{2}}}{n_1 (n_0^2 -n_1^2)}. 
\end{align*}
Note that the last inequality above, we used  (see Lemma \ref{lem xpsin} in  Appendix \ref{sec Appendix} for detailed computation)
\begin{align*}
\abs{\int_a^b x^p \sin x \, dx}  \lesssim  a^p + b^p .
\end{align*}

]

\noindent {\bf Case III:} $\frac{1}{z_{n_1}} \leq r < \frac{1}{z_{n_2}}$.

Now we have two asymptotic approximations in \eqref{eq ex}
\begin{align*}
& \quad \frac{ \sqrt{z_{n_0} z_{n_1} z_{n_2} z_{n_3}}  z_{n_0} z_{n_2}}{z_{n_0}^2 -z_{n_1}^2}  \int_{\frac{1}{z_{n_1}}}^{\frac{1}{z_{n_2}}} r   J_1 (z_{n_0} r) J_0 (z_{n_1} r)   J_1 (z_{n_2} r)  J_0 (z_{n_3} r)  \, dr\\
& \sim \frac{ \sqrt{z_{n_0} z_{n_1} z_{n_2} z_{n_3}}  z_{n_0} z_{n_2}}{z_{n_0}^2 -z_{n_1}^2}   \int_{\frac{1}{z_{n_1}}}^{\frac{1}{z_{n_2}}} r \frac{\sin (z_{n_0} r -\frac{\pi}{4})}{\sqrt{z_{n_0} r}} \frac{\cos (z_{n_1} r -\frac{\pi}{4})}{\sqrt{z_{n_1} r}}  (z_{n_2} r) \, dr \\
& \sim \frac{ \sqrt{z_{n_0} z_{n_1} z_{n_2} z_{n_3}}  z_{n_0} z_{n_2}}{z_{n_0}^2 -z_{n_1}^2}   \frac{z_{n_2}}{\sqrt{z_{n_0} z_{n_1}}} \int_{\frac{1}{z_{n_1}}}^{\frac{1}{z_{n_2}}}  r  \sin (z_{n_0} r -\frac{\pi}{4}) \cos (z_{n_1} r -\frac{\pi}{4}) \, dr \\
& \lesssim  \frac{ \sqrt{z_{n_0} z_{n_1} z_{n_2} z_{n_3}}  z_{n_0} z_{n_2}}{z_{n_0}^2 -z_{n_1}^2}  \frac{z_{n_2}}{\sqrt{z_{n_0} z_{n_1}}}  \frac{1}{z_{n_2}z_{n_0}} \sim \frac{n_2^{\frac{3}{2}} n_3^{\frac{1}{2}}}{ n_0^2 -n_1^2},
\end{align*}
where in the last inequality, we used Lemma  \ref{lem xpsincos}  to obtain
\begin{align*}
 \int_{\frac{1}{z_{n_1}}}^{\frac{1}{z_{n_2}}}  r  \sin (z_{n_0} r -\frac{\pi}{4}) \cos (z_{n_1} r -\frac{\pi}{4}) \, dr \lesssim  \frac{1}{z_{n_2}z_{n_0}}.
\end{align*}

\noindent {\bf Case IV:} $\frac{1}{z_{n_2}} \leq r < \frac{1}{z_{n_3}}$.

Now we have two asymptotic approximations in $F(r)$, and one more from $\phi' (r)$. To reduce the number of trig functions in \eqref{eq ex}, we use the  trigonometric identity $\sin \alpha \cos \beta  =  \frac{1}{2} (\sin (\alpha + \beta) + \sin (\alpha -\beta))$ to re-write
\begin{align*}
F(r) & \sim  \frac{\sqrt{z_{n_0} z_{n_1}} }{z_{n_0}^2 -z_{n_1}^2}  r z_{n_0} \frac{\sin(z_{n_0} r -\frac{\pi}{4})}{\sqrt{z_{n_0} r}} \frac{\cos (z_{n_1} r - \frac{\pi}{4})}{\sqrt{z_{n_1} r}} - \frac{\sqrt{z_{n_0} z_{n_1}} }{z_{n_0}^2 -z_{n_1}^2}  r z_{n_1} \frac{\cos(z_{n_0} r - \frac{\pi}{4})}{\sqrt{z_{n_0} r}} \frac{\sin(z_{n_1} r -\frac{\pi}{4})}{\sqrt{z_{n_1} r}} \\
& = \frac{z_{n_0}}{z_{n_0}^2 -z_{n_1}^2} \sin(z_{n_0} r -\frac{\pi}{4}) \cos(z_{n_1} r -\frac{\pi}{4}) - \frac{z_{n_1}}{z_{n_0}^2 -z_{n_1}^2} \cos(z_{n_0} r -\frac{\pi}{4}) \sin(z_{n_1} r -\frac{\pi}{4}) \\
& \sim    -\frac{1}{z_{n_0} -z_{n_1}}  \cos  ((z_{n_0}+z_{n_1})r ) +  \frac{1}{z_{n_0} + z_{n_1}}  \sin ((z_{n_0}-z_{n_1})r) . 
\end{align*}
Recall 
\begin{align*}
\phi'(r) &   \sim \sqrt{z_{n_2} z_{n_3}}( z_{n_2}  J_1 (z_{n_2} r)  J_0 (z_{n_3} r) + z_{n_3} J_0 (z_{n_2} r)  J_1 (z_{n_3} r) )
\end{align*}
then we write
\begin{align}\label{eq Fphi' new}
\int_{\frac{1}{z_{n_1}}}^{\frac{1}{z_{n_2}}} F(r) \phi'(r) \, dr \sim \int_{\frac{1}{z_{n_1}}}^{\frac{1}{z_{n_2}}} \parenthese{-\frac{1}{z_{n_0} -z_{n_1}}  \cos  ((z_{n_0}+z_{n_1})r ) + \frac{1}{z_{n_0} + z_{n_1}}  \sin ((z_{n_0}-z_{n_1})r)  } \phi'(r) \, dr .
\end{align}

Take the following term as an example (since it will contribute the most in \eqref{eq Fphi' new}, hence \eqref{eq ex})
\begin{align*}
& \quad \int_{\frac{1}{z_{n_2}}}^{\frac{1}{z_{n_3}}}  \frac{1}{z_{n_0} -z_{n_1}} \cos  ((z_{n_0}+z_{n_1})r )  \sqrt{z_{n_2} z_{n_3}} z_{n_2}  J_1 (z_{n_2} r)  J_0 (z_{n_3} r)  \, dr \\
& \sim \frac{ \sqrt{z_{n_2} z_{n_3}} z_{n_2}}{z_{n_0} -z_{n_1}} \int_{\frac{1}{z_{n_2}}}^{\frac{1}{z_{n_3}}} \cos  ((z_{n_0}+z_{n_1})r ) \frac{\sin (z_{n_2} r - \frac{\pi}{4})}{\sqrt{z_{n_2} r}} \, dr \\
& \lesssim \frac{ \sqrt{z_{n_3}} z_{n_2}}{z_{n_0} -z_{n_1}}  \int_{\frac{1}{z_{n_2}}}^{\frac{1}{z_{n_3}}} r^{-\frac{1}{2}} \cos  ((z_{n_0}+z_{n_1})r ) \sin (z_{n_2} r - \frac{\pi}{4})  \, dr \\
& \lesssim \frac{ \sqrt{z_{n_3}} z_{n_2}}{z_{n_0} -z_{n_1}} \frac{z_{n_2}^{\frac{1}{2}}}{z_{n_0}+ z_{n_1}}  \sim \frac{ n_2^{\frac{3}{2}} n_3^{\frac{1}{2}}}{n_0^2 -n_1^2} ,
\end{align*}
where in the last inequality, we used Lemma  \ref{lem xpsincos}  to obtain
\begin{align*}
\int_{\frac{1}{z_{n_2}}}^{\frac{1}{z_{n_3}}} r^{-\frac{1}{2}} \cos  ((z_{n_0}+z_{n_1})r ) \sin (z_{n_2} r - \frac{\pi}{4})  \, dr \lesssim \frac{z_{n_2}^{\frac{1}{2}}}{z_{n_0} + z_{n_1}}   .
\end{align*}

\noindent {\bf Case V:} $\frac{1}{z_{n_3}} \leq r \leq 1$.

In the last case, we need asymptotic approximations for all four terms, and write
\begin{align*}
& \quad \int_{\frac{1}{z_{n_3}}}^1 \frac{1}{z_{n_0} -z_{n_1}} \cos  ((z_{n_0}+z_{n_1})r )  \sqrt{z_{n_2} z_{n_3}} z_{n_2}  J_1 (z_{n_2} r)  J_0 (z_{n_3} r)  \, dr \\
& \sim \frac{ \sqrt{z_{n_2} z_{n_3}} z_{n_2}}{z_{n_0} -z_{n_1}}  \int_{\frac{1}{z_{n_3}}}^1 \cos  ((z_{n_0}+z_{n_1})r ) \frac{\sin (z_{n_2} r - \frac{\pi}{4})}{\sqrt{z_{n_2} r}} \frac{\cos (z_{n_3} r - \frac{\pi}{4})}{\sqrt{z_{n_3} r} } \, dr \\
& \lesssim \frac{  z_{n_2}}{z_{n_0} -z_{n_1}}   \int_{\frac{1}{z_{n_3}}}^1  r^{-1} \cos  ((z_{n_0}+z_{n_1})r ) \cos ((z_{n_2} +z_{n_3})r) \, dr \\
& \sim \frac{  z_{n_2}}{z_{n_0} -z_{n_1}}   \frac{  z_{n_3}}{z_{n_0} +z_{n_1}} \sim \frac{n_2n_3}{n_0^2 -n_1^2} ,
\end{align*}
where in the last inequality, we used again Lemma  \ref{lem xpsincos}  to obtain
\begin{align*}
 \int_{\frac{1}{z_{n_3}}}^1  r^{-1} \cos  ((z_{n_0}+z_{n_1})r ) \cos ((z_{n_2} +z_{n_3})r) \, dr \lesssim \frac{z_{n_3}}{z_{n_0} + z_{n_1}} .
\end{align*}

Gathering the bounds in all these five cases, we see that {\it (4)} follows, hence the proof of Claim \ref{claim e_n} is complete.
\end{proof}
Hence we finish the proof of Lemma \ref{lem weak e_n}.
\end{proof}

Let us move on to the proof of Lemma \ref{lem weak fn}.
\begin{proof}[Proof of Lemma \ref{lem weak fn}]
Using Plancherel and convolution theorem, we write
\begin{align*}
A & := \int_{\R} \overline{w_{n_0} (t)}   f_{n_1}(t) \overline{g_{n_2}(t)} h_{n_3}(t) \,  dt\\
& = \int_{\R} \widehat{\overline{w}_{n_0} }(\tau)   \parenthese{ f_{n_1} \overline{g}_{n_2} h_{n_3} }^{\wedge} (\tau)\,  d\tau\\
& = \int_{\R} \widehat{\overline{w}_{n_0} }(\tau)  \int_{\tau_1 - \tau_2 + \tau_3 = \tau} \widehat{f_{n_1}}(\tau_1) \widehat{ \overline{g}_{n_2}}(\tau_2) \widehat{h_{n_3}}(\tau_3) \, d \tau_1 d \tau_2 d \tau_3   d\tau  .
\end{align*}

To make up $X^{0,b}$ norms out of $A$, we introduce the following new notations
\begin{align}\label{eq wfgh}
\begin{aligned}
\widetilde{w}_{n_0}(\tau , n_0) & = \widehat{\overline{w}_{n_0} }(\tau)  \inner{\tau - n_0^{2\alpha}}^{b}  & \widetilde{f}_{n_1} (\tau_1 , n_1) & = \widehat{f_{n_1}}(\tau_1)  \inner{\tau_1 - n_1^{2\alpha}}^{b} \\
\widetilde{g}_{n_2} (\tau_2 , n_2) & = \widehat{ \overline{g}_{n_2}}(\tau_2) \inner{\tau_2 - n_2^{2\alpha}}^{b} & \widetilde{h}_{n_3} (\tau_3 , n_3) & = \widehat{h_{n_3}}(\tau_3)  \inner{\tau_3 - n_3^{2\alpha}}^{b} .
\end{aligned}
\end{align}
Note that the $L_{n_i}^2 L_{\tau}^2$ norms of the functions above are in fact the $X^{0,b}$ norms of $w, f, g,h$.

Then using the new notations and H\"older inequality, we obtain
\begin{align*}
A & = \int_{\R} \int_{\tau_1 - \tau_2 + \tau_3 = \tau} \frac{\widetilde{w}_{n_0} (\tau , n_0)  }{\inner{\tau - n_0^{2\alpha}}^{b} }  \,  \frac{\widetilde{f}_{n_1} (\tau_1 , n_1) }{\inner{\tau_1 - n_1^{2\alpha}}^{b}} \,  \frac{\widetilde{g}_{n_2}  (\tau_2 , n_2)}{\inner{\tau_2 - n_2^{2\alpha}}^{b}}  \,  \frac{\widetilde{h}_{n_3}  (\tau_3 , n_3)}{ \inner{\tau_3 - n_3^{2\alpha}}^{b}} \, d \tau_1 d \tau_2 d \tau_3   d\tau \\
& \lesssim \norm{\widetilde{w}_{n_0} (\tau , n_0) \widetilde{f}_{n_1} (\tau_1 , n_1) \widetilde{g}_{n_2}  (\tau_2 , n_2) \widetilde{h}_{n_3}  (\tau_3 , n_3)}_{L_{\tau}^1 L_{\tau_1, \tau_2,\tau_3 ({\tau_1 - \tau_2 + \tau_3 = \tau})}^2}  \\
& \quad \times \norm{\frac{1}{\inner{\tau - n_0^2}^{b} \inner{\tau_1 - n_1^2}^{b} \inner{\tau_2 - n_2^2}^{b} \inner{\tau_3 - n_3^2}^{b} }}_{L_{\tau}^{\infty} L_{\tau_1, \tau_2,\tau_3 ({\tau_1 - \tau_2 + \tau_3 = \tau}) }^2} \\
& =:A_1 \times A_2 .
\end{align*}

For $A_1$ term, by Cauchy-Schwarz inequality, Young's convolution inequality and Definition \ref{defn Xsb}, we get
\begin{align*}
A_1 & = \int_{\R}  \widetilde{w}_{n_0} (\tau , n_0) \parenthese{ \int_{\tau_1 - \tau_2 + \tau_3 = \tau} \parenthese{ \widetilde{f}_{n_1} (\tau_1 , n_1) \widetilde{g}_{n_2}  (\tau_2 , n_2) \widetilde{h}_{n_3}  (\tau_3 , n_3)}^2 \, d \tau_1 d \tau_2 d \tau_3 }^{\frac{1}{2}}  d\tau\\
& \lesssim \norm{\widetilde{w}_{n_0} }_{L_{\tau}^2} \norm{\parenthese{\widetilde{f}_{n_1}^2 *  \widetilde{g}_{n_2}^2 * \widetilde{h}_{n_3}^2 }^{\frac{1}{2}}}_{L_{\tau }^2} \\
& = \norm{\widetilde{w}_{n_0} }_{L_{\tau}^2} \norm{\widetilde{f}_{n_1}^2 *  \widetilde{g}_{n_2}^2 * \widetilde{h}_{n_3}^2 }_{L_{\tau }^1}^{\frac{1}{2}} \\
& \lesssim \norm{\widetilde{w}_{n_0} }_{L_{\tau}^2} \norm{\widetilde{f}_{n_1}}_{L_{\tau }^2} \norm{ \widetilde{g}_{n_2}}_{L_{\tau }^2}  \norm{ \widetilde{h}_{n_3}}_{L_{\tau }^2} \\
& = \norm{w}_{X^{0, b}} \norm{f}_{X^{0, b}} \norm{ g}_{X^{0, b}}  \norm{h}_{X^{0, b}} .
\end{align*}

For the $A_2$ term, a quick observation gives  the integrability of the integrand 
\begin{align*}
A_2 \lesssim 1 ,
\end{align*}
since $b = \frac{1}{2}+$. 

To be more precise, we estimate of the term $A_2$ using the following lemma from \cite{DET}. See also the related treatment for similar integrals in \cite{KPV}.
\begin{lem}[Lemma 1 in \cite{DET}]
If $\gamma \geq 1$, then
\begin{align*}
\int_{\R} \frac{1}{\inner{\tau -k_1}^{\gamma} \inner{\tau- k_2}^{\gamma}} \, d \tau \lesssim \inner{k_1 - k_2}^{-\gamma} .
\end{align*}
\end{lem}

Continuing the computation using the summing lemma above, we arrive at
\begin{align*}
A_2 & =  \sup_{\tau} \parenthese{ \int_{\tau_1 - \tau_2 + \tau_3 = \tau} \frac{1}{\inner{\tau - n_0^{2\alpha}}^{2b} \inner{\tau_1 - n_1^{2\alpha}}^{2b} \inner{\tau_2 - n_2^{2\alpha}}^{2b} \inner{\tau_3 - n_3^{2\alpha}}^{2b} } \, d \tau_1 d \tau_2 d \tau_3  }^{\frac{1}{2}} \\
& \lesssim   \parenthese{ \int \frac{1}{\inner{ \tau_2 - \tau_3 + n_0^{2\alpha} - n_1^{2\alpha}}^{2b}  \inner{\tau_2 - n_2^{2\alpha}}^{2b} \inner{\tau_3 - n_3^{2\alpha}}^{2b} } \,  d \tau_2 d \tau_3  }^{\frac{1}{2}} \\
& \lesssim   \parenthese{ \int \frac{1}{\inner{  \tau_3 - n_0^{2\alpha} + n_1^{2\alpha} - n_2^{2\alpha}}^{2b}  \inner{\tau_3 - n_3^{2\alpha}}^{2b} } \,  d \tau_3  }^{\frac{1}{2}} \\
& \lesssim    \frac{1}{\inner{   n_0^{2\alpha} - n_1^{2\alpha} + n_2^{2\alpha} - n_3^{2\alpha}}^{b}}   .
\end{align*}
Now putting the calculation on $A_1$ and $A_2$ together, we finish the proof of Lemma \ref{lem weak fn}.
\end{proof}

\begin{rmk}[Variations on assumptions in Proposition \ref{prop weak}]\label{rmk weak}
\begin{enumerate}
\item
In fact, if assuming $n_0 \geq n_1 \geq n_2 \geq n_3$ and $n_0 \geq 2 n_3$ in Proposition \ref{prop weak}, we should expect a very similar bound with slight modification in the proof of Claim \ref{claim e_n}
\begin{align}\label{eq J bdd'}
J(\underline{n}) = \int_{\R \times \Theta} \overline{w}  \cdot f  \cdot \overline{g} \cdot h \, dx dt  \lesssim \frac{n_1^{\frac{3}{2}} n_2^{\frac{1}{2}}}{n_0^2} \frac{\norm{w}_{X^{0, b}} \norm{f}_{X^{0, b}} \norm{g}_{X^{0, b}}  \norm{h}_{X^{0, b}} }{ \inner{   n_0^{2\alpha} - n_1^{2\alpha} + n_2^{2\alpha} - n_3^{2\alpha}}^{b}}   . 
\end{align}

\item
Instead of $X^{0,b}$, we can take $w \in L_{t,x}^2$, in which case, the only difference will be in the change of variables in \eqref{eq wfgh}. In fact, by not touching on $\widetilde{w}_{n_0}(\tau , n_0)  = \widehat{\overline{w}_{n_0} }(\tau) $, the rest of the argument follows perfectly, but resulting the appearance of $L_{t,x}^2$ norm of $w$ in the bound
\begin{align}\label{eq J bdd''}
J(\underline{n}) = \int_{\R \times \Theta} \overline{w}  \cdot f  \cdot \overline{g} \cdot h \, dx dt  \lesssim \frac{n_2^{\frac{3}{2}} n_3^{\frac{1}{2}}}{n_0^2} \norm{w}_{L_{t,x}^2} \norm{f}_{X^{0, b}} \norm{g}_{X^{0, b}}  \norm{h}_{X^{0, b}}    . 
\end{align}
\end{enumerate}
\end{rmk}

\section{Energy increment}\label{sec energy increment}
In this section, we compute the energy increment of the I-operator modified equation on a short time interval. This will be the key ingredient in this iterative argument in Section \ref{sec gwp}.

\begin{prop}[Energy increment]\label{prop energy increment}
Consider $u$ as in \eqref{fNLSI} with $\alpha \in (\frac{1}{2} ,1]$  defined on  $[0, \delta ] \times \Theta$, then for $s > \frac{1}{2}$ and sufficiently large  $N$, the solution $u$ satisfies the following energy increment
\begin{align*}
E(Iu(\delta)) -E(Iu(0))  & \lesssim   N^{\frac{1}{2} - \alpha+} N^{-\frac{4(\alpha -s)(b- b(\alpha-) )}{1 + 2b -4b(s)}}    \norm{Iu_0}_{H^{\alpha}}^4 +   N^{2- 3\alpha- }  \norm{Iu_0}_{H^{\alpha}}^4   \\
& \quad + N^{ 2-3\alpha}  N^{-\frac{\alpha -s}{1 + 2b -4b(s)}}     \norm{Iu_0}_{H^{\alpha}}^6 + N^{ \frac{7}{2} -6\alpha +}     \norm{Iu_0}_{H^{\alpha}}^6 + N^{ 2-4\alpha +}  \norm{Iu_0}_{H^{\alpha}}^6 .
\end{align*}
\end{prop}

\begin{proof}[Proof of Proposition \ref{prop energy increment}]
We first start with writing the energy conservation
\begin{align*}
\frac{d}{dt} E( u(t)) & = \re \int_{\Theta} \overline{u_t} (\abs{u}^2 u + (-\Delta)^{\alpha} u) \, dx  = \re \int_{\Theta} \overline{u_t} (\abs{u}^2 u + (-\Delta)^{\alpha} u - i u_t) \, dx = 0 .
\end{align*}
Similarly we can compute the rate of change in the energy of the modified equation \eqref{fNLSI}. 
\begin{align*}
\frac{d}{dt} E( I u(t)) &  = \re \int_{\Theta} \overline{Iu_t} (\abs{Iu}^2 Iu + (-\Delta)^{\alpha} Iu - i Iu_t) \, dx  = \re \int_{\Theta} \overline{Iu_t} (\abs{Iu}^2 Iu - I(\abs{u}^2 u) ) \, dx \\
& = \im \int_{\Theta} \overline{(-\Delta)^{\alpha} I u} (\abs{Iu}^2 Iu - I(\abs{u}^2 u) ) \, dx  + \im \int_{\Theta} \overline{ I (\abs{u}^2 u) } (\abs{Iu}^2 Iu - I(\abs{u}^2 u) ) \, dx .
\end{align*}
Then by the fundamental theorem of calculus, we obtain
\begin{align}
& \quad E(Iu)(\delta) -  E(Iu)(0) \notag\\
& = \im \int_0^{\delta} \int_{\Theta} \overline{(-\Delta)^{\alpha} I u} (\abs{Iu}^2 Iu - I(\abs{u}^2 u) ) \, dx dt +  \im \int_0^{\delta} \int_{\Theta} \overline{ I (\abs{u}^2 u) } (\abs{Iu}^2 Iu - I(\abs{u}^2 u) ) \, dx  dt\notag \\
& = : \text{Term I} + \text{Term II} . \label{eq 2term}
\end{align}
To conclude the energy increment in this proposition, we just need to estimate the two terms in \eqref{eq 2term}.

\subsection{Estimate on Term I}

First we decompose each $u$ in Term I as we did in \eqref{eq u decomp} in Lemma \ref{lem nonlinear est}, then write for the quadruple $\underline{N} = (N_0, N_1, N_2, N_3)$, 
\begin{align*}
\text{Term I} \sim \sum_{\underline{N}} \int_0^{\delta} \int_{\Theta}  \overline{(-\Delta)^{\alpha} I  u_{0}}   (Iu_{1}  \overline{Iu_{2}} Iu_{3}  - I (u_1 \overline{u_2} u_3) ) \, dxdt ,
\end{align*}
where $u_{i} = \sum_{N_i \leq \inner{z_n} < 2N_i} P_n u$, $i = 0, 1 ,2, 3$.

Let
\begin{align*}
\text{Term I}  (\underline{N}) : = \int_0^{\delta} \int_{\Theta}  \overline{(-\Delta)^{\alpha}  I u_{0}}   (Iu_{1}  \overline{Iu_{2}} Iu_{3}  - I (u_1 \overline{u_2} u_3) ) \, dxdt .
\end{align*}
Without loss of generality, we assume $N_1 \geq N_2 \geq N_3$, and analyze the following different scenarios. Let us outline the cases that we will be considering.
\begin{enumerate}[-]
\item
{\bf Case I-1}: trivial cases;
\item
{\bf Case I-2}: the maximum frequency is much larger than the second highest frequency;
\item
The largest two frequencies are comparable;
\begin{enumerate}[$*$]
\item
{\bf Case I-3}: the largest two frequencies are $N_1 = N_0$,
\item
{\bf Case I-4}: the largest two frequencies are $N_1 =N_2$.
\end{enumerate}
\end{enumerate}

\noindent {\bf Case I-1:} The trivial cases. After the decomposition in frequencies, we have the following  two trivial cases:
\begin{itemize}
\item
All frequencies are not comparable to $N$, that is $N_0, N_1 , N_2, N_3 \ll N$,
\item
$N_0 = N_1 \geq N \geq N_2 \geq N_3$.
\end{itemize}
In  both cases, we have
\begin{align*}
\text{Term I}  (\underline{N}) =0 ,
\end{align*}
hence
\begin{align*}
\sum_{\underline{N} \in \textbf{ Case I-1}} \text{Term I}  (\underline{N}) =0 .
\end{align*}

Now we focus on the regime where at least the maximum frequency is larger than $N$, that is $\max \{ N_0, N_1 , N_2, N_3 \} \gtrsim N$. 

Using the abuse the notation in Remark \ref{rmk abuse}, we write
\begin{align}
\text{Term I}  (\underline{N}) &  \leq  \abs{ \int_0^{\delta} \int_{\Theta}  \overline{(-\Delta)^{\alpha} I  u_{0}}   (Iu_{1}  \overline{Iu_{2}} Iu_{3}  ) \, dxdt} + \abs{\int_0^{\delta} \int_{\Theta}  \overline{(-\Delta)^{\alpha} I  u_{0}}    I (u_1 \overline{u_2} u_3)  \, dxdt} \notag\\
& \lesssim  \abs{ \int_0^{\delta} \int_{\Theta}  \overline{(-\Delta)^{\alpha}  I u_{0}}   (Iu_{1}  \overline{Iu_{2}} Iu_{3}  ) \, dxdt} +  \frac{m (N_0)}{m (N_1) m (N_2) m (N_3)}  \abs{\int_0^{\delta} \int_{\Theta}  \overline{(-\Delta)^{\alpha} I  u_{0}}    (I u_1 \overline{I u_2} Iu_3)  \, dxdt} \notag\\
& = (1+ \frac{m (N_0)}{m (N_1) m (N_2) m (N_3)})  \abs{\int_0^{\delta} \int_{\Theta}  \overline{(-\Delta)^{\alpha} I  u_{0}}    (I u_1 \overline{I u_2} Iu_3)  \, dxdt} \notag \\
&  = :  M (\underline{N}) \times \text{Term I}' (\underline{N})  . \label{eq term I}
\end{align}

\noindent {\bf Case I-2:} The maximum frequency is much larger than the second highest frequency.

\noindent {\bf Case I-2a:} $\max \{ n_0, n_1 , n_2, n_3 \} =n_0  \gtrsim N$ and $n_0 \gg n_1$.

Take $\textbf{Term I}' (\underline{N})$ in \eqref{eq term I} first. 
Applying  the weak interaction between frequency localized functions in Proposition \ref{prop weak} (where we take $w = P_{n_0} I u$, $f = P_{n_1} I u$, $g =P_{n_2} I u$ and $h =P_{n_3} I u$) and taking out the derivative on $P_{n_0} u$, we are able to write 
\begin{align*}
\text{Term I}' (\underline{n})  \lesssim n_0^{2\alpha} \frac{n_2^{\frac{3}{2}} n_3^{\frac{1}{2}}}{n_0^2} \frac{\norm{P_{n_0} I u}_{X^{0, b}} \norm{P_{n_1} I u}_{X^{0, b}} \norm{P_{n_2} I u}_{X^{0, b}}  \norm{P_{n_3} I u}_{X^{0, b}} }{ \inner{   n_0^{2\alpha} - n_1^{2\alpha} + n_2^{2\alpha} - n_3^{2\alpha}}^{b}}   ,
\end{align*}
where $b = \frac{1}{2}+$. Note that  in the rest of the proof we will constantly use the notation $b = \frac{1}{2}+$ for simplicity.

Then we estimate $M(\underline{n})$ by 
\begin{align}\label{eq Mn}
M(\underline{n}) \lesssim 
\begin{cases}
\parenthese{\frac{N}{n_0}}^{\alpha -s} & \text{ if } n_3 \leq n_2 \leq  n_1 \leq N \lesssim n_0\\
\parenthese{\frac{n_1}{n_0}}^{\alpha -s} &  \text{ if } n_3 \leq n_2 \leq  N  \leq n_1 \lesssim n_0\\
\parenthese{\frac{n_1 n_2}{n_0 N}}^{\alpha -s} &  \text{ if } n_3 \leq N \leq  n_2  \leq n_1 \lesssim n_0\\
\parenthese{\frac{n_1 n_2 n_3}{n_0 N^2}}^{\alpha -s}  &  \text{ if } N \leq n_3 \leq  n_2  \leq n_1 \lesssim n_0 .
\end{cases}
\end{align}
Now summing over all $n_i$ using Cauchy-Schwarz inequality, Bernstein inequality and Definition \ref{defn Xsb}, we have
\begin{align*}
\sum_{\underline{n}} \text{Term I} (\underline{n})  & \lesssim \sum_{\underline{n} } M(\underline{n}) \cdot  n_0^{2\alpha} \frac{n_2^{\frac{3}{2}} n_3^{\frac{1}{2}} }{n_0^2} \frac{\norm{P_{n_0} I u}_{X^{0, b}} \norm{P_{n_1} I u}_{X^{0, b}} \norm{P_{n_2} I u}_{X^{0, b}}  \norm{P_{n_3} I u}_{X^{0, b}} }{ \inner{   n_0^{2\alpha} - n_1^{2\alpha} + n_2^{2\alpha} - n_3^{2\alpha}}^{b}}  \\
& \lesssim \parenthese{\sum_{\underline{n}}  \parenthese{\frac{n_2^{\frac{3}{2}} n_3^{\frac{1}{2}}}{n_0^2} \frac{M(\underline{n}) \cdot  n_0^{2\alpha} }{n_0^{\alpha} n_1^{\alpha} n_2^{\alpha} n_3^{\alpha}  \inner{   n_0^{2\alpha} - n_1^{2\alpha} + n_2^{2\alpha} - n_3^{2\alpha}}^{b}} }^2}^{\frac{1}{2}}  \norm{Iu}_{X_{\delta}^{\alpha, b}}^4 .
\end{align*}
Notice that $n_0 \gg n_1$ implies $\inner{   n_0^{2\alpha} - n_1^{2\alpha} + n_2^{2\alpha} - n_3^{2\alpha}} \gtrsim \inner{n_0^{2\alpha}} $. 
Then we compute the sum. In the first case $M(\underline{n}) = \parenthese{\frac{N}{n_0}}^{\alpha -s} $, we have
\begin{align*}
& \quad \sum_{\underline{n}}   \parenthese{ \frac{n_2^{\frac{3}{2}} n_3^{\frac{1}{2}}}{n_0^2}   \parenthese{  \frac{N}{n_0}}^{\alpha-s} \frac{ n_0^{2\alpha}}{n_0^{\alpha} n_1^{\alpha} n_2^{\alpha} n_3^{\alpha}  \inner{   n_0^{2\alpha} - n_1^{2\alpha} + n_2^{2\alpha} - n_3^{2\alpha}}^{b}} }^2 \lesssim \sum_{\underline{n}}    \frac{n_2^3 n_3}{n_0^4} \frac{ N^{2\alpha -2s} \cdot n_0^{2\alpha}}{ n_0^{2\alpha -2s} n_1^{2\alpha} n_2^{2\alpha} n_3^{2\alpha}  \inner{   n_0^{2\alpha} }^{1+}}  \\
& \lesssim N^{2\alpha -2s} \sum_{n_0} n_0^{-4-2\alpha +2s -}\sum_{n_1} n_1^{-2\alpha}\sum_{n_2} n_2^{3-2\alpha} \sum_{n_3} n_3^{1-2\alpha}\\
& \lesssim  N^{2\alpha -2s} \sum_{n_0} n_0^{3 - 8 \alpha +2s- } \lesssim N^{4-6\alpha -} .
\end{align*}
The same bound holds for  other cases of $M(\underline{n})$ in \eqref{eq Mn}. 

Therefore, 
\begin{align*}
\sum_{\underline{n} \in \textbf{ Case I-2a}} \text{Term I}  (\underline{n}) \lesssim N^{2-3\alpha - } \norm{Iu}_{X_{\delta}^{\alpha, b}}^4 .
\end{align*}

\noindent {\bf Case I-2b:} $\max \{ n_0, n_1 , n_2, n_3 \} =n_1 \gtrsim N $ and $n_1 \gg  \max \{ n_0 , n_2 \}$. 

This case is in fact similar to the previous  {\bf Case I-2a}. Using 
\begin{align}\label{eq Mn2}
M(\underline{n}) \lesssim 
\begin{cases}
\parenthese{\frac{n_1}{N}}^{\alpha -s} & \text{ if } n_3 , n_2 , n_0 \leq N \lesssim n_1 \\
\parenthese{\frac{n_1}{n_0}}^{\alpha -s} &  \text{ if } n_3 \leq n_2 \leq  N  \lesssim n_0 \leq n_1\\
\parenthese{\frac{n_1 n_2}{ N^2}}^{\alpha -s} &  \text{ if } n_3, n_0 \leq N \lesssim  n_2  \leq n_1 \\
\parenthese{\frac{n_1 n_2 n_0}{ N^3}}^{\alpha -s}  &  \text{ if }  n_3 \leq  N \lesssim n_2  \leq n_1, N \leq n_0 \\
\parenthese{\frac{n_1 n_2 n_3}{N^3}}^{\alpha -s}  &  \text{ if }  n_0 \leq  N \lesssim n_3 \leq n_2  \leq n_1 \\
\parenthese{\frac{n_1 n_2 n_3}{n_0 N^2}}^{\alpha -s}  &  \text{ if }   N \lesssim n_3 \leq n_2  \leq n_1, N \leq n_0 .
\end{cases}
\end{align}
and similar calculation  in {\bf Case I-2a}, we see that
\begin{align*}
 \sum_{\underline{n} \in \textbf{  Case I-2b}} \text{Term I}  (\underline{n}) \lesssim   N^{2- 3\alpha- } \norm{Iu}_{X_{\delta}^{\alpha, b}}^4 .
\end{align*}
Therefore, by Cauchy-Schwarz inequality and Definition \ref{defn Xsb} we have in {\bf  Case I-2}
\begin{align*}
 \sum_{\underline{N} \in \textbf{  Case I-2}} \text{Term I}  (\underline{N}) \lesssim   N^{2- 3\alpha- } \norm{Iu}_{X_{\delta}^{\alpha, b}}^4 .
\end{align*}

Now we focus on the case when the largest two frequencies are comparable. Under our assumption on $N_0 , N_1 , N_2 ,N_3$, there are only following two possibilities and we will discuss them separately.
\begin{itemize}
\item
{\bf Case I-3:}  $N_1 = N_0 \gtrsim N$ and $N_1 \geq N_2 \geq N_3$;
\item
{\bf Case I-4:} $N_1 = N_2 \gtrsim N$ and $N_1 \geq N_0$. 
\end{itemize}

\noindent {\bf Case I-3:} The largest two frequencies are comparable. $N_1 = N_0 \gtrsim N$.

In this case, we write
\begin{align*}
M(\underline{N}) \lesssim \frac{m(N_0)}{ m(N_1) m (N_2) m (N_3)}  \sim \frac{1}{m (N_2) m (N_3)} .
\end{align*}
Recall the bilinear estimate that we obtained in \eqref{eq inter bilinear}, 
\begin{align*}
\norm{ f_1 f_2}_{L_{t,x}^2 ((0, \delta) \times \Theta)} \lesssim  N_2^{\beta} \delta^{2(b- b(\beta) )} \norm{f_1}_{X_{\delta}^{0, b} (\Theta)} \norm{f_2}_{X_{\delta}^{0, b} (\Theta)} .
\end{align*}
for  $b(\beta) =\frac{1}{4}+ (1-\beta)\frac{1}{2}+$, $\beta \in (\frac{1}{2} , 1]$.
Taking $\beta = \alpha$ and combining with $\delta \sim N^{-\frac{2(\alpha -s)}{1 + 2b -4b(s)}} $ in Proposition \ref{prop LWPI}, we have
\begin{align}\label{eq bi6}
\norm{ f_1 f_2}_{L_{t,x}^2 ((0,\delta) \times \Theta)} \lesssim  N_2^{\alpha} N^{-\frac{4(\alpha -s)(b- b(\alpha) )}{1 + 2b -4b(s)}}    \norm{f_1}_{X_{\delta}^{0, b} (\Theta)} \norm{f_2}_{X_{\delta}^{0, b} (\Theta)} .
\end{align}

\noindent {\bf Case I-3a:} $N_1 = N_0 \gtrsim N \geq N_2 \geq N_3$.
This is trivial as showed in {\bf Case I-1}.

\noindent {\bf Case I-3b:} $N_1 = N_0 \geq N_2 \gtrsim N \geq N_3$.

Using 
\begin{align*}
M(\underline{N})  \lesssim (N^{-1} N_2)^{\alpha -s} 
\end{align*}
with H\"older inequality, Proposition \ref{prop bilinear}, \eqref{eq bi6} and Bernstein inequality, we have
\begin{align*}
\text{Term I}  (\underline{N}) & \lesssim N_0^{2\alpha} (N^{-1} N_2)^{\alpha -s} \norm{Iu_0 Iu_2}_{L_{t,x}^2 ([0,\delta] \times \Theta)} \norm{Iu_1 Iu_3}_{L_{t,x}^2 ([0,\delta] \times \Theta)} \\
& \lesssim N_0^{2\alpha} (N^{-1} N_2)^{\alpha -s} N_2^{\frac{1}{2}+} N_3^{\alpha} N^{-\frac{4(\alpha -s)(b- b(\alpha) )}{1 + 2b -4b(s)}}   \prod_{i=0}^3 \norm{Iu_i}_{X_{\delta}^{0 , b}  }  \\
& \lesssim  \frac{1}{N^{\alpha -s}}\frac{N_0^{2\alpha} N_2^{\frac{1}{2}+} N_3^{\alpha-}}{N_0^{\alpha} N_1^{\alpha} N_2^{\alpha} N_3^{\alpha}}  N^{-\frac{4(\alpha -s)(b- b(\alpha-) )}{1 + 2b -4b(s)}}   \prod_{i=0}^3 \norm{Iu_i}_{X_{\delta}^{\alpha , b}  } \\
& \lesssim  N^{\frac{1}{2} - 2\alpha+s+} N^{-\frac{4(\alpha -s)(b- b(\alpha-) )}{1 + 2b -4b(s)}}  N_2^{0-} N_3^{0-} \parenthese{\frac{N_0}{ N_1}}^{\alpha}  \prod_{i=0}^3 \norm{Iu_i}_{X_{\delta}^{\alpha , b}  }  .
\end{align*}
Therefore, by Cauchy-Schwarz inequality and Definition \ref{defn Xsb} we have
\begin{align*}
 \sum_{\underline{N} \in \textbf{  Case I-3b}} \text{Term I}  (\underline{N}) \lesssim   N^{\frac{1}{2} - 2\alpha+s +} N^{-\frac{4(\alpha -s)(b- b(\alpha-) )}{1 + 2b -4b(s)}}     \norm{Iu}_{X_{\delta}^{\alpha , b}}^3  .
\end{align*}

\noindent {\bf Case I-3c:} $N_1 = N_0 \geq N_2  \geq N_3 \gtrsim N$.

Using 
\begin{align*}
M(\underline{N}) \lesssim  (N^{-2} N_2 N_3)^{\alpha -s} 
\end{align*}
H\"older inequality, Proposition \ref{prop bilinear} and Bernstein inequality, we have
\begin{align*}
\text{Term I}  (\underline{N}) & \lesssim N_0^{2\alpha} (N^{-2} N_2 N_3)^{\alpha -s} \norm{Iu_0 Iu_2}_{L_{t,x}^2 ([0,\delta] \times \Theta)} \norm{Iu_1 Iu_3}_{L_{t,x}^2 ([0,\delta] \times \Theta)} \\
& \lesssim N_0^{2\alpha}  (N^{-2} N_2 N_3)^{\alpha -s}(N_2 N_3)^{\frac{1}{2}+}  \prod_{i=0}^3 \norm{Iu_i}_{X_{\delta}^{0 , b}  }  \\
& \lesssim N^{-2(\alpha -s)} \frac{N_0^{2\alpha} (N_2 N_3)^{\frac{1}{2}+} (N_2 N_3)^{\alpha -s}   }{N_0^{\alpha} N_1^{\alpha} N_2^{\alpha} N_3^{\alpha}}   \prod_{i=0}^3 \norm{Iu_i}_{X_{\delta}^{\alpha , b}  }  \\
& \lesssim  N^{1-2 \alpha+}  N_2^{0-} N_3^{0-} \parenthese{ \frac{N_0}{N_1}}^{\alpha} \prod_{i=0}^3 \norm{Iu_i}_{X_{\delta}^{\alpha , b}  }   .
\end{align*}
Therefor, by Cauchy-Schwarz inequality and Definition \ref{defn Xsb} we have
\begin{align*}
 \sum_{\underline{N} \in \textbf{  Case I-3c}} \text{Term I}  (\underline{N}) \lesssim   N^{1-2 \alpha+}  \norm{Iu}_{X_{\delta}^{\alpha , b}  }^4 .
\end{align*}

To sum up, the bound in { \bf Case I-3} is given by
\begin{align*}
 \sum_{\underline{N} \in \textbf{  Case I-3}} \text{Term I}  (\underline{N}) \lesssim    N^{\frac{1}{2} - 2\alpha+s +} N^{-\frac{4(\alpha -s)(b- b(\alpha-) )}{1 + 2b -4b(s)}}    \norm{Iu}_{X_{\delta}^{\alpha , b}  }^4 + N^{1-2 \alpha+}  \norm{Iu}_{X_{\delta}^{\alpha , b}  }^4 .
\end{align*}

\noindent {\bf Case I-4:} The largest two frequencies are comparable. $N_1 = N_2 \gtrsim N$ and $N_1 \geq N_0$. 

In this case, we write
\begin{align*}
M(\underline{N}) & \lesssim   \frac{m (N_0)}{m (N_1) m (N_2) m (N_3)}  \lesssim \frac{m (N_0)}{m (N_3)}  (N^{-2} N_1 N_2)^{\alpha -s} ,
\end{align*}
where
\begin{align}\label{eq Mn03}
\frac{m (N_0)}{m (N_3)} \lesssim 
\begin{cases}
1 & \text{ if } N_0 , N_3 \lesssim N\\
\parenthese{\frac{N}{N_0}}^{\alpha -s} & \text{ if } N_3 \lesssim N \lesssim N_0 \\
\parenthese{\frac{N_3}{N}}^{\alpha -s} & \text{ if } N_0 \lesssim N \lesssim N_3 \\
\parenthese{\frac{N_3}{N_0}}^{\alpha -s} & \text{ if } N \lesssim N_0 , N_3  .
\end{cases}
\end{align}
Take the first case  in \eqref{eq Mn03}, then by H\"older inequality, Proposition \ref{prop bilinear}, \eqref{eq bi6} and Bernstein inequality, we write
\begin{align*}
\text{Term I}  (\underline{N}) & \lesssim N_0^{2\alpha} (N^{-2} N_1 N_2)^{\alpha -s} \norm{Iu_0 Iu_2}_{L_{t,x}^2 ([0,\delta] \times \Theta)} \norm{Iu_1 Iu_3}_{L_{t,x}^2 ([0,\delta] \times \Theta)} \\
& \lesssim N_0^{2\alpha}  (N^{-2} N_1 N_2)^{\alpha -s}  N_0^{\frac{1}{2}+} N_3^{\alpha-} N^{-\frac{4(\alpha -s)(b- b(\alpha-) )}{1 + 2b -4b(s)}}     \prod_{i=0}^3 \norm{Iu_i}_{X_{\delta}^{0 , b}  }   \\
& \lesssim N^{-2(\alpha -s)}  \frac{N_0^{2\alpha}  (N_1 N_2 )^{\alpha -s} N_0^{\frac{1}{2}+} N_3^{\alpha-} }{N_0^{\alpha} N_1^{\alpha} N_2^{\alpha} N_3^{\alpha}} N^{-\frac{4(\alpha -s)(b- b(\alpha-) )}{1 + 2b -4b(s)}}     \prod_{i=0}^3 \norm{Iu_i}_{X_{\delta}^{\alpha , b}  }  \\
& \lesssim N^{\frac{1}{2}-\alpha+} N^{-\frac{4(\alpha -s)(b- b(\alpha-) )}{1 + 2b -4b(s)}}      N_0^{0-} N_1^{0-}  N_3^{0-} \prod_{i=0}^3 \norm{Iu_i}_{X_{\delta}^{\alpha , b}  } .
\end{align*}
The same bound hold for the second case in \eqref{eq Mn03}.

For the third  case  in \eqref{eq Mn03},  using  H\"older inequality, Proposition \ref{prop bilinear} and   Bernstein inequality, we have
\begin{align*}
\text{Term I}  (\underline{N}) & \lesssim N_0^{2\alpha} (N^{-2} N_1 N_2)^{\alpha -s} \parenthese{\frac{N_3}{N}}^{\alpha -s}  \norm{Iu_0 Iu_2}_{L_{t,x}^2 ([0,\delta] \times \Theta)} \norm{Iu_1 Iu_3}_{L_{t,x}^2 ([0,\delta] \times \Theta)} \\
& \lesssim N_0^{2\alpha}  (N^{-2} N_1 N_2)^{\alpha -s} \parenthese{\frac{N_3}{N}}^{\alpha -s}   N_0^{\frac{1}{2}+} N_3^{\frac{1}{2}+}    \prod_{i=0}^3 \norm{Iu_i}_{X_{\delta}^{0 , b}  }   \\
& \lesssim N^{-3(\alpha -s)}  \frac{N_0^{2\alpha}  (N_1 N_2 N_3)^{\alpha -s} N_0^{\frac{1}{2}+} N_3^{\frac{1}{2}+} }{N_0^{\alpha} N_1^{\alpha} N_2^{\alpha} N_3^{\alpha}}    \prod_{i=0}^3 \norm{Iu_i}_{X_{\delta}^{\alpha , b}  }  \\
& \lesssim N^{1-2\alpha+}  N_0^{0-} N_1^{0-}  N_3^{0-} \prod_{i=0}^3 \norm{Iu_i}_{X_{\delta}^{\alpha , b}  } .
\end{align*}
The same bound hold for the forth case in \eqref{eq Mn03}.

Therefor, by Cauchy-Schwarz inequality and Definition \ref{defn Xsb} we have that in {\bf Case I-4}.
\begin{align*}
 \sum_{\underline{N} \in \textbf{  Case I-4}} \text{Term I}  (\underline{N}) \lesssim   N^{\frac{1}{2}-\alpha+} N^{-\frac{4(\alpha -s)(b- b(\alpha-) )}{1 + 2b -4b(s)}}  \norm{Iu}_{X_{\delta}^{\alpha , b}  }^4 +  N^{1-2\alpha +}   \norm{Iu}_{X_{\delta}^{\alpha , b}  }^4  .
\end{align*}

Now we summarize the  estimation on  {\bf Term I}.
\begin{align*}
\text{Term I} &  \lesssim \parenthese{ \sum_{\underline{N} \in \textbf{  Case I-2}} + \sum_{\underline{N} \in \textbf{  Case I-3}} + \sum_{\underline{N} \in \textbf{  Case I-4}}  } \text{Term I} (\underline{N}) \\
&  \lesssim N^{\frac{1}{2} - \alpha+} N^{-\frac{4(\alpha -s)(b- b(\alpha-) )}{1 + 2b -4b(s)}}    \norm{Iu}_{X_{\delta}^{\alpha , b}  }^4 +   N^{2- 3\alpha- }  \norm{Iu}_{X_{\delta}^{\alpha , b}  }^4  .
\end{align*}

\subsection{Estimate on Term II}

Now let us focus on {\bf Term II}: 
\begin{align*}
\im \int_0^{\delta} \int_{\Theta} \overline{ I (\abs{u}^2 u) } (\abs{Iu}^2 Iu - I(\abs{u}^2 u) ) \, dx .
\end{align*}
Similarly, we decompose each $u$ in {\bf Term II} using the orthonomal basis $e_n$'s of the radial Dirichlet Laplacian $-\Delta$. That is, for the quadruple $\underline{N} = (N_0, N_1,  N_2 , N_3)$
\begin{align*}
\text{Term II} \sim \sum_{\underline{N}} \int_0^{\delta} \int_{\Theta}  \overline{I P_{N_0}(\abs{u}^2 u)}  (Iu_{1}  \overline{Iu_{2}} Iu_{3}  - I (u_1 \overline{u_2} u_3) ) \, dxdt ,
\end{align*}
where $u_{i} = \sum_{N_i \leq \inner{z_n} < 2N_i} P_n u$, $i = 1 ,2, 3$. 

Let 
\begin{align*}
\text{Term II} (\underline{N}) : =  \int_0^{\delta} \int_{\Theta}  \overline{I P_{N_0}(\abs{u}^2 u)}  (Iu_{1}  \overline{Iu_{2}} Iu_{3}  - I (u_1 \overline{u_2} u_3) ) \, dxdt .
\end{align*} 
Without loss of generality, we assume $N_1 \geq N_2 \geq N_3$. Let us outline the cases that we will be considering.
\begin{enumerate}[-]
\item
{\bf Case II-1}: trivial cases;
\item
{\bf Case II-2}: the maximum frequency is much larger than the minimum frequency;
\item
{\bf Case II-3}: all frequencies are comparable.
\end{enumerate}

Again, we start with the trivial cases.

\noindent {\bf Case II-1:} The trivial cases. We have two trivial cases as in {\bf Term I}
\begin{itemize}
\item
All frequencies are not comparable to $N$, that is $N_0, N_1 , N_2, N_3 \ll N$. 
\item
$N_0 = N_1 \geq N \geq N_2 \geq N_3$
\end{itemize}
In both cases, we have
\begin{align*}
\sum_{\underline{N} \in \textbf{ Case II-1}} \text{Term II}  (\underline{N}) =0 .
\end{align*}

This implies that at least the maximum frequency is larger than the cutoff frequency $N$, that is $\max \{ N_0, N_1 , N_2, N_3 \} \gtrsim N$.

Similar as in \eqref{eq term I}, we write  with abuse notation in Remark \ref{rmk abuse}.
\begin{align*}
\text{Term II}  (\underline{N}) &  \lesssim   (1+ \frac{m (N_0)}{m (N_1) m (N_2) m (N_3)})  \abs{\int_0^{\delta} \int_{\Theta}  \overline{I P_{N_0}(\abs{u}^2 u)}   (I u_1 \overline{I u_2} Iu_3)  \, dxdt}\\
&  = :  M (\underline{N}) \times \text{Term II}' (\underline{N})   .
\end{align*}
\noindent {\bf Case II-2:} The maximum frequency is much larger than minimum frequency. 

\noindent {\bf Case II-2a:} $\max \{ n_0, n_1 , n_2, n_3 \} = n_0 \gtrsim N $ and $n_0 \gg n_3$.

In this case, using the same estimate of $M(\underline{N})$ as in \eqref{eq Mn} and Remark \ref{rmk weak}, we write
\begin{align*}
\text{Term II} (\underline{n}) & \lesssim  M(\underline{n}) \frac{n_1^{\frac{3}{2}} n_2^{\frac{1}{2}}}{n_0^2}  \frac{1 }{n_1^{\alpha} n_2^{\alpha} n_3^{\alpha}}  \norm{I P_{n_0}(\abs{u}^2 u)}_{L_{t,x}^2}  \prod_{i=1}^3 \norm{P_{n_i}Iu}_{X^{\alpha, b}}  .
\end{align*}

In the following claim, we compute the term $\norm{I P_{n_0}(\abs{u}^2 u)}_{L_{t,x}^2} $ separately.
\begin{claim}\label{claim B}
\begin{align*}
B : =\norm{I P_{n_0}(\abs{u}^2 u)}_{L_{t,x}^2} & \lesssim m(n_0) B(N)  \norm{Iu}_{X_{\delta}^{\alpha , b}}^3,
\end{align*}
where
\begin{align*}
B(N) : = N^{-\frac{\alpha -s}{1 + 2b -4b(s)}}   +  N^{ \frac{3}{2} -3\alpha +} +  N^{-\alpha + } .
\end{align*}
\end{claim}

Assuming Claim \ref{claim B}, we can write
\begin{align*}
\text{Term II} (\underline{n}) & \lesssim  M(\underline{n}) \frac{n_1^{\frac{3}{2}} n_2^{\frac{1}{2}}}{n_0^2}  \frac{ m(n_0)}{n_1^{\alpha} n_2^{\alpha} n_3^{\alpha}}  \prod_{i=1}^3 \norm{P_{n_i}Iu}_{X^{\alpha, b}} \parenthese{ B(N) \norm{Iu}_{X_{\delta}^{\alpha , b}}^3 } .
\end{align*}
Using Cauchy-Schwarz inequality, \eqref{eq Mn} and Definition \ref{defn Xsb}, we obtain  that 
\begin{align*}
\sum_{\underline{n} \in \textbf{ Case II-2a}} \text{Term II} (\underline{n}) & \lesssim \parenthese{\sum_{\underline{n}}  \parenthese{\frac{n_1^{\frac{3}{2}} n_2^{\frac{1}{2}}}{n_0^2} \frac{M(\underline{n}) \cdot  m(n_0)}{n_1^{\alpha} n_2^{\alpha} n_3^{\alpha}  } }^2}^{\frac{1}{2}}  B(N) \norm{Iu}_{X_{\delta}^{\alpha, b}}^6 \\
& \lesssim  N^{\frac{3}{2} -2\alpha} B(N) \norm{Iu}_{X_{\delta}^{\alpha , b}}^6 .
\end{align*}

Now we present the calculation for Claim \ref{claim B}.
\begin{proof}[Proof of Claim \ref{claim B}]
Decompose $B$ into dyadic frequencies, then we have
\begin{align*}
\norm{I P_{n_0}(u_4 \overline{u_5} u_6)}_{L_{t,x}^2}  \sim  \frac{m (n_0)}{m (N_4) m (N_5) m (N_6)} \norm{Iu_4 Iu_5 Iu_6}_{L_{t,x}^2} = : m(n_0) \times B(\underline{N_{456}})   .
\end{align*}
Let us focus on $B (\underline{N_{456}}) $. We fist rewrite it using H\"older inequality and Sobolev embedding
\begin{align*}
B(\underline{N_{456}})  & \lesssim \frac{1}{m (N_4) m (N_5) m (N_6)} \norm{Iu_4 Iu_5}_{L_{t,x}^2} \norm{Iu_6}_{L_{t,x}^{\infty}} \\
& \lesssim \frac{1}{m (N_4) m (N_5) m (N_6)} N_6^{1-\alpha} \norm{Iu_4 Iu_5}_{L_{t,x}^2} \norm{Iu_6}_{X_{\delta}^{\alpha , b}} .
\end{align*}
Without loss of generality, we assume $N_4 \geq N_5 \geq N_6$. 

Similar as the computation that we did earlier, we will consider the following two cases
\begin{enumerate}[-]
\item
{\bf Case B-1}: all frequencies are smaller than $N$;
\item
{\bf Case B-2}: at least on frequency is larger than $N$.
\end{enumerate}

\noindent {\bf Case B-1:} $N_6 \leq N_5 \leq N_4 \ll N$.

Using
\begin{align*}
\frac{1}{m (N_4) m (N_5) m (N_6)}  = 1,
\end{align*}
with \eqref{eq bi6} and Bernstein inequality, we obtain
\begin{align*}
B(\underline{N_{456}})  & \lesssim   N_5 N^{-\frac{\alpha -s}{1 + 2b -4b(s)}}    N_6^{1-\alpha} \norm{Iu_4}_{X_{\delta}^{0 ,b}} \norm{Iu_5}_{X_{\delta}^{0 , b}} \norm{Iu_6}_{X_{\delta}^{\alpha ,b}}  \\
& \lesssim  N^{-\frac{\alpha -s}{1 + 2b -4b(s)}}    \frac{N_5  N_6^{1-\alpha}}{N_4^{\alpha} N_5^{\alpha}} \prod_{i=4}^6 \norm{Iu_i}_{X_{\delta}^{\alpha , b}}  .
\end{align*}
Then by  Cauchy-Schwarz inequality with Definition \ref{defn Xsb},
\begin{align*}
\sum_{N_4, N_5 , N_6 \in \textbf{ Case B-1}} B(\underline{N_{456}})  \lesssim  NN^{-\frac{\alpha -s}{1 + 2b -4b(s)}}    \norm{Iu}_{X_{\delta}^{\alpha , b}}^3 .
\end{align*}

\noindent {\bf Case B-2:} $N \lesssim N_4$. 

We  compute 
\begin{align}\label{eq Mn3}
\frac{1}{m (N_4) m (N_5) m (N_6)} &  \lesssim 
\begin{cases}
(N^{-1}N_4)^{\alpha -s} & \text{ if } N_6 \leq N_5 \ll N \leq N_4\\
(N^{-2}N_4 N_5)^{\alpha -s} & \text{ if } N_6  \ll N \leq N_5 \leq N_4\\
(N^{-3}N_4 N_5 N_6)^{\alpha -s}  & \text{ if } N \ll N_6  \leq N_5 \leq N_4 .
\end{cases} 
\end{align}
Now take the first case in \eqref{eq Mn3}, with  the bilinear estimate in Proposition \ref{prop bilinear}, we then obtain
\begin{align*}
B(\underline{N_{456}}) &  \lesssim    (N^{-1}N_4)^{\alpha -s}  N_5^{\frac{1}{2} +}  N_6^{1-\alpha}   \norm{Iu_4}_{X_{\delta}^{0 ,b}} \norm{Iu_5}_{X_{\delta}^{0 , b}} \norm{Iu_6}_{X_{\delta}^{\alpha ,b}}  \\
&  \lesssim   (N^{-1}N_4 )^{\alpha -s}  \frac{ N_5^{\frac{1}{2} +} N_6^{1-\alpha} }{N_4^{\alpha} N_5^{\alpha}}  \prod_{i=4}^6 \norm{Iu_i}_{X_{\delta}^{\alpha , b}} \\
& \lesssim 
\begin{cases}
N^{\frac{3}{2}-3\alpha+} N_4^{0-} N_5^{0-} N_6^{0-}   \prod_{i=4}^6 \norm{Iu_i}_{X_{\delta}^{\alpha , b}} & \text{ if } \alpha < \frac{3}{4} \\
N^{-\alpha+} N_4^{0-} N_5^{0-} N_6^{0-}   \prod_{i=4}^6 \norm{Iu_i}_{X_{\delta}^{\alpha , b}} & \text{ if }   \alpha \geq \frac{3}{4}
\end{cases}\\
& = \max \{ N^{\frac{3}{2}-3\alpha+} , N^{-\alpha+}  \} N_4^{0-} N_5^{0-} N_6^{0-}   \prod_{i=4}^6 \norm{Iu_i}_{X_{\delta}^{\alpha , b}} .
\end{align*}
Similarly, the second and third cases are bounded receptively by
\begin{align*}
\max \{ N^{\frac{3}{2}-3\alpha+} , N^{-2\alpha+s+}  \} N_4^{0-} N_5^{0-} N_6^{0-}   \prod_{i=4}^6 \norm{Iu_i}_{X_{\delta}^{\alpha , b}} 
\end{align*}
and
\begin{align*}
\max \{ N^{\frac{3}{2}-3\alpha+} , N^{-3\alpha+2s+}  \} N_4^{0-} N_5^{0-} N_6^{0-}   \prod_{i=4}^6 \norm{Iu_i}_{X_{\delta}^{\alpha , b}}  .
\end{align*}

Then by  Cauchy-Schwarz inequality with Definition \ref{defn Xsb},
\begin{align*}
\sum_{N_4, N_5 , N_6 \in \textbf{ Case B-2}} B(\underline{N_{456}})  \lesssim  \max \{ N^{ \frac{3}{2} -3\alpha +}, N^{-\alpha + } \}  \norm{Iu}_{X_{\delta}^{\alpha , b}}^3 .
\end{align*}
Therefore, by putting the two cases together, we finish the proof of Claim \ref{claim B}.
\begin{align*}
\sum_{N_4, N_5 , N_6} B(\underline{N_{456}}) &  \lesssim \parenthese{\sum_{N_4, N_5 , N_6 \in \textbf{ Case B-1}} + \sum_{N_4, N_5 , N_6 \in \textbf{ Case B-2}} }  B(\underline{N_{456}}) \\
& \lesssim  N^{-\frac{\alpha -s}{1 + 2b -4b(s)}}  \norm{Iu}_{X_{\delta}^{\alpha , b}}^3 +  N^{ \frac{3}{2} -3\alpha +} \norm{Iu}_{X_{\delta}^{\alpha , b}}^3 +  N^{-\alpha + } \norm{Iu}_{X_{\delta}^{\alpha , b}}^3 .
\end{align*}
\end{proof}

\noindent {\bf Case II-2b:} $\max \{ n_0, n_1 , n_2, n_3 \} = n_1 \gtrsim N $ and $n_1 \gg \min \{ n_0, n_3 \}$

Writing 
\begin{align*}
M(\underline{n}) \lesssim   \frac{m (n_0)}{m (n_1) m (n_2) m (n_3)} 
\end{align*}
and using  Remark \ref{rmk weak} and Claim \ref{claim B}, we have
\begin{align*}
\text{Term II}  (\underline{N}) & \lesssim  M(\underline{n}) \frac{n_2^{\frac{3}{2}} n_3^{\frac{1}{2}}}{n_1^2}  \frac{1 }{n_1^{\alpha} n_2^{\alpha} n_3^{\alpha}}  \norm{I P_{n_0}(\abs{u}^2 u)}_{L_{t,x}^2}  \prod_{i=1}^3 \norm{Iu_i}_{X^{\alpha, b}} \\
& \lesssim M(\underline{n}) \frac{n_2^{\frac{3}{2}} n_3^{\frac{1}{2}}}{n_1^2}  \frac{m(n_0)}{n_1^{\alpha} n_2^{\alpha} n_3^{\alpha}}   \prod_{i=1}^3 \norm{Iu_i}_{X^{\alpha, b}}  \parenthese{ B(N)  \norm{Iu}_{X_{\delta}^{\alpha , b}}^3 } .
\end{align*}
Applying the analysis as in \eqref{eq Mn2} and same calculation as in {\bf Case II-2a}, we obtain the bound 
\begin{align*}
\sum_{\underline{n} \in \textbf{ Case II-2b}} \text{Term II} (\underline{n}) \lesssim N^{2 -3\alpha} B(N) \norm{Iu}_{X_{\delta}^{\alpha , b}}^6 .
\end{align*}

Now we have the only case left. 

\noindent {\bf Case II-3:} All frequencies are comparable. $N_0 = N_1  = N_2 = N_3 \gtrsim N$. 

In this last case, we have
\begin{align*}
M(\underline{N})  \lesssim  \frac{m (N_0)}{m (N_1) m (N_2) m (N_3)}  \sim \frac{1}{ m (N_2) m (N_3)}   \lesssim (N^{-2} N_2 N_3)^{\alpha -s} .
\end{align*}
Then using the $M(\underline{N})$ bound above, Proposition \ref{prop bilinear}  and Claim \ref{claim B}, we write
\begin{align*}
\text{Term II} (\underline{N}) & \lesssim (N^{-2} N_2 N_3)^{\alpha -s}  \abs{\int_0^{\delta} \int_{\Theta}  \overline{I P_{N_0}(\abs{u}^2 u)}   (I u_1 \overline{I u_2} Iu_3)  \, dxdt} \\
& \lesssim  (N^{-2} N_2 N_3)^{\alpha -s}  \norm{Iu_1 Iu_2 Iu_3}_{L_{t,x}^2}  \norm{I P_{N_0}(\abs{u}^2 u)}_{L_{t,x}^2}  \\
& \lesssim \parenthese{\frac{N_2 N_3}{N^2}}^{\alpha -s} \parenthese{ \frac{ N_2^{\frac{1}{2} +} N_3^{1-\alpha} }{N_1^{\alpha} N_2^{\alpha}}  \prod_{i=1}^3 \norm{Iu_i}_{X_{\delta}^{\alpha , b}} }  \parenthese{ m(N_0) B(N)  \norm{Iu}_{X_{\delta}^{\alpha , b}}^3 } \\
& \lesssim \parenthese{\frac{N_2 N_3}{N_0 N}}^{\alpha -s} \frac{ N_2^{\frac{1}{2} +} N_3^{1-\alpha} }{N_1^{\alpha} N_2^{\alpha}} B(N)  \prod_{i=1}^3 \norm{Iu_i}_{X_{\delta}^{\alpha , b}}    \norm{Iu}_{X_{\delta}^{\alpha , b}}^3 .
\end{align*}

Then by  Cauchy-Schwarz inequality with Definition \ref{defn Xsb},
\begin{align*}
\sum_{\underline{N} \in \textbf{ Case II-3}} \text{Term II} (\underline{N})  \lesssim  N^{ \frac{3}{2} -3\alpha}  B(N)  \norm{Iu}_{X_{\delta}^{\alpha , b}}^6 .
\end{align*}

Therefore, we summarize the estimation on all the cases in {\bf Term II} ,
\begin{align*}
\text{Term II} &  \lesssim \parenthese{ \sum_{\underline{N} \in \textbf{  Case II-2a}} + \sum_{\underline{N} \in \textbf{  Case II-2b}} + \sum_{\underline{N} \in \textbf{  Case II-3}}  } \text{Term II} (\underline{N}) \\
& \lesssim   N^{ 2-3\alpha}  B(N)  \norm{Iu}_{X_{\delta}^{\alpha , b}}^6   \lesssim   N^{ 2-3\alpha}  B(N)   \norm{Iu_0}_{H^{\alpha}}^6 .
\end{align*}

Let us finish the calculation on the energy increment by combining the computation for both {\bf Term I} and {\bf Term II}
\begin{align*}
& \quad E(Iu(\delta)) - E(Iu(0)) \\
& \lesssim   N^{\frac{1}{2} - \alpha+} N^{-\frac{4(\alpha -s)(b- b(\alpha-) )}{1 + 2b -4b(s)}}    \norm{Iu_0}_{H^{\alpha}}^4 +   N^{2- 3\alpha- }  \norm{Iu_0}_{H^{\alpha}}^4   \\
& \quad + N^{ 2-3\alpha}  N^{-\frac{\alpha -s}{1 + 2b -4b(s)}}      \norm{Iu_0}_{H^{\alpha}}^6 + N^{ 2-3\alpha}  N^{ \frac{3}{2} -3\alpha +}     \norm{Iu_0}_{H^{\alpha}}^6 + N^{ 2-3\alpha} N^{-\alpha + }    \norm{Iu_0}_{H^{\alpha}}^6 \\
& \lesssim   N^{\frac{1}{2} - \alpha+} N^{-\frac{4(\alpha -s)(b- b(\alpha-) )}{1 + 2b -4b(s)}}    \norm{Iu_0}_{H^{\alpha}}^4 +   N^{2- 3\alpha- }  \norm{Iu_0}_{H^{\alpha}}^4   \\
& \quad + N^{ 2-3\alpha}  N^{-\frac{\alpha -s}{1 + 2b -4b(s)}}     \norm{Iu_0}_{H^{\alpha}}^6 + N^{ \frac{7}{2} -6\alpha +}     \norm{Iu_0}_{H^{\alpha}}^6 + N^{ 2-4\alpha +}  \norm{Iu_0}_{H^{\alpha}}^6 .
\end{align*}
The last inequality follows from Remark \ref{rmk ID}. Then we finish the proof of Proposition \ref{prop energy increment}.
\end{proof}

\section{Global well-posedness}\label{sec gwp}
In this section, we  finally show the global well-posedness result stated in Theorem \ref{thm GWP} by iteration.  We also will see the choices of the parameters in previous sections and the constraint on the regularity. As a consequence of these choices, we obtain a polynomial bound of the solution as presented in Theorem \ref{thm GWP}.

\begin{proof}[Proof of Theorem \ref{thm GWP}]
By the definition of energy and the Gagliardo–Nirenberg interpolation inequality, we have 
\begin{align*}
E(Iu_0) = \frac{1}{2} \norm{Iu_0}_{\dot{H}^{\alpha}}^2 + \frac{1}{4} \norm{Iu_0}_{L^4}^4 \lesssim \norm{Iu_0}_{\dot{H}^{\alpha}}^2 + \norm{Iu_0}_{L^2}^{4-\frac{2}{\alpha}} \norm{Iu_0}_{\dot{H}^{\alpha}}^{\frac{2}{\alpha}} \lesssim \norm{Iu_0}_{\dot{H}^{\alpha}}^{\frac{2}{\alpha}} \lesssim N^{(\alpha -s) \frac{2}{\alpha}} .
\end{align*}
Then the energy increment obtained in Proposition \ref{prop energy increment} becomes
\begin{align}
E(Iu(\delta)) -  E(Iu(0)) & \lesssim   N^{\frac{1}{2} - \alpha+} N^{-\frac{4(\alpha -s)(b- b(\alpha-) )}{1 + 2b -4b(s)}}    \norm{Iu_0}_{H^{\alpha}}^4 +   N^{2- 3\alpha- }  \norm{Iu_0}_{H^{\alpha}}^4  \notag \\
& \quad + N^{ 2-3\alpha}  N^{-\frac{\alpha -s}{1 + 2b -4b(s)}}      \norm{Iu_0}_{H^{\alpha}}^6 + N^{ \frac{7}{2} -6\alpha +}     \norm{Iu_0}_{H^{\alpha}}^6 + N^{ 2-4\alpha +}  \norm{Iu_0}_{H^{\alpha}}^6 \notag\\
& \lesssim   N^{\frac{1}{2} - \alpha+} N^{-\frac{4(\alpha -s)(b- b(\alpha-) )}{1 + 2b -4b(s)}}    N^{4(\alpha -s) } +   N^{2- 3\alpha- }  N^{4(\alpha -s) }   \notag\\
& \quad + N^{ 2-3\alpha}  N^{-\frac{\alpha -s}{1 + 2b -4b(s)}}     N^{6(\alpha -s) } + N^{ \frac{7}{2} -6\alpha +}    N^{6(\alpha -s) } + N^{ 2-4\alpha +}  N^{6(\alpha -s) } . \label{eq incr} 
\end{align}

In order to reach a fixed time $T \gg 1$, the number of steps in the iteration is at most, 
\begin{align}
\frac{T}{\delta} \sim T   N^{\frac{2(\alpha -s)}{1 + 2b -4b(s)}}  \sim T N^{\frac{2(\alpha -s)}{2s-1}} . \label{eq step}
\end{align}
Combining \eqref{eq incr}  and \eqref{eq step}, we write  the modified energy at time $T$  as
\begin{align*}
E(Iu(T)) & \lesssim E(Iu(0)) + \frac{T}{\delta} \parenthese{ N^{\frac{1}{2} - \alpha+} N^{-\frac{4(\alpha -s)(b- b(\alpha-) )}{1 + 2b -4b(s)}}    N^{4(\alpha -s) } +   N^{2- 3\alpha- }  N^{4(\alpha -s) }} \\
& \quad + \frac{T}{\delta} \parenthese{ N^{ 2-3\alpha}  N^{-\frac{\alpha -s}{1 + 2b -4b(s)}}      N^{6(\alpha -s) } + N^{ \frac{7}{2} -6\alpha +}    N^{6(\alpha -s) } + N^{ 2-4\alpha +}  N^{6(\alpha -s)  } } \\
& \lesssim N^{(\alpha -s) \frac{2}{\alpha}}   + T   N^{\frac{2(\alpha -s)}{1 + 2b -4b(s)}}  \parenthese{ N^{\frac{1}{2} - \alpha+} N^{-\frac{4(\alpha -s)(b- b(\alpha-) )}{1 + 2b -4b(s)}}    N^{4(\alpha -s) } +   N^{2- 3\alpha- }  N^{4(\alpha -s) }} \\
& \quad + T   N^{\frac{2(\alpha -s)}{1 + 2b -4b(s)}} \parenthese{ N^{ 2-3\alpha}  N^{-\frac{\alpha -s}{1 + 2b -4b(s)}}      N^{6(\alpha -s) } + N^{ \frac{7}{2} -6\alpha +}    N^{6(\alpha -s) } + N^{ 2-4\alpha +}  N^{6(\alpha -s)  } } .
\end{align*}
In order to keep this iteration valid at each step , we have to ensure that the total energy increment is always being controlled by the initial energy $E(Iu(0))$, that is
\begin{align*}
& \quad T   N^{\frac{2(\alpha -s)}{1 + 2b -4b(s)}}  \parenthese{ N^{\frac{1}{2} - \alpha+} N^{-\frac{4(\alpha -s)(b- b(\alpha-) )}{1 + 2b -4b(s)}}    N^{4(\alpha -s) } +   N^{2- 3\alpha- }  N^{4(\alpha -s) }} \\
& \quad + T   N^{\frac{2(\alpha -s)}{1 + 2b -4b(s)}} \parenthese{ N^{ 2-3\alpha}  N^{-\frac{\alpha -s}{1 + 2b -4b(s)}}      N^{6(\alpha -s) } + N^{ \frac{7}{2} -6\alpha +}    N^{6(\alpha -s) } + N^{ 2-4\alpha +}  N^{6(\alpha -s)  } }  \\
&  \lesssim E(Iu(0)) \lesssim N^{(\alpha -s) \frac{2}{\alpha}} .
\end{align*}
The requirement above implies the following five inequalities
\begin{align*}
\begin{cases}
T   N^{\frac{2(\alpha -s)}{1 + 2b -4b(s)}}  N^{\frac{1}{2} - \alpha+} N^{-\frac{4(\alpha -s)(b- b(\alpha-) )}{1 + 2b -4b(s)}}    N^{4(\alpha -s) }  \lesssim  N^{(\alpha -s) \frac{2}{\alpha}} \\
T   N^{\frac{2(\alpha -s)}{1 + 2b -4b(s)}}   N^{2- 3\alpha- }  N^{4(\alpha -s) }   \lesssim  N^{(\alpha -s) \frac{2}{\alpha}} \\
T   N^{\frac{2(\alpha -s)}{1 + 2b -4b(s)}} N^{ 2-3\alpha}  N^{-\frac{\alpha -s}{1 + 2b -4b(s)}}     N^{6(\alpha -s) }   \lesssim  N^{(\alpha -s) \frac{2}{\alpha}} \\
T   N^{\frac{2(\alpha -s)}{1 + 2b -4b(s)}}   N^{ \frac{7}{2} -6\alpha +}    N^{6(\alpha -s) }  \lesssim  N^{(\alpha -s) \frac{2}{\alpha}} \\
T   N^{\frac{2(\alpha -s)}{1 + 2b -4b(s)}}  N^{ 2-4\alpha +}  N^{6(\alpha -s)  }   \lesssim  N^{(\alpha -s) \frac{2}{\alpha}} .
\end{cases}
\end{align*}
whose solutions are give by
\begin{align*}
\begin{cases}
s > s_1(\alpha) : = \frac{1}{4} \parenthese{\frac{4\alpha^2 - \alpha - 1}{2\alpha-1}  + \sqrt{ \frac{5\alpha^2 -4\alpha +1}{(2\alpha -1)^2}} }  \\
s > s_2 (\alpha) : = \frac{1}{4} \parenthese{ \frac{\alpha^2 + \alpha -1}{2\alpha -1}  + \sqrt{\frac{\alpha^4 + 10 \alpha^3 -5\alpha^2 - 2\alpha +1}{(2\alpha-1)^2}}} , \alpha > \frac{2}{3}\\
s > s_3 (\alpha) : =  \frac{1}{8} \parenthese{\frac{6\alpha^2 + 5\alpha -2}{3\alpha -1} + \sqrt{\frac{36 \alpha^4 - 36 \alpha^3 +33 \alpha^2 -20 \alpha +4}{(3\alpha -1)^2}}}  , \alpha > \frac{2}{3}\\
s > s_4 (\alpha) : =  \frac{1}{8} \parenthese{\frac{7\alpha -2}{3\alpha -1} + \sqrt{\frac{96\alpha^3 - 55\alpha^2 -4 \alpha +4}{(3\alpha -1)^2}}} , \alpha > \frac{7}{12}\\
s > s_5 (\alpha) : = \frac{2\alpha^2 + 2\alpha -1}{6\alpha -2} .
\end{cases}
\end{align*}
Let us remark here that the restriction $\alpha \in (\frac{2}{3} , 1]$ on the fractional Laplacian in this paper comes from the solutions above. If we track further back to the place where we obtained the second and the third conditions, we see that in {\bf Case I-2}, the bounds of both {\bf Case I-2a} and {\bf Case I-2b} are $N^{2-3\alpha-} \norm{Iu}_{X_{\delta}^{\alpha,b}}^4$. In the almost conservation law of energy, we anticipate a small factor $N^{2-3\alpha-}$, hence $2-3\alpha <0$ and $\alpha > \frac{2}{3}$.

We can see the global well-posedness index easily  from the picture below.
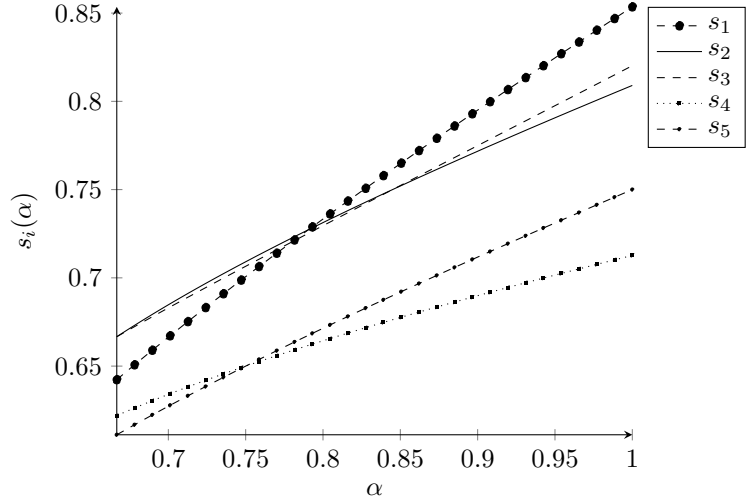
\begin{center}
\begin{tikzpicture}
\begin{axis}[ legend pos=outer north east,
    axis lines = left,
    xlabel = $\alpha$,
    ylabel = {$s_i(\alpha)$},
]

\addplot [
	dashed, mark=*, mark options={scale=0.8},
    domain=2/3:1, 
    samples=30, 
]
{ 0.25* ((-4 *x^2 + x + 1)/(1 - 2* x) + sqrt((5 *x^2 - 4 *x + 1)/(1 - 2* x)^2)) };
\addlegendentry{$s_1$}

\addplot [
	%dashed,
	%dotted, mark=*, mark options={scale=0.3},
    domain=2/3:1, 
    samples=100, 
    ]
    {(0.25 *(x^2 + x - 1))/(2 *x - 1) + 0.25 *sqrt((x^4 + 10 *x^3 - 5 *x^2 - 2 *x + 1)/(2 *x - 1)^2) };
\addlegendentry{$s_2$}

\addplot [
	%dotted, mark=*, mark options={scale=0.3},
	dashed, 
    domain=2/3:1, 
    samples=30, 
    ]
     {(0.125 *(6 * x^2 + 5 *x - 2))/(3 *x - 1) + 0.125 * sqrt((36 *x^4 - 36 *x^3 + 33 *x^2 - 20 *x + 4)/(3 *x - 1)^2) };
\addlegendentry{$s_3$}

\addplot [
	dotted, mark=square*, mark options={scale=0.3},
    domain=2/3:1, 
    samples=30, 
    ]
    {0.125 *sqrt((96 *x^3 - 55 *x^2 - 4 *x + 4)/(3 *x - 1)^2) + (0.125 * (7 *x - 2))/(3 *x - 1) };
\addlegendentry{$s_4$}

\addplot [
	dashed, mark=*, mark options={scale=0.3},
    domain=2/3:1, 
    samples=30, 
    ]
    {(2*x^2 + 2* x -1)/(6 *x -2)};
\addlegendentry{$s_5$}
\end{axis}
\end{tikzpicture}
\end{center}

Therefore,  the global well-posedness index that we obtain in this work is
\begin{align*}
s > s_* (\alpha) = \max \{ s_1(\alpha) , s_2(\alpha) \}  . 
\end{align*}
Moreover, with  the choice of $N$ solved from $s_1(\alpha) , s_2(\alpha)$, we have
\begin{align*}
T \sim \min \{ N^{(\alpha -s)(\frac{2}{\alpha} -4 - \frac{2\alpha+1}{2s-1}) + \alpha - \frac{1}{2}-} , N^{(\alpha -s) (\frac{2}{\alpha} -4) + 3\alpha -2+}  \} = : N^p.
\end{align*}
Then as a consequence of the I-method, we establish the following polynomial bound of the global solution 
\begin{align*}
\norm{u(T)}_{H^s (\Theta)} \lesssim \norm{Iu(T)}_{H^{\alpha} (\Theta)} \lesssim E(Iu(T))^{\frac{1}{2}} \sim N^{(\alpha -s)\frac{1}{\alpha}} = T^{(\alpha -s)\frac{p}{\alpha}} .
\end{align*}
Now we finish the proof of Theorem \ref{thm GWP}.
\end{proof}

\appendix
\section{Some calculations needed in the proof of Claim \ref{claim e_n}}\label{sec Appendix}

\subsection{Anti-derivative of $F(r)$}
Recall that  in Claim \ref{claim e_n}, we wrote $F$ function in the following form
\begin{align*}
F(r) & = \int_0^r \gamma \, e_{n_0}(\gamma) \, e_{n_1}(\gamma) \, d \gamma  =  \norm{J_0 (z_{n_0} \cdot)}_{L^2(\Theta)}^{-1} \norm{J_0 (z_{n_1} \cdot)}_{L^2(\Theta)}^{-1} \int_0^r \gamma \, J_0 (z_{n_0} \gamma)\, J_0 (z_{n_1} \gamma) \, d \gamma .
\end{align*}
Now we compute the anti-derivative of $F(r)$. 
\begin{lem}\label{lem intF}
Let $a \neq b$, then
\begin{align*}
\int_0^r  \gamma J_0(a \gamma) \, J_0 (b \gamma) \, d \gamma = \frac{1}{a^2 - b^2} (a r J_1 (ar)  \, J_0 (br) -  b r J_1 (br)  \, J_0 (ar)) .
\end{align*}
\end{lem}

\begin{proof}[Proof of Lemma \ref{lem intF}]
Using \eqref{eq dJ0} and \eqref{eq dJ1} with integration by parts, we write
\begin{align*}
\int_0^r a^2 \gamma J_0(a \gamma)  \, J_0 (b \gamma) \, d \gamma & = a \int_0^r J_0 (b \gamma) d ( \gamma J_1 (a \gamma))   = a r J_1 (ar)  \, J_0 (br) + ab \int_0^r J_1 (a\gamma)  \, J_1 (b \gamma) \, d \gamma \\
\int_0^r b^2 \gamma J_0(a \gamma)  \, J_0 (b \gamma) \, d \gamma & = b \int_0^r J_0 (a \gamma) d ( \gamma J_1 (b \gamma))   = b r J_1 (br)  \, J_0 (ar) + ab \int_0^r J_1 (a\gamma)  \, J_1 (b \gamma) \, d \gamma .
\end{align*}
Then taking the difference of the equations above yields the following anti-derivative
\begin{align*}
\int_0^r  \gamma J_0(a \gamma) \, J_0 (b \gamma) \, d \gamma = \frac{1}{a^2 - b^2} (a r J_1 (ar)  \, J_0 (br) -  b r J_1 (br)  \, J_0 (ar)) .
\end{align*}
\end{proof}

\subsection{Control of error terms}\label{ssec Error}

Recall that we used the following formulas in \eqref{eq J0} and \eqref{eq Jinfty} 
\begin{align*}
&\text{ when }\abs{x} < 1, & J_n(x) & = \frac{1}{n! 2^n}  x^n + \mathcal{O} (x^{n+2}) ,   \\
& \text{ when }\abs{x} \geq 1,  & J_n(x) & = \sqrt{\frac{2}{\pi}} \frac{\cos(x- \frac{n \pi}{2} - \frac{ \pi}{4})}{\sqrt{x}} + \mathcal{O} (x^{-\frac{3}{2}}) .
\end{align*}
to approximate $J_0$ and $J_1$ in the proof of Claim \ref{claim e_n}. We dealt with the contribution from the main terms above, and in this subsection we will verify that the error terms (namely, $\mathcal{O} (x^{n+2})$ when $\abs{x} <1$ and $\mathcal{O} (x^{-\frac{3}{2}})$ when $\abs{x} \geq 1$) share the same bounds.

\noindent {\bf Case I:} $0 \leq r < \frac{1}{z_{n_0}}$.

The error term in the case will be dominated by the error in estimating $J_1(z_{n_0} r)$, which is $\mathcal{O} ((z_{n_0}r)^3)$, hence it is bounded by
\begin{align*}
& \quad  \frac{ \sqrt{z_{n_0} z_{n_1} z_{n_2} z_{n_3}}  z_{n_0} z_{n_2}}{z_{n_0}^2 -z_{n_1}^2}  \int_0^{\frac{1}{z_{n_0}}}  r   (z_{n_0} r)^{3} J_0 (z_{n_1} r)   J_1 (z_{n_2} r)  J_0 (z_{n_3} r)  \, dr\\
& \sim  \frac{ \sqrt{z_{n_0} z_{n_1} z_{n_2} z_{n_3}}  z_{n_0} z_{n_2}}{z_{n_0}^2 -z_{n_1}^2}   \int_0^{\frac{1}{z_{n_0}}}  r   (z_{n_0} r)^{3}  (z_{n_2} r)   \, dr\\
& \sim  \frac{ \sqrt{z_{n_0} z_{n_1} z_{n_2} z_{n_3}}  z_{n_0} z_{n_2}}{z_{n_0}^2 -z_{n_1}^2} z_{n_0}^3 z_{n_2} \,  r^6 \Big|_0^{\frac{1}{z_{n_0}}} \sim \frac{ n_1^{\frac{1}{2}} n_2^{\frac{5}{2}} n_3^{\frac{1}{2}}}{n_0^{\frac{3}{2}}(n_0^2 -n_1^2)} .
\end{align*}

\noindent {\bf Case II:} $\frac{1}{z_{n_0}} \leq r < \frac{1}{z_{n_1}}$.

The error term from  estimating $J_1(z_{n_0} r)$ is $\mathcal{O} ((z_{n_0}r)^{-\frac{3}{2}})$, and its contribution is 
\begin{align*}
& \quad \frac{ \sqrt{z_{n_0} z_{n_1} z_{n_2} z_{n_3}}  z_{n_0} z_{n_2}}{z_{n_0}^2 -z_{n_1}^2}  \int_{\frac{1}{z_{n_0}}}^{\frac{1}{z_{n_1}}} r   (z_{n_0} r)^{-\frac{3}{2}} J_0 (z_{n_1} r)   J_1 (z_{n_2} r)  J_0 (z_{n_3} r)  \, dr\\
& \sim \frac{ \sqrt{z_{n_0} z_{n_1} z_{n_2} z_{n_3}}  z_{n_0} z_{n_2}}{z_{n_0}^2 -z_{n_1}^2}  \int_{\frac{1}{z_{n_0}}}^{\frac{1}{z_{n_1}}} r (z_{n_0} r)^{-\frac{3}{2}}  (z_{n_2} r)  \, dr\\
& \sim  \frac{ \sqrt{z_{n_0} z_{n_1} z_{n_2} z_{n_3}}  z_{n_0} z_{n_2}}{z_{n_0}^2 -z_{n_1}^2}  z_{n_0}^{-\frac{3}{2}} z_{n_2} \int_{\frac{1}{z_{n_0}}}^{\frac{1}{z_{n_1}}} r^{\frac{1}{2}}   \, dr\\
& \sim  \frac{ \sqrt{z_{n_0} z_{n_1} z_{n_2} z_{n_3}}  z_{n_0} z_{n_2}}{z_{n_0}^2 -z_{n_1}^2} z_{n_0}^{-\frac{3}{2}} z_{n_2} r^{\frac{3}{2}}  \Big|_{\frac{1}{z_{n_0}}}^{\frac{1}{z_{n_1}}}  \lesssim \frac{ n_2^{\frac{5}{2}} n_3^{\frac{1}{2}}}{n_1 (n_0^2 - n_1^2)} .
\end{align*}

The error term from  estimating $J_0(z_{n_1} r)$ is $\mathcal{O} ((z_{n_1}r)^{2})$, and its contribution is 
\begin{align*}
& \quad \frac{ \sqrt{z_{n_0} z_{n_1} z_{n_2} z_{n_3}}  z_{n_0} z_{n_2}}{z_{n_0}^2 -z_{n_1}^2}  \int_{\frac{1}{z_{n_0}}}^{\frac{1}{z_{n_1}}} r   J_1 (z_{n_0} r) (z_{n_1} r)^2   J_1 (z_{n_2} r)  J_0 (z_{n_3} r)  \, dr\\
& \sim \frac{ \sqrt{z_{n_0} z_{n_1} z_{n_2} z_{n_3}}  z_{n_0} z_{n_2}}{z_{n_0}^2 -z_{n_1}^2}  \int_{\frac{1}{z_{n_0}}}^{\frac{1}{z_{n_1}}} r  \frac{\sin (z_{n_0} r -\frac{\pi}{4})}{\sqrt{z_{n_0} r}} (z_{n_1} r)^2    (z_{n_2} r)   \, dr\\
& \sim \frac{ \sqrt{z_{n_0} z_{n_1} z_{n_2} z_{n_3}}  z_{n_0} z_{n_2}}{z_{n_0}^2 -z_{n_1}^2}  \frac{1}{\sqrt{z_{n_0}}} z_{n_1}^2 z_{n_2} \int_{\frac{1}{z_{n_0}}}^{\frac{1}{z_{n_1}}} r^{\frac{7}{2}} \sin (z_{n_0} r -\frac{\pi}{4}) \, dr \\
& \lesssim \frac{ \sqrt{z_{n_0} z_{n_1} z_{n_2} z_{n_3}}  z_{n_0} z_{n_2}}{z_{n_0}^2 -z_{n_1}^2} \frac{1}{\sqrt{z_{n_0}}}  z_{n_1}^2 z_{n_2} (\frac{z_{n_0}}{z_{n_1}})^{\frac{7}{2}} z_{n_0}^{-\frac{9}{2}} \sim \frac{n_2^{\frac{5}{2}} n_3^{\frac{1}{2}}}{n_1 (n_0^2 -n_1^2)} .
\end{align*}

\noindent {\bf Case III:} $\frac{1}{z_{n_1}} \leq r < \frac{1}{z_{n_2}}$.

The error term from estimating $J_0 (z_{n_1} r)$ is $\mathcal{O} ((z_{n_1}r)^{-\frac{3}{2}})$, and its contribution is 
\begin{align*}
& \quad \frac{ \sqrt{z_{n_0} z_{n_1} z_{n_2} z_{n_3}}  z_{n_0} z_{n_2}}{z_{n_0}^2 -z_{n_1}^2}  \int_{\frac{1}{z_{n_1}}}^{\frac{1}{z_{n_2}}} r   J_1 (z_{n_0} r) (z_{n_1}r)^{-\frac{3}{2}}  J_1 (z_{n_2} r)  J_0 (z_{n_3} r)  \, dr\\
& \sim \frac{ \sqrt{z_{n_0} z_{n_1} z_{n_2} z_{n_3}}  z_{n_0} z_{n_2}}{z_{n_0}^2 -z_{n_1}^2}   \int_{\frac{1}{z_{n_1}}}^{\frac{1}{z_{n_2}}} r \frac{\sin (z_{n_0} r -\frac{\pi}{4})}{\sqrt{z_{n_0} r}} (z_{n_1}r)^{-\frac{3}{2}}  (z_{n_2} r) \, dr \\
& \sim \frac{ \sqrt{z_{n_0} z_{n_1} z_{n_2} z_{n_3}}  z_{n_0} z_{n_2}}{z_{n_0}^2 -z_{n_1}^2}   \frac{1}{\sqrt{z_{n_0} }} z_{n_1}^{-\frac{3}{2}} z_{n_2}  \int_{\frac{1}{z_{n_1}}}^{\frac{1}{z_{n_2}}}    \sin (z_{n_0} r -\frac{\pi}{4}) \, dr \\
& \lesssim  \frac{ \sqrt{z_{n_0} z_{n_1} z_{n_2} z_{n_3}}  z_{n_0} z_{n_2}}{z_{n_0}^2 -z_{n_1}^2}   \frac{1}{\sqrt{z_{n_0} }} z_{n_1}^{-\frac{3}{2}} z_{n_2}  z_{n_0}^{-1} \sim \frac{n_2^{\frac{5}{2}} n_3^{\frac{1}{2}}}{ n_1(n_0^2 -n_1^2)}.
\end{align*}

The error term from estimating $J_1 (z_{n_2} r)$ is $\mathcal{O} ((z_{n_2} r)^3)$,  and its contribution is 
\begin{align*}
& \quad \frac{ \sqrt{z_{n_0} z_{n_1} z_{n_2} z_{n_3}}  z_{n_0} z_{n_2}}{z_{n_0}^2 -z_{n_1}^2}  \int_{\frac{1}{z_{n_1}}}^{\frac{1}{z_{n_2}}} r   J_1 (z_{n_0} r) J_0 (z_{n_1}r)   (z_{n_2} r)^3  J_0 (z_{n_3} r)  \, dr\\
& \sim \frac{ \sqrt{z_{n_0} z_{n_1} z_{n_2} z_{n_3}}  z_{n_0} z_{n_2}}{z_{n_0}^2 -z_{n_1}^2}   \int_{\frac{1}{z_{n_1}}}^{\frac{1}{z_{n_2}}} r \frac{\sin (z_{n_0} r -\frac{\pi}{4})}{\sqrt{z_{n_0} r}} \frac{\cos (z_{n_1} r -\frac{\pi}{4})}{\sqrt{z_{n_1} r}}  (z_{n_2} r)^3 \, dr \\
& \sim \frac{ \sqrt{z_{n_0} z_{n_1} z_{n_2} z_{n_3}}  z_{n_0} z_{n_2}}{z_{n_0}^2 -z_{n_1}^2}   \frac{1}{\sqrt{z_{n_0} z_{n_1}}}  z_{n_2}^3  \int_{\frac{1}{z_{n_1}}}^{\frac{1}{z_{n_2}}}  r^3  \sin (z_{n_0} r -\frac{\pi}{4}) \cos (z_{n_1} r -\frac{\pi}{4}) \, dr \\
& \lesssim  \frac{ \sqrt{z_{n_0} z_{n_1} z_{n_2} z_{n_3}}  z_{n_0} z_{n_2}}{z_{n_0}^2 -z_{n_1}^2}   \frac{1}{\sqrt{z_{n_0} z_{n_1}}}  z_{n_2}^3 \frac{1}{z_{n_0} z_{n_2}^3} \sim \frac{n_2^{\frac{3}{2}} n_3^{\frac{1}{2}}}{ n_0^2 -n_1^2},
\end{align*}
where in the last inequality, we used Lemma  \ref{lem xpsincos}  to obtain
\begin{align*}
\int_{\frac{1}{z_{n_1}}}^{\frac{1}{z_{n_2}}}  r^3  \sin (z_{n_0} r -\frac{\pi}{4}) \cos (z_{n_1} r -\frac{\pi}{4}) \, dr  \lesssim \frac{1}{z_{n_0} z_{n_2}^3}.
\end{align*}

\noindent {\bf Case IV:} $\frac{1}{z_{n_2}} \leq r < \frac{1}{z_{n_3}}$.

Recall 
\begin{align*}
\int_{\frac{1}{z_{n_2}}}^{\frac{1}{z_{n_3}}} F(r) \phi'(r) \, dr \sim \int_{\frac{1}{z_{n_2}}}^{\frac{1}{z_{n_3}}} [-\frac{1}{z_{n_0} -z_{n_1}}  \cos  ((z_{n_0}+z_{n_1})r ) +  \frac{1}{z_{n_0} + z_{n_1}}  \sin ((z_{n_0}-z_{n_1})r)   ] \phi' \, dr .
\end{align*}
The error term from estimating $J_1 (z_{n_2} r)$ is $\mathcal{O} ((z_{n_2} r)^{-\frac{3}{2}})$,  and its contribution is bounded by
\begin{align*}
& \quad \int_{\frac{1}{z_{n_2}}}^{\frac{1}{z_{n_3}}}  \frac{1}{z_{n_0} -z_{n_1}} \cos  ((z_{n_0}+z_{n_1})r )  \sqrt{z_{n_2} z_{n_3}} z_{n_2}   (z_{n_2}r)^{-\frac{3}{2}} J_0 (z_{n_3} r)  \, dr \\
& \sim \frac{ \sqrt{z_{n_2} z_{n_3}} z_{n_2}}{z_{n_0} -z_{n_1}} \int_{\frac{1}{z_{n_2}}}^{\frac{1}{z_{n_3}}} \cos  ((z_{n_0}+z_{n_1})r )  (z_{n_2}r)^{-\frac{3}{2}} \, dr \\
& \lesssim \frac{ \sqrt{z_{n_3}} }{z_{n_0} -z_{n_1}}  \int_{\frac{1}{z_{n_2}}}^{\frac{1}{z_{n_3}}} \frac{\cos((z_{n_0}+z_{n_1})r )}{ r^{\frac{3}{2}} }\, dr \lesssim \frac{ \sqrt{z_{n_3}} }{z_{n_0} -z_{n_1}} (\frac{z_{n_2}}{z_{n_0} +z_{n_1}})^{\frac{3}{2}} \sqrt{z_{n_0}} \sim \frac{n_2^{\frac{3}{2}} n_3^{\frac{1}{2}}}{(n_0 -n_1)^2} ,
\end{align*}
where in the last inequality, we used Lemma \ref{lem xpsin} and a change of variables.

The error term from estimating $J_0 (z_{n_3} r)$ is $\mathcal{O} ((z_{n_3} r)^2)$,  and its contribution is 
\begin{align*}
& \quad \int_{\frac{1}{z_{n_2}}}^{\frac{1}{z_{n_3}}}  \frac{1}{z_{n_0} -z_{n_1}} \cos  ((z_{n_0}+z_{n_1})r )  \sqrt{z_{n_2} z_{n_3}} z_{n_2}  J_1 (z_{n_2} r)  (z_{n_3} r)^2  \, dr \\
& \sim \frac{ \sqrt{z_{n_2} z_{n_3}} z_{n_2}}{z_{n_0} -z_{n_1}} \int_{\frac{1}{z_{n_2}}}^{\frac{1}{z_{n_3}}} \cos  ((z_{n_0}+z_{n_1})r ) \frac{\sin(z_{n_2}r - \frac{\pi}{4})}{\sqrt{z_{n_2} r}} (z_{n_3} r)^2  \, dr \\
& \sim  \frac{ \sqrt{z_{n_2} z_{n_3}} z_{n_2}}{z_{n_0} -z_{n_1}}  \frac{z_{n_3}^3}{\sqrt{z_{n_2}}} \int_{\frac{1}{z_{n_2}}}^{\frac{1}{z_{n_3}}} r^{\frac{3}{2}} \cos((z_{n_0}+z_{n_1})r ) \sin (z_{n_2} r - \frac{\pi}{4}) \, dr \\
& \lesssim \frac{ \sqrt{z_{n_2} z_{n_3}} z_{n_2}}{z_{n_0} -z_{n_1}}  \frac{z_{n_3}^3}{\sqrt{z_{n_2}}} \frac{1}{z_{n_0} } z_{n_3}^{-\frac{3}{2}}  \sim \frac{n_2 n_3^2}{n_0^2},
\end{align*}
where in the last inequality, we used Lemma  \ref{lem xpsincos}  to obtain
\begin{align*}
\int_{\frac{1}{z_{n_2}}}^{\frac{1}{z_{n_3}}} r^{\frac{3}{2}} \cos((z_{n_0}+z_{n_1})r ) \sin (z_{n_2} r - \frac{\pi}{4}) \, dr   \lesssim  \frac{1}{z_{n_0} } z_{n_3}^{-\frac{3}{2}} .
\end{align*}

\noindent {\bf Case V:} $\frac{1}{z_{n_3}} \leq r \leq 1$.

The error term in the case will be dominated by the error in estimating $J_0(z_{n_3} r)$, which is $\mathcal{O} (z_{n_3}r)^{-\frac{3}{2}}$, hence it is bounded by
\begin{align*}
& \quad \int_{\frac{1}{z_{n_3}}}^1 \frac{1}{z_{n_0} -z_{n_1}} \cos  ((z_{n_0}+z_{n_1})r )  \sqrt{z_{n_2} z_{n_3}} z_{n_2}  J_1 (z_{n_2} r)  (z_{n_3} r)^{-\frac{3}{2}} \, dr \\
& \sim \frac{ \sqrt{z_{n_2} z_{n_3}} z_{n_2}}{z_{n_0} -z_{n_1}}  \int_{\frac{1}{z_{n_3}}}^1 \cos  ((z_{n_0}+z_{n_1})r ) \frac{\sin (z_{n_2} r - \frac{\pi}{4})}{\sqrt{z_{n_2} r}}  (z_{n_3} r)^{-\frac{3}{2}} \, dr \\
& \sim \frac{ z_{n_2}}{(z_{n_0} -z_{n_1})z_{n_3}} \int_{\frac{1}{z_{n_3}}}^1   \frac{1}{r^2}  \cos  ((z_{n_0}+z_{n_1})r ) \sin (z_{n_2} r - \frac{\pi}{4}) \, dr \\
& \lesssim \frac{ z_{n_2}}{(z_{n_0}-z_{n_1}) z_{n_3}} \frac{z_{n_3}^2}{z_{n_0}} \sim \frac{n_2 n_3}{n_0^2} ,
\end{align*}
where in the last inequality, we used Lemma \ref{lem xpsincos} to obtain
\begin{align*}
\int_{\frac{1}{z_{n_3}}}^1   \frac{1}{r^2} \cos  ((z_{n_0}+z_{n_1})r ) \sin (z_{n_2} r - \frac{\pi}{4})  \, dr \lesssim \frac{z_{n_3}^2}{z_{n_0}} .
\end{align*}

Now we conclude that all the error terms in estimating \eqref{eq ex} are controlled by the same bound in {\it (4)} in Claim \ref{claim e_n}, which is $\mathcal{O} \parenthese{\frac{n_2^{\frac{3}{2}} n_3^{\frac{1}{2}}}{n_0^2}} $.

\subsection{Estimates on some trigonometric integrals}
\begin{lem}\label{lem xpsin}
For $a,b >0$ and $p \neq 0$, we have 
\begin{align*}
\abs{\int_a^b x^p \sin x \, dx}  + \abs{\int_a^b x^p \cos x \, dx}  \lesssim  a^p + b^p .
\end{align*}
\end{lem}
\begin{proof}[Proof of Lemma \ref{lem xpsin}]
We integrate by parts and write
\begin{align*}
\int_a^b x^p \sin x  \, dx & = -x^p\cos x  \Big|_{a}^b +  p \int_a^b x^{p-1}\cos x  \, dx  , \\
\abs{\int_a^b x^p \sin x \, dx} & \lesssim \abs{b^p \cos b- a^p \cos a} + \abs{ \int_a^b x^{p-1} \, dx } \leq  a^p + b^p .
\end{align*}
The cosine  integral follows similarly.
\end{proof}

\begin{lem}\label{lem xpsincos}
For $a,b >0$, $m \leq n$ and $p \neq 0$, we have
\begin{align*}
& \abs{\int_a^b x^{p} \sin (mx) \cos (nx) \, dx}  + \abs{\int_a^b x^{p} \cos (mx) \sin (nx) \, dx} + \abs{\int_a^b x^{p} \sin (mx) \sin (nx) \, dx} + \abs{\int_a^b x^{p} \cos (mx) \cos (nx) \, dx} \\
& \lesssim \frac{n}{n^2 -m^2} (a^{p} + b^{p})  
\end{align*}
\end{lem}

\begin{proof}[Proof of Lemma \ref{lem xpsincos}]
We only present the calculation for the bound of the first term in this lemma. Similar calculation works for other terms. 

Performing  integration by parts, we write
\begin{align*}
& \quad \int_a^b x^{p} \sin (mx) \cos (nx) \, dx  = \frac{1}{n} \int_a^b x^{p} \sin (mx) \, d \sin (nx) \\
& = \frac{1}{n} \parenthese{ x^{p} \sin (mx) \sin (nx) \Big|_a^b  -p \int_a^b x^{p-1} \sin (mx) \sin (nx) \, dx } - \frac{m}{n} \int_a^b  x^{p} \cos(mx) \sin (nx) \, dx .
\end{align*}
Take the second term above and integrate by parts again, then we have 
\begin{align*}
& \quad \int_a^b  x^{p} \cos(mx) \sin (nx) \, dx = - \frac{1}{n} \int_a^b x^{p} \cos (cx) \, d \cos (nx) \\
& = -\frac{1}{n} \parenthese{  x^{p} \cos (mx) \cos (nx) \Big|_a^b  -p \int_a^b x^{p-1} \cos (mx) \cos (nx) \, dx} -\frac{m}{n} \int_a^b x^{p} \sin (mx) \cos (nx) \, dx  .
\end{align*}
Now combining the copulation above, we arrive at
\begin{align*}
A & : = \int_a^b x^{p} \sin (mx) \cos (nx) \, dx  \\
& = \frac{1}{n} \parenthese{ x^{p} \sin (mx) \sin (nx) \Big|_a^b  -p \int_a^b x^{p-1} \sin (mx) \sin (nx) \, dx } - \frac{m}{n} \int_a^b  x^{p} \cos(mx) \sin (nx) \, dx \\
& = \frac{1}{n} \parenthese{ x^{p} \sin (mx) \sin (nx) \Big|_a^b  -p \int_a^b x^{p-1} \sin (mx) \sin (nx) \, dx } \\
& \quad   + \frac{m}{n^2} \parenthese{  x^{p} \cos (mx) \cos (nx) \Big|_a^b -p \int_a^b x^{p-1} \cos (mx) \cos (nx) \, dx} + \frac{m^2}{n^2}A .
\end{align*}
Therefore, solving  for $A$ gives us 
\begin{align*}
(1-\frac{m^2}{n^2}) A & = \frac{1}{n} \parenthese{ x^{p} \sin (mx) \sin (nx) \Big|_a^b  - p \int_a^b x^{p-1} \sin (mx) \sin (nx) \, dx } \\
& \quad   + \frac{m}{n^2} \parenthese{  x^{p} \cos (mx) \cos (nx) \Big|_a^b -p \int_a^b x^{p-1} \cos (mx) \cos (nx) \, dx}  ,
\end{align*}
then the triangle inequality and boundedness of trig functions yield the desired bound
\begin{align*}
\abs{A} & \lesssim \frac{n^2}{n^2 -m^2} \frac{1}{n} \abs{ x^{p} \sin (mx) \sin (nx) \Big|_a^b  -p \int_a^b x^{p-1} \sin (mx) \sin (nx) \, dx }\\
\abs{A} & \lesssim \frac{n}{n^2 -m^2} (a^{p} + b^{p})  .
\end{align*}

\end{proof}

\bibliography{FNLS}
\bibliographystyle{plain}
\end{document}